\documentclass[11pt]{amsart}

\ExplSyntaxOn
\msg_redirect_name:nnn { kernel } { variant-same-as-base } { info }
\ExplSyntaxOff

\usepackage{amsmath,amssymb,amsthm,url,scalerel}
\usepackage[utf8]{inputenc}
\usepackage[T1]{fontenc}
\usepackage{libertine}
\usepackage[libertine,cmintegrals,cmbraces]{newtxmath}
\usepackage{verbatim}
\usepackage[shortlabels]{enumitem}
\usepackage{stmaryrd}
\usepackage{mathtools}
\usepackage{microtype}
\usepackage{tabto}
\usepackage{xcolor}
\usepackage{tikz}
\usepackage{tikz-cd}
\tikzcdset{arrow style=tikz,
           diagrams={>=Straight Barb}
           }
\usepackage{nicefrac}
\usepackage{tabularx}
\usepackage{dsfont}
\usepackage{bbold}
\usepackage{todonotes}
\usepackage{quiver}
\usepackage{a4wide}
\usepackage{euscript}
\usepackage[all]{xy}
\usepackage{mathrsfs,comment}
\numberwithin{equation}{subsection}
\usepackage{adjustbox}
\usepackage{multicol}
\usepackage{listings}
\usepackage[table]{xcolor}
\usepackage[dvipsnames]{xcolor}
\usepackage{BOONDOX-calo}
\usepackage{xfrac}
\usepackage{booktabs}
\usepackage{float}
\usepackage{epigraph}

\usepackage[pagebackref]{hyperref}
  
\hypersetup{%
  bookmarksnumbered=true,%%
  colorlinks=true,%
  linkcolor=black,%
  citecolor=black,%
  filecolor=blue,%
  menucolor=black,%
  urlcolor=blue,%
  pdfnewwindow=true,%
  pdfstartview=FitBH}
\usepackage[capitalise]{cleveref}

\usepackage{etoolbox}

\makeatletter
\apptocmd{\thebibliography}{%
  \item[]\hspace*{-\leftmargin}%
  \begin{minipage}{\dimexpr\linewidth+\leftmargin\relax}
  \small
  \textbf{Note on publication data.}
  In the spirit of the \textit{Cost of Knowledge} campaign, journal information is included in this bibliography only as metadata identifying published versions of the cited works. No credit is intended to accrue to the journals themselves: the research cited here is due to the authors, and the labour of peer review belongs to the research community.
  \end{minipage}
  \vspace{1em}
}{}{}
\makeatother

\newcommand{\fib}{\mathsf{fib}}

\DeclareMathOperator{\id}{\mathrm{Id}}

\newtheorem{thm}{Theorem}[subsection]
\newtheorem*{thmA}{Theorem A}
\newtheorem*{thmB}{Theorem B}
\newtheorem*{thmC}{Theorem C}
\newtheorem{prop}[thm]{Proposition}
\newtheorem{lem}[thm]{Lemma}
\newtheorem{cor}[thm]{Corollary}
\newtheorem{conj}[thm]{Conjecture}

\newtheorem*{thm*}{Theorem}

\theoremstyle{definition}
\newtheorem{definition}[thm]{Definition}
\newtheorem{ex}[thm]{Example}

\newtheorem{rem}[thm]{Remark}
\newtheorem{ques}[thm]{Question}

\title{An abelian model for the Goodwillie tower of the circle}
\author{Marco Nervo}
\date{}

\begin{document}

\begin{abstract}
We study the Goodwillie tower of the identity evaluated at the circle. Our main result is the \(p\)-local equivalence
\[
\Omega P_{p^k}I(S^1)
\simeq
\mathbb{Z} \times \Omega^{\infty+2k}\tau_{>0}\mathbb{Sp}^{p^k}
\]
obtained from the work of Behrens and Kuhn on the Whitehead conjecture. In particular, this identifies the Goodwillie approximations of the circle as infinite loop spaces, a phenomenon which does not hold in general. The equivalence also reduces the computation of their homotopy groups to a stable problem, making them accessible to tools such as the Adams spectral sequence and to computer calculations. We place these computations in a more general framework by developing a Goodwillie-calculus analogue of the Gray sequence, which relates the homotopy groups of successive stages of the tower. 
\end{abstract}

\maketitle
\tableofcontents

\section*{Introduction}

\begin{flushright}
\begin{minipage}{0.65\textwidth}
\small\itshape
Let's say the word for circle is: O. This language has a melody system to illustrate comparatives. We'll represent this by the diacritical marks: \v{ }, \={ }, and \^{ }, respectively, mean smallest, ordinary, and biggest. So what would \v{O} mean?

\medskip
\hfill --- Samuel R. Delany, \textit{Babel-17}
\end{minipage}
\end{flushright}

\medskip

As stereotypical unstable homotopy theorists, our ultimate goal is to compute the homotopy groups of spheres. A first simplification of this problem is provided by the \emph{EHP sequence}, which reduces computations for even spheres to computations for odd spheres. A second simplification is to work one prime at a time, and from now on everything will be implicitly localized at a fixed prime $p$. Once these reductions are made, one can follow two different approaches: a classical one, and a more modern one, which we advocate in this paper.

The classical approach begins with the loop-suspension filtration

% https://q.uiver.app/#q=WzAsNyxbMCwwLCJcXE9tZWdhIFNeMSJdLFsxLDAsIlxcY2RvdHMiXSxbMiwwLCJcXE9tZWdhXnsybi0xfVNeezJuLTF9Il0sWzMsMCwiXFxPbWVnYV57Mm4rMX1TXnsybisxfSJdLFszLDEsIlxcT21lZ2FeezJuLTJ9V157Mm4tMX0iXSxbNCwwLCJcXGNkb3RzIl0sWzUsMCwiXFxPbWVnYV57Mk4rMX1TXnsyTisxfSJdLFswLDFdLFsxLDJdLFsyLDNdLFszLDRdLFszLDVdLFs1LDZdXQ==
\[\begin{tikzcd}
	{\Omega S^1} & \cdots & {\Omega^{2n-1}S^{2n-1}} & {\Omega^{2n+1}S^{2n+1}} & \cdots & {\Omega^{2N+1}S^{2N+1}} \\
	&&& {\Omega^{2n-2}W^{2n-1}}
	\arrow[from=1-1, to=1-2]
	\arrow[from=1-2, to=1-3]
	\arrow[from=1-3, to=1-4]
	\arrow[from=1-4, to=1-5]
	\arrow[from=1-4, to=2-4]
	\arrow[from=1-5, to=1-6]
\end{tikzcd}\]
where $W^{2n-1}$ is the fiber of the double suspension map $S^{2n-1}\to \Omega^2S^{2n+1}$. Gray \cite{Gray1989} proved that this fiber fits into a fiber sequence
$$
\Omega^2S^{2np-1} \longrightarrow W^{2n-1} \longrightarrow \Omega^3 S^{2np+1}
$$
and from this one obtains a spectral sequence
\[\begin{tikzcd}
	\pi_*\begin{array}{c} 
        \begin{Bmatrix}
         \Omega S^1 \\
         \Omega^{2n}S^{2pn-1} \\
         \Omega^{2n+1}S^{2pn+1}
     \end{Bmatrix}_{n \leq N} 
     \end{array} 
     & 
     \pi_*\begin{array}{c} 
        \begin{Bmatrix}
         \Omega S^1 \\
         \Omega^{2n-2}W^{2n-1}
     \end{Bmatrix}_{n \leq N} 
     \end{array}
     &\pi_*{\Omega^{2N+1}S^{2N+1}}
	\arrow[Rightarrow, from=1-1, to=1-2]
    \arrow[Rightarrow, from=1-2, to=1-3]
\end{tikzcd}\]
The attractive feature of this approach is that the initial term is simple, since $\Omega S^1 \simeq \mathbb{Z}$. Its drawback is that Gray's sequence relates $S^{2n+1}$ to $S^{2np+1}$, so computations do not proceed by direct induction on the dimension. Nevertheless, this framework already underlies classical work of Toda \cite{todaoriginal} and others.

The modern approach is to replace spheres by more accessible approximations. Let $P_mI$ denote the $m$-th polynomial approximation of the identity functor in the sense of Goodwillie calculus \cite{GoodwillieIII}, and define the \emph{$k$-th approximation $n$-sphere} by
$$
S^n_k := P_{p^k}I(S^n)
$$
There is then a natural comparison map
$$
S^n \longrightarrow S^n_k
$$
Our first main result is a generalization of the Freudenthal suspension theorem (thm. \ref{genfreudenthal}).

\begin{thmA}[Generalized Freudenthal theorem]
The map $S^n \to S^n_k$ is an equivalence on $\pi_*$ for
\[
* <
\begin{cases}
(n+1)p^{k+1} - 2k - 3 & p=2 \text{ or $n$ odd} \\
2np^k - 2k - 1 & p>2 \text{ and $n$ even}
\end{cases}
\]
\end{thmA}
Thus, $\pi_* S^n$ can be recovered in any fixed range from $\pi_* S^n_k$ for $k$ large enough.

These approximation spheres retain the same formal structure as ordinary spheres. In particular, one can again form the loop-suspension filtration

% https://q.uiver.app/#q=WzAsNyxbMCwwLCJcXE9tZWdhIFNeMV9rIl0sWzEsMCwiXFxjZG90cyJdLFsyLDAsIlxcT21lZ2FeezJuLTF9U157Mm4tMX1fayJdLFszLDAsIlxcT21lZ2FeezJuKzF9U157Mm4rMX1fayJdLFszLDEsIlxcT21lZ2FeezJuLTJ9V157Mm4tMX1fayJdLFs0LDAsIlxcY2RvdHMiXSxbNSwwLCJcXE9tZWdhXnsyTisxfVNeezJOKzF9X2siXSxbMCwxXSxbMSwyXSxbMiwzXSxbMyw0XSxbMyw1XSxbNSw2XV0=
\[\begin{tikzcd}
	{\Omega S^1_k} & \cdots & {\Omega^{2n-1}S^{2n-1}_k} & {\Omega^{2n+1}S^{2n+1}_k} & \cdots & {\Omega^{2N+1}S^{2N+1}_k} \\
	&&& {\Omega^{2n-2}W^{2n-1}_k}
	\arrow[from=1-1, to=1-2]
	\arrow[from=1-2, to=1-3]
	\arrow[from=1-3, to=1-4]
	\arrow[from=1-4, to=1-5]
	\arrow[from=1-4, to=2-4]
	\arrow[from=1-5, to=1-6]
\end{tikzcd}\]
where $W^{2n-1}_k$ denotes the fiber of the double suspension $S^{2n-1}_k \to \Omega^2S^{2n+1}_k$. Our next result shows that these fibers fit into a calculus analogue of Gray's sequence (thm. \ref{graysequence}).

\begin{thmB}[Calculus version of Gray's sequence]
There is a fiber sequence
\[
\Omega^2S^{2np-1}_{k-1} \longrightarrow W^{2n-1}_k \longrightarrow \Omega^3S^{2np+1}_{k-1}
\]
\end{thmB}

This leads to a new spectral sequence
\[\begin{tikzcd}
	\pi_*\begin{array}{c} 
        \begin{Bmatrix}
         \Omega S^1_k \\
         \Omega^{2n}S^{2pn-1}_{k-1} \\
         \Omega^{2n+1}S^{2pn+1}_{k-1}
     \end{Bmatrix}_{n \leq N} 
     \end{array} 
     & 
     \pi_*\begin{array}{c} 
        \begin{Bmatrix}
         \Omega S^1_k \\
         \Omega^{2n-2}W^{2n-1}_k
     \end{Bmatrix}_{n \leq N} 
     \end{array}
     &\pi_*{\Omega^{2N+1}S^{2N+1}_k}
	\arrow[Rightarrow, from=1-1, to=1-2]
    \arrow[Rightarrow, from=1-2, to=1-3]
\end{tikzcd}\]
At this point the picture changes substantially. As in the classical setting, Gray's sequence relates spheres in larger and larger dimensions, so induction on the dimension is not available. In the calculus setting, however, the index $k$ decreases, and one can therefore induct on the degree of approximation. The remaining problem is to understand the initial term $\Omega S^1_k$, which is no longer simple. Our main theorem resolves exactly this point by giving an explicit description of it (thm. \ref{model}).

\begin{thmC}[Model for the approximation circles]
There is an equivalence
\[
\Omega S^1_k \simeq \mathbb{Z} \times \Omega^{\infty+2k}\tau_{>0}\mathbb{Sp}^{p^k}
\]
\end{thmC}

Here $\mathbb{Sp}^m$ denotes the $m$-th symmetric product spectrum (def. \ref{symmprod}). This identifies the initial term of the inductive procedure with a classical family of spectra whose properties are already well understood (props. \ref{symmprodfp} and \ref{symmprodkn}). Combined with theorem A and theorem B, this yields an inductive framework for studying the homotopy groups of spheres.

We use these results mainly for explicit calculations. For example, when $n$ and $p$ are odd, one obtains (ex. \ref{todasn})
\begin{align*}
\pi_{n+*}S^n \cong 
\begin{cases}
\mathbb{Z} & * = 0\\
\pi_*\Sigma^\infty B\Sigma_p^{(p-1)(n-1)} \ /\  \mathbb{F}_p & * = 2(p^2-2)\\
\pi_*\Sigma^\infty B\Sigma_p^{(p-1)(n-1)} & * < 4(p^2-2), \text{otherwise}
\end{cases}
\end{align*}

where $B\Sigma_p^m$ denotes the $m$-skeleton of $B\Sigma_p$. To appreciate the computational effectiveness of the method, the reader may try to recover the first $27$ stems of $S^3$ at $p=3$ starting only from this statement and the homotopy groups of the sphere spectrum $\mathbb{S}$ (ex. \ref{metastableS3}).

Beyond computations, the calculus version of Gray's sequence also allows us to revisit the work of Arone--Mahowald \cite{AroneMahowald} and explore its relation to the classical Cohen--Moore--Neisendorfer theorem \cite{CMN}.

The paper is organized as follows.

In the first section we construct the approximation spheres and, using the calculations of Arone--Mahowald \cite{AroneMahowald}, prove the generalized Freudenthal theorem. We then establish the calculus version of Gray's sequence and discuss several of its applications.

The second section reviews the main properties of the symmetric product spectra and proves theorem C, relying on the work of Behrens \cite{Behrens2011} and Kuhn \cite{Kuhn2014} on the Whitehead conjecture.

In the third section we apply the previous results to explicit computations: we revisit some of the calculations of Behrens \cite{Behrens2010} and Mahowald \cite{Metastable}, and we prove a general statement about the homotopy groups of $S^3$.

The final section collects ideas for further applications and possible directions for future research.

Finally, the paper contains two appendices. Appendix A collects some technical material on the Steenrod algebra used in the section on the approximation circles, while Appendix B provides a short script for generating Adams spectral sequence charts.

\subsection*{Acknowledgments}

I would like to thank my supervisor Gijs Heuts for his support, trust, and for giving me the freedom to pursue my own ideas and interests. I am also deeply grateful to Gregory Arone, both mathematically and personally, for introducing me to calculus and symmetric products, and for everything he taught me.

This work builds heavily on ideas of Mark Behrens and Nick Kuhn. I also thank Niall Taggart for reading parts of the manuscript at an earlier stage and for his useful feedback, as well as my colleagues at Utrecht University and in the broader homotopy theory community around Utrecht, Amsterdam, and Nijmegen.

This research was supported by the European Research Council through the grant \textit{Chromatic homotopy theory of spaces} (grant no.~950048). It also benefited from Hood Chatham's Steenrod algebra and Adams spectral sequence calculators, from Nathanael Arkor's \textit{Quiver}, and from tools developed by OpenAI, which helped improve the exposition of the paper.

\section{Approximation spheres}

The goal of this section is to study the \emph{approximation spheres}. These are the Goodwillie approximations of the identity functor evaluated on spheres. We estimate the range of convergence of the comparison maps, thereby generalizing the classical Freudenthal suspension theorem. We then prove a version of the Gray sequence for the approximation spheres, which allows us to express each approximation in terms of earlier ones. Finally, we discuss the phenomenon of early convergence.

In the first subsection, we review Goodwillie calculus.

In the second subsection, we construct the approximation spheres, review the Arone--Dwyer models \cite{AroneDwyer} for their associated graded, and recall the Arone--Mahowald computations \cite{AroneMahowald} of their $\mathbb{F}_p$-cohomology. Using these results, we determine their connectivity and prove the generalized Freudenthal theorem. We also recall the classical EHP sequence and Behrens' calculus version of it \cite{Behrens2010}.

In the third subsection, we introduce the fiber of the double suspension map and its approximations. We review the Gray sequence \cite{Gray1989} and explain how, in the case $p=2$, it can be deduced from the EHP sequence. We then prove a calculus version of the Gray sequence, using a functorial construction due to Richter \cite{Richter2014}. We conclude by explaining how this calculus version could be used to recover the results of Arone--Mahowald and may provide an alternative approach to the Cohen--Moore--Neisendorfer theorem \cite{CMN}.

In the final subsection, we study the phenomenon of early convergence of the approximations, describe its connections with classical problems in homotopy theory, and present some illustrative examples.

\subsection{Goodwillie calculus}

Goodwillie's calculus of functors \cite{GoodwillieIII} provides a systematic way to approximate functors by polynomial ones, in analogy with Taylor series in ordinary calculus. In this paper, the relevant setting is that of functors defined on the category of pointed spaces $\mathcal{S}_*$ or on the category of spectra $\mathcal{Sp}$. To such a functor $F$, Goodwillie calculus associates a tower
$$
F \longrightarrow \cdots \longrightarrow P_mF \longrightarrow P_{m-1}F \longrightarrow \cdots \longrightarrow P_0F
$$
where each $P_mF$ is the universal approximation to $F$ by an $m$-excisive functor. Recall that a functor is \textit{$m$-excisive} if it sends strongly cocartesian $(m+1)$-cubes to cartesian cubes. Informally, this means that it behaves like a polynomial of degree at most $m$.

The fiber
$$
D_m F := \mathrm{fib}(P_mF \longrightarrow P_{m-1}F)
$$
is called the \textit{$m^\mathrm{th}$ homogeneous layer} of $F$ and plays the role of the $m^\mathrm{th}$ homogeneous term in the Taylor expansion. Homogeneous functors of degree $m$ admit an infinite delooping $\mathbb{D}_m F$, and they are classified by spectra $\partial_m F$ with a $\Sigma_m$-action, namely
$$
D_m F(X) \simeq \Omega^\infty \mathbb{D}_m F(X) \simeq \Omega^\infty \big( \partial_m F \otimes X^{\otimes m}\big)_{h\Sigma_m}
$$
The spectrum $\partial_m F$ is called the \textit{$m^\mathrm{th}$ derivative} of $F$, and this description allows one to import tools from stable homotopy theory into the study of unstable functors.

A natural question is whether the tower $\{P_mF\}$ converges, meaning that the canonical map
$$
F(X) \longrightarrow \mathrm{lim}_m P_m F(X)
$$
is an equivalence. This depends in general on both the functor and the input. A functor is said to be \textit{$\rho$-analytic} if, roughly, its tower converges on $\rho$-connective spaces, and it is said to be \textit{analytic} if it is $\rho$-analytic for some $\rho$.

We will need the following two properties of analytic functors. Although analyticity is probably not essential, we could not find a reference in the literature that avoids assuming it.

\begin{lem}[{\cite[prop. 4.6]{AroneMahowald}, \cite[proof of lem. 2.1.2]{Behrens2010}}]
	\label{calcspheres}
	If $F \to G$ is a map between reduced analytic functors that is an equivalence on (all, even or odd) spheres, then $P_mF \to P_mG$ and $D_mF \to D_mG$ are equivalences on (all, even or odd) spheres. 
\end{lem}

We will use this lemma as a convenient way to keep the discussion entirely within Goodwillie calculus. Some of the arguments below could probably be formulated more naturally in the language of orthogonal or unitary calculus, which might provide a cleaner conceptual framework. However, developing that language would take us too far afield, and this lemma is sufficient for our purposes.

\begin{lem}[{\cite[lem. 1.4]{AK}, \cite[lem. 2.1.3]{Behrens2010}}]
	If $F$ is an analytic functor, then
	\label{calcpowers}
	\begin{align*}
		P_m(F(-)^{\otimes k}) \simeq P_{\lfloor m/k \rfloor}F \circ (-)^{\otimes k} \qquad D_m (F(-)^{\otimes k}) \simeq \begin{cases}
		D_{m/k}F \circ (-)^{\otimes k} &\text{ if } k \mid m \\ * &\text{ if } k \nmid m
	\end{cases}
	\end{align*}
\end{lem}

In light of the last lemma, we will adopt the convention $P_m := P_{\lfloor m \rfloor}$ and $D_m := *$ if $m$ is not an integer. This is consistent with the intuition that $P_m$ should represent a polynomial of degree $\leq m$, while $D_m$ should be homogeneous of degree exactly $m$.

\subsection{Generalized Freudenthal theorem}

In this subsection, we use Goodwillie calculus to construct approximations of spheres and determine the range in which these approximations converge. More precisely, we rely on descriptions of the layers of the Goodwillie tower of the identity functor on odd spheres, which allow us to compute their connectivity. We also explain how even spheres can be treated via the calculus version of the EHP sequence. Combining these tools, we obtain a generalized form of the Freudenthal suspension theorem, extending the classical stable and metastable ranges.

We begin by specializing the general discussion of Goodwillie calculus to the identity functor $I\colon \mathcal{S}_* \to \mathcal{S}_*$. Its Taylor tower converges on every connected $\mathbb{Z}$-complete space \cite{convergence}, and in particular on every sphere $S^n$ with $n\geq 1$. The case of $S^0$ is not covered by this general convergence theorem; we will return to it in remark \ref{convS0}, where we prove a $2$-local convergence statement.

The homogeneous layers $\mathbb{D}_mI(S^n)$ have been studied extensively. In particular, Arone--Mahowald \cite[thm. 0.1]{AroneMahowald} proved that, after $p$-localization and for $n$ odd, these layers vanish unless $m=p^k$. This motivates the following notation:
\[
	S^n_k := P_{p^k}I(S^n)
	\qquad
	S^n(k) := \mathbb{D}_{p^k}I(S^n)
\]
We will call $S^n_k$ the \emph{$k$-approximation $n$-sphere}. The zeroth approximation is $\Omega^\infty\mathbb{S}^n$, corresponding to stable homotopy.

\begin{ex}
	We have $S^n(0) \simeq \mathbb{S}^n$ and $S^n_0 \simeq \Omega^\infty\mathbb{S}^n$. The first equivalence follows from
	\[
	S^n(0) \simeq \mathrm{colim}_j \Sigma^{-j}\Sigma^\infty I(\Sigma^j S^n) \simeq \mathrm{colim}_j \mathbb{S}^n \simeq \mathbb{S}^n
	\]
	The second equivalence then follows from the fact that this is the first nontrivial layer.
\end{ex}

\begin{rem}
	When $n$ is even, there are additional nontrivial layers in degrees $m = 2p^k$, but apart from this subsection we will focus mainly on the case of odd spheres.
\end{rem}

For $n$ odd and $k>0$, the spectra $S^n(k)$ admit several equivalent descriptions, due to Johnson \cite{derivatives} and Arone--Dwyer \cite{AroneDwyer}. These descriptions will not play a major role later, but it is still useful to collect them here, both for convenience and because they form a beautiful piece of mathematics. To state them, we recall some auxiliary constructions.

The \textit{partition complex} $P_m$ is the geometric realization of the poset of proper, nontrivial partitions of the set $\{1,\dots,m\}$, ordered by refinement. The symmetric group $\Sigma_m$ acts on $P_m$ by permuting the blocks. Non-equivariantly, it is a wedge of $(m-1)!$ spheres of dimension $m - 3$.

The \textit{Tits building} $T_k$ is the geometric realization of the poset of proper, nontrivial subspaces of $\mathbb{F}_p^k$, ordered by inclusion. It carries an action of the affine group $\mathrm{Aff}_k\mathbb{F}_p := \mathbb{F}_p^k \rtimes \mathrm{GL}_k\mathbb{F}_p$, where the subgroup $\mathbb{F}_p^k$ acts trivially. Non-equivariantly, it is a wedge of $p^{\binom{k}{2}}$ spheres of dimension $k - 2$.

The \textit{Steinberg idempotent} $\epsilon^\mathrm{st}_k$ is a distinguished idempotent in the group ring $\mathbb{F}_p[\mathrm{GL}_k\mathbb{F}_p]$. It can be lifted to an element of $\mathbb{Z}[\mathrm{GL}_k\mathbb{F}_p]$ and hence made to act on any spectrum with a $\mathrm{GL}_k\mathbb{F}_p$-action. Acting on such a spectrum $X$, it splits off a direct summand $\epsilon^\mathrm{st}_k \cdot X$. Arone--Dwyer show \cite[cor. 9.6]{AroneDwyer} that this summand is related to the Tits building by
\[
\epsilon^\mathrm{st}_k \cdot X \simeq \Sigma^{1-k}(T_k^\diamond \otimes X)_{h\mathrm{GL}_k\mathbb{F}_p}
\]
If $m=p^k$, both $\Sigma_m$ and $\mathrm{Aff}_k\mathbb{F}_p$ act on the $m$-element set $\mathbb{F}_p^k$: the former by permutations and the latter through the standard affine action. We denote either action by $\sigma$. In this situation, there are natural maps $\mathrm{Aff}_k\mathbb{F}_p \to \Sigma_m$ and $T_k \to P_m$.

We write $(-)^\diamond$ for unreduced suspension and $(-)^\vee$ for Spanier--Whitehead duality.

\begin{prop}
For $n$ odd and $k > 0$, the spectra $S^n(k)$ admit the following equivalent descriptions:
    \begin{enumerate}
        \item $((\Sigma P_{p^k}^\diamond)^\vee \otimes \mathbb{S}^{n\sigma})_{h\Sigma_{p^k}}$
        \item $((\Sigma T_k^\diamond)^\vee \otimes \mathbb{S}^{n\sigma})_{h\mathrm{Aff}_k\mathbb{F}_p}$
        \item $\Sigma^{1-2k}(P_{p^k}^\diamond \otimes \mathbb{S}^{n\sigma})_{h\Sigma_{p^k}}$ 
        \item $\Sigma^{1-2k}(T_k^\diamond \otimes \mathbb{S}^{n\sigma})_{h\mathrm{Aff}_k\mathbb{F}_p}$ 
        \item $\Sigma^{1-2k}(T_k^\diamond \otimes \mathbb{S}^{n\sigma}_{h\mathbb{F}_p^k})_{h\mathrm{GL}_k\mathbb{F}_p}^{}$ 
        \item $\Sigma^{-k}\epsilon^\mathrm{st}_k \cdot \mathbb{S}^{n\sigma}_{h\mathbb{F}_p^k}$
    \end{enumerate}
\end{prop}

\begin{proof}
    The first description follows from the classification of homogeneous functors and Johnson's computation $\partial_mI \simeq (\Sigma P_m^\diamond)^\vee$ \cite{derivatives}, which holds in general. The remaining equivalences are established in Arone--Dwyer \cite{AroneDwyer}. The map $T_k \to P_{p^k}$ induces an equivalence after smashing with tensor powers of odd spheres and taking homotopy orbits, both directly and in dual form \cite[thm. 1.8--1.9]{AroneDwyer}, giving (1) $\simeq$ (2) and (3) $\simeq$ (4). The equivalence (2) $\simeq$ (4) follows from the self-duality of Tits buildings \cite[thm. 8.2]{AroneDwyer}, namely $(T_k^\diamond)^\vee \simeq \Sigma^{2(1-k)}T_k^\diamond$. The equivalence (4) $\simeq$ (5) uses the fact that $\mathbb{F}_p^k$ acts trivially on the Tits building. Finally, (5) $\simeq$ (6) follows from the relation between the Tits building and the Steinberg idempotent stated above.
\end{proof}

\begin{ex}
	We now identify $S^n(1)$ for $n$ odd. There is an equivalence
	\[
	S^n(1)\simeq \Sigma^{\infty+n-1} B\Sigma_p / B\Sigma_p^{(p-1)(n-1)}
	\]
	To see this, we use description (4) of the proposition. Since $\mathbb{F}_p$ has no nontrivial proper subspaces, we have $T_1=\emptyset$ and hence $T_1^\diamond \simeq S^0$. Therefore
	\[
	S^n(1) \simeq \Sigma^{-1}\mathbb{S}^{n\sigma}_{h\mathrm{Aff}_1\mathbb{F}_p}
	\]
	where $\mathrm{Aff}_1\mathbb{F}_p\cong \mathbb{F}_p\rtimes \mathbb{F}_p^\times$. The inclusion $\mathrm{Aff}_1\mathbb{F}_p \leq \Sigma_p$ induces an equivalence on homotopy orbits; see, for example, \cite[rem. IV.1.6]{NikolausScholze}. Hence
	\[
	S^n(1) \simeq \Sigma^{-1}\mathbb{S}^{n\sigma}_{h\Sigma_p}
	\]
	Writing $\sigma \cong \mathbf{1}\oplus \bar\sigma$, where $\bar\sigma$ denotes the standard $(p-1)$-dimensional representation, this becomes
	\[
	S^n(1) \simeq \Sigma^{n-1}\mathbb{S}^{n\bar\sigma}_{h\Sigma_p}
	\]
	We now compute the $\mathbb{F}_p$-cohomology. Since $\mathbb{S}^{n\bar\sigma}$ has a single nonzero cohomology group, in degree $(p-1)n$, the homotopy orbits spectral sequence gives
	\[
	\mathbb{F}_p^*\mathbb{S}^{n\bar\sigma}_{h\Sigma_p}
	\cong H^{*-(p-1)n}(B\Sigma_p;\mathbb{F}_p\otimes \mathrm{sgn})
	\cong H^{*-(p-1)(n-1)}(B\Sigma_p;\mathbb{F}_p)
	\]
	where we used that the cohomology of $\Sigma_p$ with sign coefficients agrees with the cohomology with trivial coefficients up to a shift by $(p-1)$ \cite[lem. 1.4]{MayAlg}. Using in addition the $2(p-1)$-periodicity of $H^*(B\Sigma_p;\mathbb{F}_p)$, it follows that
	\[
	\mathbb{F}_p^*\mathbb{S}^{n\bar\sigma}_{h\Sigma_p}
	\cong \mathbb{F}_p^*B\Sigma_p
	\qquad\text{for } *>(p-1)(n-1)
	\]
	Moreover, the canonical map $S^0\to S^{n\bar\sigma}$ induces the expected inclusion
	\[
	\mathbb{F}_p^*\mathbb{S}^{n\bar\sigma}_{h\Sigma_p}
	\longrightarrow
	\mathbb{F}_p^*B\Sigma_{p+}
	\]
	It follows that $\mathbb{S}^{n\bar\sigma}_{h\Sigma_p}$ is the suspension spectrum of the cofiber of the inclusion of the $(p-1)(n-1)$-skeleton of $B\Sigma_p$, which proves the claim.
\end{ex}

The computation above can be generalized to all the layers. The $\mathbb{F}_p$-cohomology of $S^n(k)$ was computed by Arone--Mahowald \cite{AroneMahowald}. To state their result, we need to introduce some conventions concerning the Steenrod algebra. 

Let $\mathcal{A}$ denote the mod-$p$ Steenrod algebra. At $p=2$ it is generated by the \emph{Steenrod squares} $Sq^i$, with degree $|Sq^i|=i$, while at odd primes it is generated by the \emph{Bockstein} $\beta$, with $|\beta|=1$, and the \emph{reduced powers} $P^i$, with $|P^i|=2(p-1)i$. For convenience, we adopt the convention that, at $p=2$, we write $\beta\leftrightarrow Sq^1$ and $P^i\leftrightarrow Sq^{2i}$. This allows a uniform treatment of the additive basis of $\mathcal{A}$, but it does not respect the multiplicative structure. Indeed, the Adem relations behave differently: for example, $P^1P^1=2P^2$, which vanishes at $p=2$, whereas $Sq^2Sq^2=Sq^3Sq^1\neq 0$. Thus, this identification is only safe to use for degrees and additive generators, not for products in $\mathcal{A}$.

As an $\mathbb{F}_p$-vector space, the Steenrod algebra $\mathcal{A}$ has basis given by the monomials $\theta^I=\beta^{\epsilon_1}P^{i_1}\cdots \beta^{\epsilon_k}P^{i_k}\beta^\epsilon$ where $I=(\epsilon_1,i_1,\dots,\epsilon_k,i_k,\epsilon)$ is \emph{admissible}, i.e.\ it satisfies $i_j\geq p\,i_{j+1}+\epsilon_{j+1}$. We call the integer $k$ the \emph{length} of the monomial. With our convention, the notion of length at $p=2$ differs from the classical one: for example, $Sq^2Sq^1\leftrightarrow P^1\beta$ is considered to have length $1$ instead of $2$. The algebra structure can only increase the length of a monomial, giving rise to the \emph{length filtration} $\mathcal{A}_{\geq k}$.

To describe the $\mathbb{F}_p$-cohomology of $S^n(k)$, we need to introduce some $\mathcal{A}$-modules.

First, we consider the module $\mathcal{A}/\mathcal{A}\beta$. As an $\mathbb{F}_p$-vector space, this has additive basis given by the monomials $\theta^I=\beta^{\epsilon_1}P^{i_1}\cdots \beta^{\epsilon_k}P^{i_k}$, where $I=(\epsilon_1,i_1,\dots,\epsilon_k,i_k)$ is admissible. Length is well defined here and again gives rise to a filtration $(\mathcal{A}/\mathcal{A}\beta)_{\geq k}$. The associated graded pieces of this filtration are the $\mathcal{A}$-modules consisting of monomials of exact length $k$, namely
\[
  (\mathcal{A}/\mathcal{A}\beta)_{=k}
  :=
  (\mathcal{A}/\mathcal{A}\beta)_{\geq k}\big/\bigl(\mathcal{A}/\mathcal{A}\beta)_{\geq k+1}
\]
The module $(\mathcal{A}/\mathcal{A}\beta)_{=k}$ has additive basis given by the admissible monomials of length exactly $k$. We can further introduce a filtration $(\mathcal{A}/\mathcal{A}\beta)_{=k}^{\geq n}$ by requiring $i_k\geq n$. This filtration is preserved by the $\mathcal{A}$-module structure and, up to a shift, coincides with the $\mathbb{F}_p$-cohomology of the layers.
    
\begin{prop}
	\label{cohomS}
    The $\mathbb{F}_p$-cohomology of the spectra $S^{2n-1}(k)$ is given by
    \begin{align*}
        \mathbb{F}_p^* S^{2n-1}(k) 
        &\cong \Sigma^{2n - 1 - 2k} (\mathcal{A}/\mathcal{A}\beta)_{=k}^{\geq n} \\
        &= \Sigma^{2n - 1 - 2k} \mathbb{F}_p\bigl\langle \theta^I \in \mathcal{A}/\mathcal{A}\beta \;\big|\; I \text{ admissible of length } k \text{ and } i_k \geq n \bigr\rangle \\
        &= \Sigma^{2n - 1 - 2k} \mathbb{F}_p\bigl\langle \beta^{\epsilon_1} P^{i_1} \cdots \beta^{\epsilon_k} P^{i_k} \;\big|\; i_k \geq n, \ i_j \geq p i_{j+1} + \epsilon_{j+1} \bigr\rangle
    \end{align*}
    In particular, the connectivity of $S^{2n-1}(k)$ is $2np^k - 2k - 1$, and the first nontrivial homotopy group is $\mathbb{Z}$ for $k = 0$ and $\mathbb{F}_p$ for $k > 0$.
\end{prop}

\begin{proof}
    The cohomology calculation can be found in \cite[proof of thm. 3.17]{AroneMahowald}. Note that this reduces to a group cohomology computation using any of the descriptions (1)--(6), since the spaces $P_m$ and $T_k$ are spherical. The connectivity can then be deduced from the cohomology calculation.

    Indeed, the monomial $\theta^I$ of lowest degree is the one with $\epsilon_j=0$, $i_k=n$, and $i_j=p\,i_{j+1}=p^{k-j}n$. It follows that its degree is $\textstyle \sum_{j=1}^{k}2(p-1)p^{k-j}n=2n(p^k-1)$, and adding the suspension shift $2n-1-2k$ gives the connectivity.

    For $k>0$, the Bockstein $\beta\theta^I$ is nonzero, so the first nontrivial \emph{integral} cohomology group, and hence the first nontrivial homotopy group, is $\mathbb{F}_p$. The case $k=0$ corresponds to the sphere spectrum.
\end{proof}

\begin{rem}
	It is worth noting that for $k>0$ the cohomology is free over the Bockstein, and in particular we can recover the $\mathbb{Z}$-cohomology as the kernel of $\beta$. This is given by the $\mathbb{F}_p$-vector space
	\[
	\mathbb{Z}^*S^{2n-1}(k)
	\cong
	\Sigma^{2n-1-2k}\mathbb{F}_p
	\bigl\langle
	\beta P^{i_1}\cdots \beta^{\epsilon_k}P^{i_k}
	\mid
	i_k\geq n,\ i_j\geq p i_{j+1}+\epsilon_{j+1}
	\bigr\rangle
	\]
	In general, the $\mathbb{F}_p$-cohomology of $S^{2n-1}(k)$ is free over the subalgebra $\mathcal{A}(k-1)$ \cite[thm. 3.17]{AroneMahowald}.
\end{rem}

\begin{ex}
	\label{cohomS1}
    Using the preceding description, we can recompute the case $k=1$ explicitly. For $n$ odd, we have
    \[
        \mathbb{F}_p^*S^{n}(1)
        \cong
        \Sigma^{n-2}\mathbb{F}_p\bigl\langle P^i, \beta P^i \mid i\geq (n+1)/2\bigr\rangle
        =:
        \Sigma^{n-2}\mathbb{F}_p\bigl\langle \theta_i, \beta\theta_i \mid i\geq (n+1)/2\bigr\rangle
    \]
    where $\theta_i$ corresponds to $P^i$ in degree $2(p-1)i$. The action of the Bockstein is evident. The rest of the $\mathcal{A}$-structure is given by
    \[
        P^\alpha \theta_i = C^\alpha_i\,\theta_{i+\alpha}
        \qquad
        P^\alpha(\beta\theta_i) = \bigl(C^\alpha_i-C^{\alpha-1}_i\bigr)\,\beta\theta_{i+\alpha}
    \]
    where
    \[
        C^\alpha_i := (-1)^\alpha \binom{(p-1)i-1}{\alpha}
    \]
    This follows from applying the Adem relations to $P^\alpha P^i$ and then killing every term of length greater than $1$. The reader should compare this with the cohomology of $B\Sigma_p$.
\end{ex}

So far, we have only discussed odd spheres, and the curious reader may wonder what happens for even spheres. Classically, even spheres are treated via James' \emph{EHP sequence} \cite{EHP}, which expresses them as an extension of two odd spheres. More precisely, for $n$ even, the classical EHP sequence gives a fiber sequence
\[
    S^{n-1} \longrightarrow \Omega S^{n} \longrightarrow \Omega S^{2n-1}
\]
Moreover, for $p>2$ the EHP sequence \emph{splits} \cite[prop. 4.4.4]{unstablemethods}, reducing their study to that of odd spheres. Since the EHP sequence can be made functorial, we obtain a calculus version of it.

\begin{prop}[EHP sequence - calculus version]
	\label{EHPcalc}
    For $p=2$ there is a fiber sequence
    \[
        S^{n-1}_k \longrightarrow \Omega S^n_k \longrightarrow \Omega S^{2n-1}_{k-1}
    \]
    For $p>2$ and $n$ even, this fiber sequence splits, yielding
    \[
        \Omega S^n_k \simeq S^{n-1}_k \times \Omega S^{2n-1}_{k-1}
    \]
\end{prop}

\begin{proof}
    The first part is stated in \cite[cor. 2.1.4]{Behrens2010}. The second part can be proved by a similar argument, but since the same strategy will later be used to establish the calculus version of the Gray sequence, we include the proof here.

    The classical EHP sequence can be promoted to a sequence of functors
    \[
        I \longrightarrow \Omega \Sigma \longrightarrow \Omega \Sigma (-)^{\otimes 2}
    \]
    which becomes a fiber sequence when evaluated on odd spheres, and which splits for $p>2$. In particular, for $p>2$ one has the functorial splitting
    \[
        \Omega \Sigma \simeq I \times \Omega \Sigma (-)^{\otimes 2}
    \]
    on odd spheres.

    From this, together with the fact that $P_m$ preserves finite limits and lemma \ref{calcspheres}, we immediately obtain
    \[
        \Omega S^n_k
        =
        \Omega P_{p^k}I(S^n)
        \simeq
        P_{p^k}(\Omega\Sigma)(S^{n-1})
        \simeq
        P_{p^k}\bigl(I \times \Omega\Sigma(-)^{\otimes 2}\bigr)(S^{n-1})
    \]
    Since $P_m$ preserves finite products and commutes with loops, this becomes
    \[
        \Omega S^n_k
        \simeq
        P_{p^k}I(S^{n-1})
        \times
        \Omega P_{p^k}\bigl(\Sigma(-)^{\otimes 2}\bigr)(S^{n-1})
        \simeq
        S^{n-1}_k
        \times
        \Omega P_{p^k}\bigl(\Sigma(-)^{\otimes 2}\bigr)(S^{n-1})
    \]
    It remains to analyze the last term. By lemma \ref{calcpowers}, we have
    \[
        P_{p^k}\bigl(\Sigma(-)^{\otimes 2}\bigr)(S^{n-1})
        \simeq
        P_{p^k/2}I\bigl(\Sigma(S^{n-1})^{\otimes 2}\bigr)
        \simeq
        P_{p^k/2}I(S^{2n-1})
    \]
    Since for odd spheres the only nontrivial layers occur in degrees $p^j$, we conclude that
    \[
        P_{p^k/2}I(S^{2n-1})
        \simeq
        P_{p^{k-1}}I(S^{2n-1})
        =
        S^{2n-1}_{k-1}
    \]
    Therefore
    \[
        \Omega S^n_k \simeq S^{n-1}_k \times \Omega S^{2n-1}_{k-1}
    \]
    as claimed.
\end{proof}

\begin{rem}
	Note that it is important that the splitting be \emph{functorial}. A non-functorial splitting does not induce a splitting on polynomial approximations. For example, the classical Hopf fibration yields $\Omega S^2 \simeq S^1 \times \Omega S^3$, but it is not true that $\Omega S^2_k \simeq S^1_k \times \Omega S^3_{k-1}$. This already fails for $k=1$, as can be seen from the metastable computations of subsection \ref{metastable}.
\end{rem}

\begin{rem}
	Observe that the discussion above implies, as already anticipated, that for even spheres the layer $\mathbb{D}_mI$ is also nontrivial when $m=2p^k$. Indeed, for $p>2$ and $n$ even one can show
	\[
    \mathbb{D}_{p^k}I(S^n) \simeq \Sigma S^{n-1}(k)
    \qquad
    \mathbb{D}_{2p^k}I(S^n) \simeq S^{2n-1}(k)
	\]
\end{rem}

Combining the convergence of the Goodwillie tower of the identity with the connectivity of the layers $S^n(k)$ for $n$ odd, together with the calculus version of the EHP sequence for $n$ even, we deduce the range of convergence of the map $S^n \to S^n_k$, thereby generalizing the classical Freudenthal theorem.

\begin{thm}[Generalized Freudenthal theorem]
    \label{genfreudenthal}
    The comparison map $S^n \to S^n_k$ is an equivalence on $\pi_*$ for
    \[
    * <
    \begin{cases}
        (n+1)p^{k+1} - 2k - 3 & p=2 \text{ or $n$ odd} \\
        2np^k - 2k - 1 & p>2 \text{ and $n$ even}
    \end{cases}
    \]
    and is surjective in the next degree.
\end{thm}

\begin{proof}
    Set $c(n,k) := (n+1)p^k-2k-1$, which for $n$ odd agrees with the connectivity of $S^n(k)$ by proposition \ref{cohomS}. Note that it satisfies the recursive formulas
    \[
    c(n+1,k)=c(n,k)+p^k
    \qquad
    c(n,k+1)=c(p(n+1)-1,k)-2
    \]

    The comparison map factors as
    \[
        S^n
        \longrightarrow
        \mathrm{lim}_j S^n_j
        \longrightarrow
        \cdots
        \longrightarrow
        S^n_{j+1}
        \longrightarrow
        S^n_j
        \longrightarrow
        \cdots
        \longrightarrow
        S^n_{k+1}
        \longrightarrow
        S^n_k
    \]
    The leftmost map is an equivalence by convergence of the Goodwillie tower of the identity.
	
	For $n$ odd, consider the fiber sequence
	\[
	\Omega^\infty S^n(j+1) \longrightarrow S^n_{j+1} \longrightarrow S^n_j
	\]
	The right-hand map induces an equivalence on $\pi_*$ in degrees $*<c(n,j+1)$ by connectivity. Since $c(n,j)$ is an increasing function of $j$ and $j\geq k$, it follows that every map in the factorization above is an equivalence on $\pi_*$ for $*<c(n,k+1)$.

	At the borderline degree $*=c(n,k+1)$, the rightmost map is only surjective, while all the others remain equivalences, since $c(n,j)$ is \emph{strictly} increasing, except in the case $(p,n,k)=(2,1,0)$, which can be handled separately.

    For $n$ even and $p=2$, one could argue in the same way, since the cohomology of the layers has the same connectivity as in the odd case \cite[thm. 3.16]{AroneMahowald}. However, since we have not discussed that computation, we give an alternative proof. From the rotated EHP sequence (prop. \ref{EHPcalc})
    \[
    \Omega^2S^{2n+1}_{k-1}
    \longrightarrow
    S^n_k
    \longrightarrow
    \Omega S^{n+1}_k
    \]
    the five lemma, together with the odd case already established, implies that the comparison map $S^n \to S^n_k$ is an isomorphism on $\pi_*$ whenever this holds for the other two terms. This is the case provided
    \begin{align*}
        * &< \min\big(c(n+1,k+1)-1,\; c(2n+1,k)-2\big) \\
          &= \min\big(c(n,k+1)+2^{k+1}-1,\; c(n,k+1)\big) \\
          &= c(n,k+1)
    \end{align*}

    For $n$ even and $p>2$, the claim follows from the splitting of the EHP sequence (prop. \ref{EHPcalc}), namely the equivalence
    \[
    \Omega S^n_k \simeq S^{n-1}_k \times \Omega S^{2n-1}_{k-1}
    \]
    Combining this with the odd case already established, the comparison map is an equivalence on $\pi_*$ provided
    \begin{align*}
        * &< \min \big(c(n-1,k+1)+1,\; c(2n-1,k)\big)  \\
		  &= \min \big(c(pn-1,k)-1,\; c(2n-1,k)\big)  \\
          &= \min \big(c(2n-1,k)+np^k(p-2)-1,\; c(2n-1,k)\big) \\
          &= c(2n-1,k)
    \end{align*}
    as required.
\end{proof}

\begin{rem} 
	For $k=0$ and $p=2$, the theorem states that the map $S^n \longrightarrow S^n_0 \simeq \Omega^\infty \mathbb{S}^n$ is an equivalence on $\pi_*$ for $* < 2n-1$ and is surjective in the next degree. This is exactly the classical Freudenthal suspension theorem. 
	
	More generally, for $k=0$ and $n$ odd, the theorem gives that $S^n \longrightarrow S^n_0 \simeq \Omega^\infty \mathbb{S}^n$ is a $p$-local equivalence on $\pi_*$ for $* < (n+1)p-3$ and is surjective in the next degree. This recovers Serre's refinement of the stable range for odd spheres \cite[prop. 4]{Serreodd}.
\end{rem}

In general, we will call the range of convergence for $k=0$ the \textit{stable range}, and the range of convergence for $k=1$ the \textit{metastable range}. Observe that, unlike in the classical usage, these ranges now depend on the implicit prime.

The table below displays the first homotopy groups of spheres at $p=2$ on the left and at $p=3$ on the right. The stable range is highlighted in purple, the metastable range in blue, and the $k=2$ range in green. As before, $\square$ denotes a copy of $\mathbb{Z}$, while $\bullet_r$ denotes a copy of $\mathbb{Z}/p^r$; when $r=1$, we omit the subscript.
\[
\adjustbox{scale=0.91,center}{
\begin{tabular}{l|l|l|l|l|l|l|l|l|l}
  & $S^1$     & $S^2$        & $S^3$ & $S^4$  & $S^5$ & $S^6$ & $S^7$ & $S^8$  & $S^9$ \\
  \hline
0 & \cellcolor{CornflowerBlue}$\square$ & \cellcolor{Orchid}$\square$ &  \cellcolor{Orchid}$\square$  & \cellcolor{Orchid}$\square$   & \cellcolor{Orchid}$\square$  & \cellcolor{Orchid}$\square$  & \cellcolor{Orchid}$\square$  & \cellcolor{Orchid}$\square$   \cellcolor{Orchid}& \cellcolor{Orchid}$\square$  \\
1 &     \cellcolor{CornflowerBlue}      & \cellcolor{CornflowerBlue}$\square$ & \cellcolor{Orchid}$\bullet$  & \cellcolor{Orchid}$\bullet$   & \cellcolor{Orchid}$\bullet$  & \cellcolor{Orchid}$\bullet$  & \cellcolor{Orchid}$\bullet$  & \cellcolor{Orchid}$\bullet$   & \cellcolor{Orchid}$\bullet$  \\
2 &  \cellcolor{SpringGreen}         & \cellcolor{CornflowerBlue}$\bullet$         & \cellcolor{CornflowerBlue}$\bullet$  & \cellcolor{Orchid}$\bullet$   & \cellcolor{Orchid}$\bullet$  & \cellcolor{Orchid}$\bullet$  & \cellcolor{Orchid}$\bullet$  & \cellcolor{Orchid}$\bullet$   & \cellcolor{Orchid}$\bullet$  \\
3 &    \cellcolor{SpringGreen}       & \cellcolor{CornflowerBlue}$\bullet$         & \cellcolor{CornflowerBlue}$\bullet_2$  & \cellcolor{CornflowerBlue}$\square \bullet_2$ & \cellcolor{Orchid}$\bullet_3$  & \cellcolor{Orchid}$\bullet_3$  & \cellcolor{Orchid}$\bullet_3$  & \cellcolor{Orchid}$\bullet_3$  & \cellcolor{Orchid}$\bullet_3$  \\
4 &    \cellcolor{SpringGreen}       & \cellcolor{CornflowerBlue}$\bullet_2$         & \cellcolor{CornflowerBlue}$\bullet$  & \cellcolor{CornflowerBlue}$\bullet \bullet$  & \cellcolor{CornflowerBlue}$\bullet$  &  \cellcolor{Orchid}  &  \cellcolor{Orchid}  &  \cellcolor{Orchid}   &   \cellcolor{Orchid} \\
5 &   \cellcolor{SpringGreen}        & \cellcolor{SpringGreen}$\bullet$         & \cellcolor{CornflowerBlue}$\bullet$  & \cellcolor{CornflowerBlue}$\bullet \bullet$  & \cellcolor{CornflowerBlue}$\bullet$  & \cellcolor{CornflowerBlue}$\square$  &  \cellcolor{Orchid}  &  \cellcolor{Orchid}   &  \cellcolor{Orchid}  \\
6 &     \cellcolor{SpringGreen}      & \cellcolor{SpringGreen}$\bullet$         &  \cellcolor{CornflowerBlue}  & \cellcolor{CornflowerBlue}$\bullet_3$   & \cellcolor{CornflowerBlue}$\bullet$  & \cellcolor{CornflowerBlue}$\bullet$  & \cellcolor{CornflowerBlue}$\bullet$  & \cellcolor{Orchid}$\bullet$   & \cellcolor{Orchid}$\bullet$  \\
7 &     \cellcolor{SpringGreen}      &     \cellcolor{SpringGreen}      &  \cellcolor{CornflowerBlue}  &  \cellcolor{CornflowerBlue}   & \cellcolor{CornflowerBlue}$\bullet$  & \cellcolor{CornflowerBlue}$\bullet_2$  & \cellcolor{CornflowerBlue}$\bullet_3$  & \cellcolor{CornflowerBlue}$\square \bullet_3$ & \cellcolor{Orchid}$\bullet_4$
\end{tabular} \qquad
\begin{tabular}{l|l|l|l|l|l|l|l|l|l}
  & $S^1$     & $S^2$        & $S^3$ & $S^4$  & $S^5$ & $S^6$ & $S^7$ & $S^8$  & $S^9$ \\
  \hline
0 & \cellcolor{Orchid}$\square$ & \cellcolor{Orchid}$\square$ &  \cellcolor{Orchid}$\square$  & \cellcolor{Orchid}$\square$   & \cellcolor{Orchid}$\square$  & \cellcolor{Orchid}$\square$  & \cellcolor{Orchid}$\square$  & \cellcolor{Orchid}$\square$   \cellcolor{Orchid}& \cellcolor{Orchid}$\square$  \\
1 &    \cellcolor{Orchid}       & \cellcolor{CornflowerBlue}$\square$ & \cellcolor{Orchid}  & \cellcolor{Orchid}   & \cellcolor{Orchid}  & \cellcolor{Orchid}  & \cellcolor{Orchid}  & \cellcolor{Orchid}   & \cellcolor{Orchid}  \\
2 &   \cellcolor{CornflowerBlue}        &   \cellcolor{CornflowerBlue}       & \cellcolor{Orchid}  & \cellcolor{Orchid}   & \cellcolor{Orchid}  & \cellcolor{Orchid}  & \cellcolor{Orchid}  & \cellcolor{Orchid}   & \cellcolor{Orchid}  \\
3 &    \cellcolor{CornflowerBlue}       &   \cellcolor{CornflowerBlue}       & \cellcolor{Orchid} $\bullet$ & \cellcolor{CornflowerBlue} $\square \bullet$ & \cellcolor{Orchid} $\bullet$  & \cellcolor{Orchid} $\bullet$ & \cellcolor{Orchid} $\bullet$ & \cellcolor{Orchid} $\bullet$  & \cellcolor{Orchid} $\bullet$ \\
4 &    \cellcolor{CornflowerBlue}    &     \cellcolor{CornflowerBlue}  $\bullet$   & \cellcolor{Orchid}  & \cellcolor{CornflowerBlue}  & \cellcolor{Orchid}  &  \cellcolor{Orchid}  &  \cellcolor{Orchid}  &  \cellcolor{Orchid}   &   \cellcolor{Orchid} \\
5 &    \cellcolor{CornflowerBlue}       &   \cellcolor{CornflowerBlue}       &  \cellcolor{Orchid} & \cellcolor{CornflowerBlue}  &  \cellcolor{Orchid} & \cellcolor{CornflowerBlue}$\square$  &  \cellcolor{Orchid}  &  \cellcolor{Orchid}   &  \cellcolor{Orchid}  \\
6 &    \cellcolor{CornflowerBlue}       &   \cellcolor{CornflowerBlue}       &  \cellcolor{CornflowerBlue} $\bullet$ &  \cellcolor{CornflowerBlue} $\bullet\bullet$ & \cellcolor{Orchid}  & \cellcolor{CornflowerBlue} & \cellcolor{Orchid}  & \cellcolor{Orchid}   & \cellcolor{Orchid}  \\
7 &    \cellcolor{CornflowerBlue}       &   \cellcolor{SpringGreen}    $\bullet$    &  \cellcolor{CornflowerBlue} $\bullet$  &  \cellcolor{CornflowerBlue}  $\bullet$ &  \cellcolor{Orchid}$\bullet$ & \cellcolor{CornflowerBlue}$\bullet$ & \cellcolor{Orchid} $\bullet$ & \cellcolor{CornflowerBlue} $\square \bullet$ & \cellcolor{Orchid} $\bullet$
\end{tabular}
}
\]

\subsection{Gray sequence}

At the prime $p=2$, the EHP sequence gives a uniform description of the fiber of the suspension map, and hence provides the classical inductive tool for studying spheres. At odd primes, the corresponding description is available only for even spheres, and so it does not give a comparable tool for treating odd spheres. For our purposes, the replacement is Gray's sequence, which describes the fiber $W$ of the double suspension map. In this subsection, we construct a calculus version of the Gray sequence, mirroring Behrens' approach to the EHP sequence. The resulting fiber sequences provide the inductive step for later calculations. Beyond its computational use, the calculus version of the Gray sequence suggests a way to recover the main results of Arone--Mahowald. It also allows us to show that the $k^\text{th}$ approximation of $W$ has exponent $p$ for $k\leq 2$, without invoking the Cohen--Moore--Neisendorfer theorem. These final two applications are independent of the rest of the paper and are meant mainly as illustrations of the scope of the construction; they may therefore be skipped on a first reading.

Define $W$ to be the fiber of the map of functors $I \to \Omega^2\Sigma^2$, and write $W^n := W(S^n)$. Since the functors $P_m$ and $\mathbb{D}_m$ preserve fiber sequences of reduced functors, we obtain fiber sequences
\[
P_mW \longrightarrow P_mI \longrightarrow \Omega^2P_mI\Sigma^2
\qquad
\mathbb{D}_mW \longrightarrow \mathbb{D}_mI \longrightarrow \Omega^2\mathbb{D}_mI\Sigma^2
\]
From this one sees that the Goodwillie tower of $W$ converges as well, and that on odd spheres the only nontrivial layers of $W$ occur in degrees $p^k$. This motivates, as before, the notation
\[
	W^n_k := P_{p^k}W(S^n)
	\qquad
    W^n(k) := \mathbb{D}_{p^k}W(S^n)
\]
With this convention, the fiber sequences above specialize to
\[
    W^n_k \longrightarrow S^n_k \longrightarrow \Omega^2S^{n+2}_k
    \qquad
    W^n(k) \longrightarrow S^n(k) \longrightarrow \Sigma^{-2}S^{n+2}(k)
\]
As before, we know the $\mathbb{F}_p$-cohomology of the layers. Define the $\mathcal{A}$-module
\[
    (\mathcal{A}/\mathcal{A}\beta)_{=k}^{=n}
    :=
    (\mathcal{A}/\mathcal{A}\beta)_{=k}^{\geq n}
    \big/
    (\mathcal{A}/\mathcal{A}\beta)_{=k}^{\geq n+1}
\]
to be the associated graded with respect to the last-index filtration. If we denote
\[
    c(n,k):=(n+1)p^k-2k-1
\]
as before, then proposition \ref{cohomS} immediately yields the following.

\begin{prop}
    The $\mathbb{F}_p$-cohomology of the spectra $W^{2n-1}(k)$ is given by
    \begin{align*}
        \mathbb{F}_p^*W^{2n-1}(k)
        &\cong \Sigma^{2n-1-2k}(\mathcal{A}/\mathcal{A}\beta)_{=k}^{=n} \\
        &= \Sigma^{2n-1-2k}\mathbb{F}_p\bigl\langle \theta^I \in \mathcal{A}/\mathcal{A}\beta \mid I \text{ admissible of length } k \text{ and } i_k=n \bigr\rangle \\
        &= \Sigma^{2n-1-2k}\mathbb{F}_p\bigl\langle \beta^{\epsilon_1}P^{i_1}\cdots \beta^{\epsilon_k}P^n \mid i_j\geq p i_{j+1}+\epsilon_{j+1} \bigr\rangle
    \end{align*}
    In particular, the connectivity of $W^{2n-1}(k)$ is $c(2n-1,k)$, and the first nontrivial homotopy group is $\mathbb{F}_p$ for $k>0$.
\end{prop}

We record a few explicit examples of the layers $W^{2n-1}(k)$ in low degrees.

\begin{ex}
    The layer $W^{2n-1}(0)$ is contractible. This follows from the fact that the double suspension map induces an equivalence on the first derivative, or equivalently from the cohomology computation above.
\end{ex}

\begin{ex}
	\label{cohomW1}
    We have $W^{2n-1}(1) \simeq \mathbb{S}^{2np-3}/p$ and $W^{2n-1}_1 \simeq \Omega^\infty\mathbb{S}^{2np-3}/p$. The first equivalence follows from the computation above, which gives
    \[
        \mathbb{F}_p^*W^{2n-1}(1)
        \cong
        \Sigma^{2n-3}\mathbb{F}_p\langle P^n,\beta P^n\rangle
    \]
    and there is a unique connective spectrum with such cohomology. The second equivalence then follows from the fact that this is the first nontrivial layer.
\end{ex}

\begin{ex}
	\label{cohomW2}
	The cohomology of $W^{2n-1}(2)$ can be described as follows:
	\begin{align*}
		\mathbb{F}_p^*W^{2n-1}(2)
		&\cong \Sigma^{2n-5}\mathbb{F}_p\bigl\langle P^iP^n,\beta P^iP^n,P^j\beta P^n,\beta P^j\beta P^n \mid i\geq np,\ j\geq np+1 \bigr\rangle \\
		&=: \Sigma^{2np-5}\mathbb{F}_p\bigl\langle \phi_i,\beta\phi_i,\psi_j,\beta\psi_j \mid i\geq np,\ j\geq np+1 \bigr\rangle
	\end{align*}
	where $\phi_i$ corresponds to $P^iP^n$, shifted to degree $2(p-1)i$, and $\psi_i$ corresponds to $P^i\beta P^n$, shifted to degree $2(p-1)i+1$. As an $\mathcal{A}$-module, it begins as follows:
	% https://q.uiver.app/#q=WzAsNixbMCwwLCJcXHBoaV97bnB9Il0sWzEsMCwiXFxiZXRhXFxwaGlfe25wfSJdLFszLDAsIlxccGhpX3tucCsxfSJdLFs0LDAsIlxcYmV0YVxccGhpX3tucCsxfSJdLFs0LDEsIlxccHNpX3tucCsxfSJdLFs1LDEsIlxcYmV0YVxccHNpX3tucCsxfSJdLFswLDEsIiIsMCx7InN0eWxlIjp7ImhlYWQiOnsibmFtZSI6Im5vbmUifX19XSxbMiwzLCIiLDAseyJzdHlsZSI6eyJoZWFkIjp7Im5hbWUiOiJub25lIn19fV0sWzQsNSwiIiwwLHsic3R5bGUiOnsiaGVhZCI6eyJuYW1lIjoibm9uZSJ9fX1dLFswLDIsIiIsMSx7ImN1cnZlIjotMiwic3R5bGUiOnsiaGVhZCI6eyJuYW1lIjoibm9uZSJ9fX1dLFsxLDQsIiIsMCx7ImN1cnZlIjoyLCJzdHlsZSI6eyJoZWFkIjp7Im5hbWUiOiJub25lIn19fV1d
	\[\begin{tikzcd}
	{\phi_{np}} & {\beta\phi_{np}} && {\phi_{np+1}} & {\beta\phi_{np+1}} \\
	&&&& {\psi_{np+1}} & {\beta\psi_{np+1}}
	\arrow[no head, from=1-1, to=1-2]
	\arrow[curve={height=-12pt}, no head, from=1-1, to=1-4]
	\arrow[curve={height=12pt}, no head, from=1-2, to=2-5]
	\arrow[no head, from=1-4, to=1-5]
	\arrow[no head, from=2-5, to=2-6]
	\end{tikzcd}\]
	Indeed, the action of the Bockstein is evident. The rest of the $\mathcal{A}$-structure follows from the Adem relations and is given by
	\begin{align*}
		P^\alpha\phi_i = C^\alpha_i\phi_{i+\alpha}
		\qquad
		&P^\alpha(\beta\phi_i) = (C^\alpha_i-C^{\alpha-1}_i)\beta\phi_{i+\alpha}+C^{\alpha-1}_i\psi_{i+\alpha} \\
		P^\alpha\psi_i = C^\alpha_i\psi_{i+\alpha}
		\qquad
		&P^\alpha(\beta\psi_i) = (C^\alpha_i-C^{\alpha-1}_i)\beta\psi_{i+\alpha}
	\end{align*}
	
	Here $C^\alpha_i$ denotes the same coefficient that appeared in the cohomology calculation of $S^{2n-1}(1)$ (example \ref{cohomS1}). To make the connection between the two computations even more explicit, observe that there is an injection
	\begin{align*}
    \mathbb{F}_p^*W^{2n-1}(2)
    &\longrightarrow
    \Sigma^{2np-5}\mathcal{A}(0)\otimes \mathbb{F}_p\langle\theta_i,\beta\theta_i \mid i\geq np\rangle
    \cong \mathbb{F}_p^*\Sigma^{-2}S^{2np-1}(1)/p \\
    \phi_i &\longmapsto 1\otimes \theta_i \\
    \beta\phi_i &\longmapsto 1\otimes \beta\theta_i+\beta\otimes \theta_i \\
    \psi_i &\longmapsto \beta\otimes \theta_i \\
    \beta\psi_i &\longmapsto -\beta\otimes \beta\theta_i
	\end{align*}

	One checks that this is an $\mathcal{A}$-module map by using the relations
	\begin{align*}
		\beta(1\otimes x) &= 1\otimes \beta x+\beta\otimes x
		&
		P^\alpha(1\otimes x) &= 1\otimes P^\alpha x \\
		\beta(\beta\otimes x) &= -\beta\otimes \beta x
		&
		P^\alpha(\beta\otimes x) &= \beta\otimes P^\alpha x
	\end{align*}
	together with the formulas displayed above for the $\mathcal{A}$-structure. The map is not surjective, since the generators $\psi_i,\beta\psi_i$ only appear for $i\geq np+1$. Its cokernel is generated by the two classes $\{1\otimes\beta\theta_{np},\beta\otimes\beta\theta_{np}\}$, still with a $(2np-5)$-shift, and this is exactly the cohomology of $\Sigma^{-1}W^{2np-1}(1)$. In conclusion, we have a short exact sequence of $\mathcal{A}$-modules
	\[
	0
	\longrightarrow
	\mathbb{F}_p^*W^{2n-1}(2)
	\longrightarrow
	\mathbb{F}_p^*\Sigma^{-2}S^{2np-1}(1)/p
	\longrightarrow
	\mathbb{F}_p^*\Sigma^{-1}W^{2np-1}(1)
	\longrightarrow
	0
	\]
	As we will see in the next section, this short exact sequence indeed arises from a sequence of spectra.
\end{ex}

\begin{rem}
	The discussion above produces a fiber sequence
	\[
	W^{2n-1}_2
	\longrightarrow
	\Omega^\infty\mathbb{S}^{2np-3}/p
	\longrightarrow
	B\Omega^\infty W^{2n-1}(2)
	\]
	Since the map $W^{2n-1} \to W^{2n-1}_2$ is a $v_2$-periodic equivalence \cite[thm. 4.1]{AroneMahowald}, this recovers Thompson's description of the $v_2$-periodic fiber of the double suspension \cite[thm. 1.5]{L2W}. In fact, $W^{2n-1}(2)$ agrees with $\Sigma^{-1}C$, where $C$ denotes the spectrum appearing in \cite[fiber sequence 1.1]{L2W}.
\end{rem}

From the cohomological calculations, or equivalently from the fiber sequences above together with the generalized Freudenthal theorem, one immediately deduces the following convergence result for $W$.

\begin{prop}[Generalized Freudenthal theorem for $W$]
	\label{genfreudenthalW}
	The comparison map $W^{2n-1} \to W^{2n-1}_k$ is an equivalence on $\pi_*$ for $* < 2np^{k+1} - 2k - 3$ and is surjective in the next degree. 
\end{prop}

In the case $p=2$, the EHP sequence identifies the fiber of the suspension map in terms of another sphere. For a general prime $p$, this is no longer the case. Instead, the fiber of the \emph{double suspension} fits into a fiber sequence of the form
\[
    BW^{2n-1} \longrightarrow \Omega^2S^{2np+1} \longrightarrow S^{2np-1}
\]
due to Gray \cite{Gray1989}. We will refer to this as the \emph{Gray sequence}.

Moreover, the map $\Omega^2S^{2np+1} \to S^{2np-1}$ factors multiplication by $p$ when composed on either side with the double suspension. Thus one has commutative diagrams
% https://q.uiver.app/#q=WzAsNixbMCwwLCJcXE9tZWdhXjJTXnsybnArMX0iXSxbMSwwLCJTXnsybnAtMX0iXSxbMSwxLCJcXE9tZWdhXjJTXnsybnArMX0iXSxbMywwLCJTXnsybnAtMX0iXSxbNCwwLCJcXE9tZWdhXjJTXnsybnArMX0iXSxbNCwxLCJTXnsybnAtMX0iXSxbMCwxXSxbMSwyLCJFXjIiXSxbMCwyLCJwIiwyXSxbMyw0LCJFXjIiXSxbNCw1XSxbMyw1LCJwIiwyXV0=
\[\begin{tikzcd}
	{\Omega^2S^{2np+1}} & {S^{2np-1}} && {S^{2np-1}} & {\Omega^2S^{2np+1}} \\
	& {\Omega^2S^{2np+1}} &&& {S^{2np-1}}
	\arrow[from=1-1, to=1-2]
	\arrow["p"', from=1-1, to=2-2]
	\arrow["{E^2}", from=1-2, to=2-2]
	\arrow["{E^2}", from=1-4, to=1-5]
	\arrow["p"', from=1-4, to=2-5]
	\arrow[from=1-5, to=2-5]
\end{tikzcd}\]
Before turning to the calculus version, let us explain the formal pattern in the case $p=2$. This discussion is intended only as motivation: the EHP sequence is available on spheres, but not as a functorial construction on all spaces. One could make this precise using orthogonal calculus, but since that machinery does not otherwise appear in this paper, we will avoid introducing it here. The knowledgeable reader may wish to keep this perspective in mind.

Recall that the fiber of a composite fits into a fiber sequence
\[
    \fib(X \to Y)
    \longrightarrow
    \fib(X \to Y \to Z)
    \longrightarrow
    \fib(Y \to Z)
\]
If the first fiber sequence deloops, then the composite
\[
    \fib(Y \to Z)
    \longrightarrow
    Y
    \longrightarrow
    B\fib(X \to Y)
\]
induces a delooping of the previous sequence, yielding
\[
    \fib(X \to Y \to Z)
    \longrightarrow
    \fib(Y \to Z)
    \longrightarrow
    B\fib(X \to Y)
\]
Applying this formally to the factorization $E^2 \simeq \Omega E\Sigma \circ E$, and using that $\fib(E)$ deloops on spheres, one obtains on spheres a fiber sequence
\[
    \fib(E^2)
    \longrightarrow
    \fib(\Omega E\Sigma)
    \simeq
    \Omega\fib(E)\Sigma
    \longrightarrow
    B\fib(E)
\]
On spheres, the EHP sequence identifies $B\fib(E)$ with $\Omega\Sigma(-)^{\otimes 2}$, and hence yields a fiber sequence
\[
    W
    \longrightarrow
    \Omega^3\Sigma^3(-)^{\otimes 2}
    \longrightarrow
    \Omega\Sigma(-)^{\otimes 2}
\]
which is precisely the looped Gray sequence at $p=2$.

For an arbitrary prime, one might hope to imitate this argument using the classical fiber sequences
\[
    J_{p-1}S^{2n}
    \longrightarrow
    \Omega S^{2n+1}
    \longrightarrow
    \Omega S^{2np+1}
\]
\[
    S^{2n-1}
    \longrightarrow
    \Omega J_{p-1}S^{2n}
    \longrightarrow
    \Omega S^{2np-1}
\]
where $J_m$ denotes the $m^\mathrm{th}$ stage of the James construction. However, the second sequence does not seem to admit a sufficiently functorial form for our purposes, and this approach also does not directly account for the factorization through multiplication by $p$. We therefore rely instead on the construction of Richter \cite{Richter2014}.

\begin{thm}[Gray sequence -- calculus version]
	\label{graysequence}
    There are fiber sequences
    \[
        W^{2n-1}_k \;\longrightarrow\; \Omega^3 S^{2np+1}_{k-1} \;\longrightarrow\; \Omega S^{2np-1}_{k-1}
    \]
    \[
        W^{2n-1}(k) \;\longrightarrow\; \Sigma^{-3} S^{2np+1}(k-1) \;\longrightarrow\; \Sigma^{-1} S^{2np-1}(k-1)
    \]
	Moreover, the right-hand maps factor multiplication by $p$ when composed on either side with the double suspension.
\end{thm}

\begin{proof}
	We first show that, on odd spheres, there is a fiber sequence of functors
	\[
	    W
	    \longrightarrow
	    \Omega^3\Sigma^{p+1}(-)^{\otimes p}
	    \longrightarrow
	    \Omega\Sigma^{p-1}(-)^{\otimes p}
	\]
	such that the right-hand map factors multiplication by $p$ when composed on either side with the double suspension.

	Once this is established, the stated fiber sequences follow by applying lemmas \ref{calcspheres} and \ref{calcpowers}. Indeed, these lemmas allow us to pass from spherewise equivalences of functors to the corresponding statements for Goodwillie approximations and layers evaluated on spheres. In this way we obtain
	\[
	    P_{p^k}W(S^{2n-1})
	    \longrightarrow
	    \Omega^3P_{p^{k-1}}I\bigl(\Sigma^{p+1}(S^{2n-1})^{\otimes p}\bigr)
	    \longrightarrow
	    \Omega P_{p^{k-1}}I\bigl(\Sigma^{p-1}(S^{2n-1})^{\otimes p}\bigr)
	\]
	and similarly for the homogeneous layers. Since $P_m$ preserves finite limits of reduced functors, the resulting sequences are again fiber sequences. The factorization through multiplication by $p$ is also preserved after applying $P_m$, since on loop spaces multiplication by $p$ is defined in terms of the canonical $H$-space structure, and $P_m$ preserves finite products and loop objects.

	We now construct the required spherewise fiber sequence. By \cite{Todadouble}, there is a sequence of functors
	\[
	    J_{p-1}
	    \longrightarrow
	    J \simeq \Omega\Sigma
	    \longrightarrow
	    \Omega\Sigma(-)^{\otimes p}
	\]
	where $J$ denotes the James construction. The evaluation of this sequence on even spheres is a fiber sequence. It therefore suffices to construct, on odd spheres, a fiber sequence
	\[
	    BW
	    \longrightarrow
	    \Omega^2\Sigma^{p+1}(-)^{\otimes p} \simeq \fib\bigl(J_{p-1}\Sigma \to J\Sigma\bigr)
	    \longrightarrow
	    \Sigma^{p-1}(-)^{\otimes p}
	\]
	The right-hand map is obtained from Richter's construction, which is functorial on suspensions; this is the content of lemma \ref{richtergraymap} below. By this lemma, its fiber on odd spheres is $BW$, and the map factors multiplication by $p$ when composed on either side with the double suspension. Looping once more gives the desired spherewise fiber sequence
	\[
    	W
	    \longrightarrow
    	\Omega^3\Sigma^{p+1}(-)^{\otimes p}
    	\longrightarrow
    	\Omega\Sigma^{p-1}(-)^{\otimes p}
	\]
	and places multiplication by $p$ on the loop space, by the Eckmann--Hilton argument. This completes the proof.
\end{proof}

\begin{lem}[{\cite[sec. 4]{Richter2014}}]
	\label{richtergraymap}
	On the subcategory of suspensions, there is a natural transformation
	\[
	    \fib\bigl(J_{p-1}\Sigma \to J\Sigma\bigr)
	    \longrightarrow
	    \Sigma^{p-1}(-)^{\otimes p}
	\]
	such that, when evaluated on an odd sphere $S^{2n-1}$, this map identifies with the map
	\[
		\Omega^2 S^{2np+1} \to S^{2np-1}
	\]
	appearing in Gray's sequence. In particular, its fiber is $BW^{2n-1}$, and the resulting map factors multiplication by $p$ when composed on either side with the double suspension.
\end{lem}

\begin{proof}
	For any co-H space, Richter \cite[lem. 3.1, $a=0$, $b=p-1$]{Richter2014} constructs a natural transformation
	\[
	    \fib(J_{p-1} \to J)
	    \longrightarrow
	    \Omega J \otimes (-)^{\otimes(p-1)}
	\]
	In particular, since suspensions are co-H spaces, there is a natural transformation
	\[
	    \fib(J_{p-1}\Sigma \to J\Sigma)
	    \longrightarrow
	    \Omega J\Sigma \otimes (\Sigma -)^{\otimes(p-1)}
	    \simeq
	    \Sigma\Omega J\Sigma \otimes \Sigma^{p-2}(-)^{\otimes(p-1)}
	\]
	Composing with the counit $\Sigma\Omega \to I$ yields a map
	\[
	    \fib(J_{p-1}\Sigma \to J\Sigma)
	    \longrightarrow
	    J\Sigma \otimes \Sigma^{p-2}(-)^{\otimes(p-1)}
	\]
	Since we are working on suspensions, we may extract one further suspension from the factor $(-)^{\otimes(p-1)}$ and compose with the map $\Sigma J \to \Sigma$ coming from the James splitting. This gives a map
	\[
	    \fib(J_{p-1}\Sigma \to J\Sigma)
	    \longrightarrow
	    \Sigma I \otimes \Sigma^{p-2}(-)^{\otimes(p-1)}
	    \simeq
	    \Sigma^{p-1}(-)^{\otimes p}
	\]
	By \cite[thm. 4.5]{Richter2014}, the fiber of this map on odd spheres coincides with $BW$, and by \cite[thm. 4.1]{Richter2014}, the map factors multiplication by $p$ when composed on either side with the double suspension.
\end{proof}

We conclude the subsection by discussing two possible applications of the calculus version of the Gray sequence. These are independent of the main argument of the paper and are meant mainly as illustrations of the scope of the construction; they may therefore be skipped on a first reading.

\subsubsection{Further direction I: the Arone--Mahowald computations}

The calculus version of the Gray sequence suggests a possible route to recovering the two main results of Arone--Mahowald \cite{AroneMahowald}: first, that $\mathbb{D}_mI \simeq *$ on odd spheres unless $m=p^k$, and second, the calculation of the cohomology in the case $m=p^k$.

We begin with the first. We prove, by induction on $k$, that $\mathbb{D}_mI$ and $\mathbb{D}_mW$ are trivial on odd spheres whenever $p^{k-1}<m<p^k$. The base case $k=0$ is clear. Suppose the claim holds for $k-1\geq 0$. Then the argument in the proof of theorem \ref{graysequence}, with the convention on $m/p$ understood, yields more generally, for every $m$, a fiber sequence
\[
\mathbb{D}_mW(S^{2n-1})
\longrightarrow
\Sigma^{-3}\mathbb{D}_{m/p}I(S^{2np+1})
\longrightarrow
\Sigma^{-1}\mathbb{D}_{m/p}I(S^{2np-1})
\]
Since $p^{k-2}<m/p<p^{k-1}$, by induction we deduce that $\mathbb{D}_mW(S^{2n-1})\simeq *$. This implies that the map $\mathbb{D}_mI(S^{2n-1}) \to \Sigma^{-2}\mathbb{D}_mI(S^{2n+1})$ is an equivalence for all $n$. Hence $\mathbb{D}_mI(S^{2n-1}) \simeq \mathbb{D}_1\mathbb{D}_mI(S^{2n-1})\simeq *$, since $m>p^{k-1}\geq 1$.

For the second result, the only missing input is the following injectivity statement, which we assume here without independent proof: the map $S^n(k) \to \Sigma^{-2}S^{n+2}(k)$ is injective in $\mathbb{F}_p$-cohomology. Recall that
\[
\mathbb{D}_mI(S^n) \simeq (\partial_mI \otimes S^{nm})_{h\Sigma_m}
\]
where $\partial_mI$ is a spherical spectrum of dimension $1-m$. By the homotopy orbits spectral sequence, its $\mathbb{F}_p$-cohomology is given by
\[
H^{*-mn+m-1}(B\Sigma_m,\mathbb{F}_p^{1-m}\partial_mI\otimes \mathrm{sgn})
\]
Thus the map $S^n(k) \to \Sigma^{-2}S^{n+2}(k)$ corresponds, in $\mathbb{F}_p$-cohomology and up to suspension shift, to multiplication by the Euler class $e(\bar\sigma)$ of the complex standard representation $\bar\sigma$:
\[
e(\bar\sigma)\cup -:
\mathbb{F}_p^*(B\Sigma_m,V)
\longrightarrow
\mathbb{F}_p^{*+2(m-1)}(B\Sigma_m,V)
\]
where $m=p^k$ and $V:=\mathbb{F}_p^{1-m}\partial_mI\otimes \mathrm{sgn}$ is the Lie representation of $\Sigma_m$.

\begin{ques}
	For which representations $V$ is multiplication by $e(\bar\sigma)$ injective? In particular, is there a quick proof of this injectivity for the representation $V=\mathbb{F}_p^{1-m}\partial_mI\otimes \mathrm{sgn}$?
\end{ques}

Assuming this injectivity, we obtain
\[
\mathbb{F}_p^*S^{2n-1}(k)
\simeq
\mathbb{F}_p^*W^{2n-1}(k)
\oplus
\Sigma^{-2}\mathbb{F}_p^*S^{2n+1}(k)
\]
Iterating yields
\[
\mathbb{F}_p^*S^{2n-1}(k)
\simeq
\begin{cases}
    \Sigma^{2n-1}\mathbb{F}_p & \text{if } k=0 \\
    \textstyle\bigoplus_{i\geq n} \Sigma^{2(n-i)}\mathbb{F}_p^*W^{2i-1}(k) & \text{if } k>0
\end{cases}
\]

To describe the cohomology of $W^{2n-1}(k)$, we appeal to the calculus version of the Gray sequence (thm. \ref{graysequence}). We claim that the map
\[
\Sigma^{-3}S^{2np+1}(k-1)
\longrightarrow
\Sigma^{-1}S^{2np-1}(k-1)
\]
is null in $\mathbb{F}_p$-cohomology; we will justify this shortly. Using this, we deduce
\[
\mathbb{F}_p^*W^{2n-1}(k)
\simeq
\Sigma^{-2}\mathbb{F}_p^*S^{2np-1}(k-1)
\oplus
\Sigma^{-3}\mathbb{F}_p^*S^{2np+1}(k-1)
\]
In particular, defining
\[
C(n,k):=\Sigma^{1-2n+2k}\mathbb{F}_p^*S^{2n-1}(k)
\]
we obtain the recurrence relations
\begin{align*}
    C(n,0) &\simeq \mathbb{F}_p \\
    C(n,k)
    &\simeq \textstyle\bigoplus_{i\geq n}
    \bigl(\Sigma^{2(p-1)i}C(ip,k-1)
    \oplus
    \Sigma^{2(p-1)i+1}C(ip+1,k-1)\bigr) \\
    &\simeq \textstyle\bigoplus_{i\geq n}
    \bigl(C(ip,k-1)P^i
    \oplus
    C(ip+1,k-1)\beta P^i\bigr)
\end{align*}

Here the symbols $P^i$ and $\beta P^i$ are only suggestive, denoting the suspension shifts, to emphasize that this vector space is precisely the one appearing in the Arone--Mahowald computations; see proposition \ref{cohomS}. We do not claim anything about the $\mathcal{A}$-module structure.

Finally, it remains to show that the map
\[
\Sigma^{-3}S^{2np+1}(k-1)
\longrightarrow
\Sigma^{-1}S^{2np-1}(k-1)
\]
is null in $\mathbb{F}_p$-cohomology. Note that the composite
\[
S^{2np-1}(k-1)
\xlongrightarrow{E^2}
\Sigma^{-2}S^{2np+1}(k-1)
\longrightarrow
S^{2np-1}(k-1)
\]
factors multiplication by $p$. Hence, in $\mathbb{F}_p$-cohomology, the composite
\[
\mathbb{F}_p^*S^{2np-1}(k-1)
\longrightarrow
\mathbb{F}_p^*\Sigma^{-2}S^{2np+1}(k-1)
\xlongrightarrow{E^2}
\mathbb{F}_p^*S^{2np-1}(k-1)
\]
is zero. Since the right-hand map is injective, it follows that the left-hand map is trivial, as claimed.

\subsubsection{Further direction II: the Cohen--Moore--Neisendorfer exponent theorem}

We now turn to a second possible application of the calculus version of the Gray sequence. This one is motivated by the Cohen--Moore--Neisendorfer theorem \cite{CMN}, which asserts that, for $p>2$, the fiber $W^{2n-1}$ of the double suspension has exponent $p$. We do not prove this theorem here. Instead, we explain how the calculus version of the Gray sequence suggests a possible inductive approach to it, and we verify the resulting exponent statement for the first two approximations.

The starting point is that the Gray map factors multiplication by $p$. This produces additional fiber sequences involving the fiber of multiplication by $p$. If we write $\{p\}$ for the fiber of the multiplication-by-$p$ map, we obtain the following.

\begin{prop}
	There are fiber sequences
    \[
        \Omega W^{2np-1}_{k-1}
        \longrightarrow
        \Omega S^{2np-1}_{k-1}\{p\}
        \longrightarrow
        W^{2n-1}_k
    \]
	\[
        \Sigma^{-1}W^{2np-1}(k-1)
        \longrightarrow
        \Sigma^{-1}S^{2np-1}(k-1)\{p\}
        \longrightarrow
        W^{2n-1}(k)
    \]
\end{prop}

\begin{proof}
	The composite
    \[
        \Omega S^{2np-1}_{k-1}
        \xlongrightarrow{E^2}
        \Omega^3S^{2np+1}_{k-1}
        \longrightarrow
        \Omega S^{2np-1}_{k-1}
    \]
	factors multiplication by $p$ by theorem \ref{graysequence}. The claimed fiber sequences then follow from the standard description of the fiber of a composite, together with theorem \ref{graysequence}. The statement for the layers is proved in exactly the same way.
\end{proof}

The Cohen--Moore--Neisendorfer theorem says that, for $p>2$, multiplication by $p$ is null on $\Omega W^{2n-1}$ for every $n$. Equivalently, the natural map
\[
    \Omega W^{2n-1}\{p\}
    \longrightarrow
    \Omega W^{2n-1}
\]
admits a retraction. Since this looped statement is spherewise functorial, it also gives exponent-$p$ structures on the approximations $\Omega W_k^{2n-1}$ and on the layers $W^{2n-1}(k)$.

The fiber sequences above suggest trying to reverse this logic: first prove the exponent result for $\Omega W_k^{2n-1}$, and then deduce it for $\Omega W^{2n-1}$ by passing to the limit. Indeed, $\Omega W_k^{2n-1}$ is the fiber of a map
\[
    \Omega W^{2np-1}_{k-1}
    \xlongrightarrow{\gamma}
    \Omega S^{2np-1}_{k-1}\{p\}
\]
By induction on $k$, one may assume that the source has exponent $p$, while the target always has exponent $p$ for $p>2$. Thus $\Omega W_k^{2n-1}$ would also have exponent $p$, provided that the map $\gamma$ is compatible with the chosen exponent structures. Concretely, this means that one would want the diagram
% https://q.uiver.app/#q=WzAsNCxbMCwwLCJcXE9tZWdhIFdeezJucC0xfV97ay0xfSJdLFsxLDAsIlxcT21lZ2EgU157Mm5wLTF9X3trLTF9XFx7cFxcfSJdLFswLDEsIlxcT21lZ2EgV157Mm5wLTF9X3trLTF9XFx7cFxcfSJdLFsxLDEsIlxcT21lZ2EgU157Mm5wLTF9X3trLTF9XFx7cFxcfVxce3BcXH0iXSxbMCwyXSxbMCwxLCJcXGdhbW1hIl0sWzIsMywiXFxnYW1tYVxce3BcXH0iLDJdLFsxLDNdXQ==
\[\begin{tikzcd}
	{\Omega W^{2np-1}_{k-1}} & {\Omega S^{2np-1}_{k-1}\{p\}} \\
	{\Omega W^{2np-1}_{k-1}\{p\}} & {\Omega S^{2np-1}_{k-1}\{p\}\{p\}}
	\arrow["\gamma", from=1-1, to=1-2]
	\arrow[from=1-1, to=2-1]
	\arrow[from=1-2, to=2-2]
	\arrow["{\gamma\{p\}}"', from=2-1, to=2-2]
\end{tikzcd}\]
to commute, where the vertical maps are the chosen retractions.

\begin{ques}
	Can one prove that this diagram commutes in general, and thereby obtain a calculus-based proof of the Cohen--Moore--Neisendorfer theorem?
\end{ques}

Although we cannot prove the general case at present, the cases $k\leq 2$ provide evidence for the strategy. In these cases, the exponent statement can be recovered directly from the fiber sequences above, without appealing to the Cohen--Moore--Neisendorfer theorem. As a byproduct, we also recover the descriptions in examples \ref{cohomW1} and \ref{cohomW2}.

\begin{ex}
	The case $k=1$ follows immediately from the fiber sequence above. Indeed, we have $W^{2n-1}_0 \simeq *$, and hence
	\[
	    W^{2n-1}(1)
	    \simeq
	    \Sigma^{-1}S^{2np-1}(0)\{p\}
	    \simeq
	    \Sigma^{-1}\mathbb{S}^{2np-1}\{p\}
	    \simeq
	    \mathbb{S}^{2np-3}/p
	\]
	This recovers example \ref{cohomW1}. Since $\mathbb{S}/p$ has exponent $p$ for $p>2$, this recovers the exponent statement in the case $k=1$.
\end{ex}

\begin{ex}
We next consider the case $k=2$. The fiber sequence above gives
\[
    \Sigma^{-1}W^{2np-1}(1) \simeq \mathbb{S}^{2np^2-4}/p
    \longrightarrow
    \Sigma^{-1}S^{2np-1}(1)\{p\}
    \longrightarrow
    W^{2n-1}(2)
\]
whose existence was asserted in example \ref{cohomW2}. We analyze the left-hand map. First note that $S^{2np-1}(1)$ is $(2np^2-3)$-connected and has first nontrivial homotopy group $\mathbb{F}_p$. Hence there is a map detecting this first nontrivial homotopy class:
\[
    \mathbb{S}^{2np^2-3}/p \longrightarrow S^{2np-1}(1)
\]
Applying $\{p\}$ and desuspending gives
\[
    (\mathbb{S}^{2np^2-4}/p)\{p\}
    \longrightarrow
    \Sigma^{-1}S^{2np-1}(1)\{p\}
\]
This map induces a surjection in cohomology detecting the classes
\[
    1\otimes\theta_{np},
    \quad
    1\otimes\beta\theta_{np},
    \quad
    \beta\otimes\theta_{np},
    \quad
    \beta\otimes\beta\theta_{np}
\]
in the notation of example \ref{cohomS1}. Its cofiber is therefore $(2np^2+2p-7)$-connected, and hence admits no nontrivial maps from $\mathbb{S}^{2np^2-4}/p$ when $p>2$. It follows that the first map in the fiber sequence factors as
\[
    \mathbb{S}^{2np^2-4}/p
    \longrightarrow
    (\mathbb{S}^{2np^2-4}/p)\{p\}
    \longrightarrow
    \Sigma^{-1}S^{2np-1}(1)\{p\}
\]
The second map is automatically compatible with the exponent structure, so it remains to analyze the first. This is simply the inclusion of the top cell, or equivalently the retraction exhibiting the exponent. Thus it suffices to show that the diagram
% https://q.uiver.app/#q=WzAsNCxbMCwwLCJcXG1hdGhiYntTfS9wIl0sWzIsMCwiXFxtYXRoYmJ7U30vcCBcXHtwXFx9Il0sWzAsMSwiXFxtYXRoYmJ7U30vcCBcXHtwXFx9Il0sWzIsMSwiXFxtYXRoYmJ7U30vcCBcXHtwXFx9IFxce3BcXH0iXSxbMCwyLCJlX3tcXG1hdGhiYntTfS9wfSIsMl0sWzAsMSwiZV97XFxtYXRoYmJ7U30vcH0iXSxbMiwzLCJlX3tcXG1hdGhiYntTfS9wfVxce3BcXH0iLDJdLFsxLDMsIlxcbWF0aGJie1N9L3AgXFxvdGltZXMgZV97XFxtYXRoYmJ7U31cXHtwXFx9fSJdXQ==
\[
\begin{tikzcd}
	{\mathbb{S}/p} && {\mathbb{S}/p \{p\}} \\
	{\mathbb{S}/p \{p\}} && {\mathbb{S}/p \{p\} \{p\}}
	\arrow["{e_{\mathbb{S}/p}}", from=1-1, to=1-3]
	\arrow["{e_{\mathbb{S}/p}}"', from=1-1, to=2-1]
	\arrow["{\mathbb{S}/p \otimes e_{\mathbb{S}\{p\}}}", from=1-3, to=2-3]
	\arrow["{e_{\mathbb{S}/p}\{p\}}"', from=2-1, to=2-3]
\end{tikzcd}
\]
commutes, where $e$ denotes the chosen retraction. Now a retraction of $X\{p\} \to X$ is equivalent to a section of $X \to X/p$, that is, to a homotopy $\mathbb{S}/p$-module structure on $X$. Under this correspondence, the commutativity of the diagram above is precisely the associativity condition for the corresponding $\mathbb{S}/p$-module structure. For $X=\mathbb{S}/p$, this reduces to the existence of a homotopy-associative unital multiplication on $\mathbb{S}/p$, which holds for $p>3$ \cite{Moorespectraassoc} \cite[ex. 1.2]{Moorespectraassoc2}.

Therefore $W^{2n-1}(2)$ has exponent $p$ for $p>3$, without appealing to the Cohen--Moore--Neisendorfer theorem. The same argument shows that $\Omega W^{2n-1}_2$ has exponent $p$, since it is the fiber of
\[
    \mathbb{S}^{2np^2-4}/p
    \longrightarrow
    (\mathbb{S}^{2np^2-4}/p)\{p\}
    \longrightarrow
    \Sigma^{-1}S^{2np-1}(1)\{p\}
    \longrightarrow
    \Omega S^{2np-1}_1\{p\}
\]
\end{ex}

We do not know how to extend this argument to $k\geq 3$. Nevertheless, the same strategy suggests a possible route to proving that $W^{2n-1}(k)$ and $\Omega W^{2n-1}_k$ have exponent $p$ for all $k$. Such an argument would provide a more direct, calculus-based approach to the Cohen--Moore--Neisendorfer theorem.

\subsection{Early convergence}
\label{earlyconv}

In this subsection, we assume that either $p=2$, or $p>2$ and $n$ is odd. Under these assumptions, propositions \ref{genfreudenthal} and \ref{genfreudenthalW} show that the comparison maps $S^n \to S^n_k$ and $W^n \to W^n_k$ induce isomorphisms on $\pi_*$ for $*<c(n,k+1)$ and surjections in degree $*=c(n,k+1)$. A natural question is whether these maps are also injective in this boundary degree, a phenomenon that we call \emph{early convergence}. This leads naturally to certain distinguished homotopy classes, which we call \emph{Whitehead elements}, and which govern whether early convergence occurs. These elements are closely connected to classical problems such as the Hopf invariant one and Kervaire invariant one problems. We conclude the subsection by proving that $W^1_k$ converges early, a fact needed later in the computation of the homotopy groups of $S^3$.

\begin{definition}
	The approximation sphere $S^n_k$ \emph{converges early} if the comparison map $S^n \to S^n_k$ is injective on $\pi_{c(n,k+1)}$.
\end{definition}

To introduce the Whitehead elements, recall from proposition \ref{cohomS} that $\pi_{c(n,k)}S^n(k)$ is cyclic. One may visualize the Goodwillie spectral sequence as a staircase: the Whitehead elements correspond to the corner classes, namely the first nontrivial classes in each row.

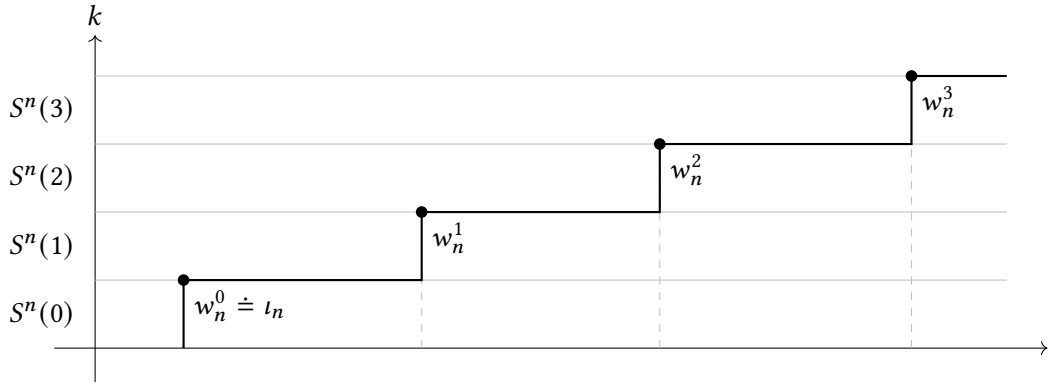
\begin{figure}[H]
\centering
\begin{tikzpicture}[scale=0.9]

    % axes
    \draw[->] (-0.6,0) -- (14.0,0);
    \draw[->] (0,-0.5) -- (0,4.6) node[above] {$k$};

    % horizontal rows and corner classes
    \foreach \k/\y/\x/\lab in {
        0/1.0/1.3/{w_n^0 \doteq \iota_n},
        1/2.0/4.8/{w_n^1},
        2/3.0/8.3/{w_n^2},
        3/4.0/12.0/{w_n^3}
    }{
        % row
        \draw[gray!50] (0,\y) -- (13.4,\y);

        % first nontrivial class
        \filldraw[black] (\x,\y) circle (2.2pt);

        % label for corner class
        \node[below right] at (\x,\y) {$\lab$};

        % vertical guide
        \draw[dashed, gray!50] (\x,0) -- (\x,\y);
    }

    % staircase line
    \draw[thick]
        (1.3,0) --
        (1.3,1.0) --
        (4.8,1.0) --
        (4.8,2.0) --
        (8.3,2.0) --
        (8.3,3.0) --
        (12.0,3.0) --
        (12.0,4.0) --
        (13.4,4.0);

    % row labels between the horizontal lines
    \node[left] at (-0.15,0.5) {$S^n(0)$};
    \node[left] at (-0.15,1.5) {$S^n(1)$};
    \node[left] at (-0.15,2.5) {$S^n(2)$};
    \node[left] at (-0.15,3.5) {$S^n(3)$};

\end{tikzpicture}
\caption{Schematic picture of the Goodwillie spectral sequence. The Whitehead elements are represented by the corner classes, namely the first nontrivial classes in each row.}
\end{figure}

\begin{definition}
	A \emph{$k$-Whitehead element} of $S^n$ is the image of a generator under the map
	\[
	\pi_{c(n,k)}S^n(k) \longrightarrow \pi_{c(n,k)}S^n_k \cong \pi_{c(n,k)}S^n
	\]
	Any $k$-Whitehead element will be denoted by $w_n^k$.
\end{definition}

\begin{rem}
	This does not determine a unique element: any two $k$-Whitehead elements of $S^n$ differ by multiplication by a unit in $\mathbb{F}_p^\times$. Accordingly, from now on all identities involving Whitehead elements are understood up to multiplication by a unit in $\mathbb{F}_p^\times$. We write $\doteq$ to emphasize this. In particular, for $k=0$ one has $w_n^0 \doteq \iota_n$, where $\iota_n \in \pi_nS^n$ denotes the identity class.

	One could instead try to rigidify the definition by specifying a preferred generator through the cohomological description of $S^n(k)$, for instance by choosing the class corresponding to $P^{p^{k-1}n}\cdots P^n$. However, this would still require fixing the relevant cohomological identifications. In retrospect, a cleaner alternative would have been to work with the associated \emph{Whitehead group}, namely the image of $\pi_{c(n,k)}S^n(k)$ in $\pi_{c(n,k)}S^n$, which avoids this ambiguity altogether.
\end{rem}

The following lemma collects the basic properties of these elements.

\begin{lem}
	A $k$-Whitehead element $w_n^k$ satisfies:
	\begin{enumerate}
	\item (Detection) $S^n_k$ converges early if and only if $w_n^{k+1} = 0$.
	\item (Instability) $\Sigma w_n^k = 0$ when $p=2$, and $\Sigma^2 w_n^k = 0$ when $p>2$ and $n$ is odd, for $k>0$.
	\item (Order) $p \cdot w_n^k = 0$ for $k>0$.
	\end{enumerate}
\end{lem}

\begin{proof}
	We prove each statement in turn.
	\begin{enumerate}
		\item Consider the fiber sequence
		\[
		\Omega^\infty S^n(k+1) \longrightarrow S^n_{k+1} \longrightarrow S^n_k
		\]
		On $\pi_{c(n,k+1)}$ this gives an exact sequence
		\[
		\pi_{c(n,k+1)}S^n(k+1)
		\longrightarrow
		\pi_{c(n,k+1)}S^n_{k+1}
		\longrightarrow
		\pi_{c(n,k+1)}S^n_k
		\longrightarrow
		0
		\]
		Since $S^n \to S^n_{k+1}$ is an isomorphism on $\pi_{c(n,k+1)}$, we may identify the middle term with $\pi_{c(n,k+1)}S^n$. By definition, $S^n_k$ converges early if and only if the middle map is injective. This happens precisely when the first map is trivial. Since $\pi_{c(n,k+1)}S^n(k+1)$ is cyclic and its image is generated by a $(k+1)$-Whitehead element, this is equivalent to $w_n^{k+1}=0$.

		\item This follows from the commutative diagram
		% https://q.uiver.app/#q=WzAsNCxbMCwwLCJcXHBpX3tjKG4sIGspfSBTXm4oaykiXSxbMSwwLCJcXHBpX3tjKG4sIGspICsgMn1TXntuKzJ9KGspIFxcY29uZyAqIl0sWzAsMSwiXFxwaV97YyhuLCBrKX0gU15uX2siXSxbMSwxLCJcXHBpX3tjKG4sIGspICsgMn0gU157bisyfV9rIl0sWzAsMV0sWzAsMl0sWzIsMywiRV4yIiwyXSxbMSwzXV0=
		\[
		\begin{tikzcd}
			{\pi_{c(n,k)}S^n(k)} & {\pi_{c(n,k)+2}S^{n+2}(k)\cong *} \\
			{\pi_{c(n,k)}S^n_k} & {\pi_{c(n,k)+2}S^{n+2}_k}
			\arrow[from=1-1, to=1-2]
			\arrow[from=1-1, to=2-1]
			\arrow[from=1-2, to=2-2]
			\arrow["{E^2}"', from=2-1, to=2-2]
		\end{tikzcd}
		\]
		The top right group vanishes because the connectivity of $S^{n+2}(k)$ is
		\[
		c(n+2,k)= c(n,k)+2p^k>c(n,k)+2
		\]
		for $k>0$. This proves the double suspension statement, and the case $p=2$ follows similarly.

		\item For $k>0$ one has $\pi_{c(n,k)}S^n(k)\cong \mathbb{F}_p$, so every $k$-Whitehead element has order $p$. \qedhere
	\end{enumerate}
\end{proof}

We now turn to the analogous constructions for the fiber of the double suspension $W^n$.

\begin{definition}
	The approximation $W^n_k$ \emph{converges early} if the comparison map $W^n \to W^n_k$ is injective on $\pi_{c(n,k+1)}$.
\end{definition}

\begin{definition}
	A \emph{$k$-Whitehead element} of $W^n$ is the image of a generator under the map
	\[
	\pi_{c(n,k)}W^n(k) \longrightarrow \pi_{c(n,k)}W^n_k \cong \pi_{c(n,k)}W^n
	\]
	Any $k$-Whitehead element will be denoted by $\hat w_n^k$.
\end{definition}

As before, one proves the following.

\begin{lem}
	A $k$-Whitehead element $\hat w_n^k$ satisfies:
	\begin{enumerate}
		\item (Detection) $W^n_k$ converges early if and only if $\hat w_n^{k+1}=0$.
		\item (Order) $p\cdot \hat w_n^k=0$.
	\end{enumerate}
\end{lem}

The Whitehead elements for $S^n$ and for $W^n$ are related as follows.

\begin{lem}
	The canonical map $W^n \to S^n$ sends $\hat w_n^k$ to $w_n^k$, and the connecting map in the Gray sequence $\Omega^2 S^{(n+1)p-1} \to W^n$ sends $w^{k-1}_{(n+1)p-1}$ to $\hat w_n^k$. Both statements are understood up to multiplication by a unit.
\end{lem}

\begin{proof}
	The Whitehead elements are detected in the first nontrivial cohomology groups of $S^n(k)$ and $W^n(k)$, and both maps induce isomorphisms on these groups.
\end{proof}

We can now deduce the early convergence of $W^1_k$, a fact that we will need later for the computation of the homotopy groups of $S^3$.

\begin{prop}
	\label{earlyconvW}
	For $k>0$, the approximation $W^1_k$ converges early.
\end{prop}

\begin{proof}
	By the detection property, $W^1_k$ converges early if and only if $\hat w_1^{k+1}=0$. By the previous lemma, the element $\hat w_1^{k+1}$ is the image of $w^k_{2p-1}$ under the connecting map in the Gray sequence. It therefore suffices to show that $w^k_{2p-1}$ lies in the image of the Gray map
	\[
	\Omega^4 S^{2p+1} \longrightarrow \Omega^2 S^{2p-1}
	\]
	By the instability property, the element $w^k_{2p-1}$ double-suspends trivially for $k>0$. Toda's lemma \cite[lem. 2.4]{Toda3} asserts that every such element lies in the image of the Gray map. This proves that $\hat w_1^{k+1}=0$, and hence that $W^1_k$ converges early.
\end{proof}

The relationship between $w_n^k$ and $\hat w_n^k$ yields a recursive formula. If we denote by $P$ the composite
\[
\Omega^2 S^{(n+1)p-1} \to W^n \to S^n
\]
of the Gray connecting map with the canonical projection, we obtain the following.

\begin{lem}
	The Whitehead elements satisfy the recursive formula
	\[
	w_n^k \doteq P_* w^{k-1}_{(n+1)p-1} \doteq P^{(k)}_*\iota_{(n+1)p^k-1}
	\]
\end{lem}

\begin{rem}
	This recursive formula could instead be taken as the definition of the Whitehead elements, thereby avoiding the ambiguity in the choice of generators above.
\end{rem}

Note that for $p=2$, the map $P$ coincides with the $P$ map in the EHP sequence, which explains the choice of notation. The reader may verify this from our construction of the Gray sequence by iterating the EHP sequence. Moreover, it follows from the definition and the properties of the Gray map that $\Sigma^2 \circ P \simeq *$ and $P \circ p \simeq *$, thereby recovering the instability and order properties of the Whitehead elements described above. Finally, the iterate $P^{(k)}$ is a map $\Omega^{2k} S^{(n+1)p^k-1} \to S^n$, providing a topological explanation for the recurring number $c(n,k)$.

Finally, let us illustrate the recursive description with some concrete cases.

\begin{ex}
	The first Whitehead elements are:
	\begin{enumerate}
		\item $w_n^0 \doteq \iota_n \neq 0$.
		\item $w_1^k=0$ for $k>0$, since $\pi_*S^1$ is torsion-free while $w_1^k$ is torsion. Note that $w_1^1\in \pi_1S^1$ for $p=2$.
		\item $w_n^1 \doteq [\iota_n,\iota_n]$ for $p=2$. This follows from the formula $P_*\alpha=[\iota_n,\Sigma^{-n-1}\alpha]$, valid whenever the desuspension exists \cite[prop. 3.6.1]{Behrens2010}. By the Hopf invariant one theorem \cite[thm. A]{AdamsAtiyah}, this vanishes if and only if $n=1,3,7$.
		\item $w_n^1$ for $p > 2$ coincides up to a unit with the mod-$p$ Whitehead element of \cite{GrayFam,WhiteheadModp}, which is nontrivial except for $n=1$ by the Hopf invariant one theorem \cite[thm. D]{AdamsAtiyah}. This is not a Whitehead product: Whitehead products vanish on H-spaces, yet $w_3^1 \neq 0$.
		\item $w_n^k = 0$ for $p = 2$ when $k>0$ and $n=1,3,7$, since $P$ vanishes on H-spaces.
		\item $w_n^k$ for $p = 2$ is denoted by $1[j_1, \dots, j_k]$ in \cite[6.5]{Behrens2010}. In some cases, this element detects $[\iota_{j_k}, \dots, \Sigma^{-*}[\iota_{j_1}, \iota_{j_1}]]$ \cite[3.6]{Behrens2010}, but this is not always true. For instance, $w_2^2 = 1[5,2]$ is expected to detect $[\iota_2, \Sigma^{-3}[\iota_5, \iota_5]]$. However, $[\iota_5, \iota_5] \in \pi_9 S^5$ does not desuspend three times; see the metastable computations in subsection \ref{metastable}. In the range covered by Behrens' calculations, we observe that $w_n^k = 0$ for $k > 0$ if and only if $n = 1,3$. We also note that the element $w_4^2 = 1[9,4]$ is missing from his charts \cite[table 6.5.8, row 11]{Behrens2010}.
		\item $w_n^2 \doteq P_*[\iota_{2n+1}, \iota_{2n+1}]$ for $p = 2$. Since $P \circ 2 \simeq *$, it follows that $w_n^2 = 0$ whenever $[\iota_{2n+1}, \iota_{2n+1}]$ is divisible by $2$. This occurs when the Kervaire invariant one class $\theta_j$ exists in $\pi_{2n}\mathbb{S}$ and has order $2$; see \cite[5.2]{Kervaireexp}. This holds for $2n = 2^{j+1} - 2$ with $j \leq 6$, by the solution of the Kervaire invariant one problem \cite{Kervairesol,Kervairelast}. In particular, this implies that $w^2_{2^i - 1} = 0$ for $i \leq 6$, giving the new vanishing results $w^2_{15} = w^2_{31} = w^2_{63} = 0$.
	\end{enumerate}
\end{ex}

These examples lead naturally to the following question.

\begin{ques}
	When is $w_n^k$ nonzero?
\end{ques}

For instance, when $p=2$ one may ask whether $w_n^k \neq 0$ whenever $n$ is not of the form $2^i-1$. More generally, for fixed $k$, are infinitely many of the classes $w_n^k$ nontrivial? Does the number of vanishing classes $w_n^k$ increase with $k$? For $p>2$ and $n>1$, do any of these classes vanish? For which values of $n$ does one obtain an infinite family? We leave these questions open for further study.

\section{Model for the approximation circles}

The goal of this section is to give an explicit description of the Goodwillie approximations to $S^1$. This model will be the main input for the computations carried out later, but it is also of independent interest: it expresses $\Omega S^1_k$ in terms of well-studied spectra, namely the symmetric product spectra.

In the first subsection, we recall the symmetric product spectra $\mathbb{Sp}^m$ and collect some of their basic properties, including their $\mathbb{F}_p$-cohomology, their $K(i)$-cohomology, and their low-degree homotopy groups. These facts will be used repeatedly in what follows.

In the second subsection, we prove the central equivalence
\[
\Omega S^1_k \simeq \mathbb{Z} \times \Omega^{\infty+2k}\tau_{>0}\mathbb{Sp}^{p^k}
\]
The proof relies on the solution of the Whitehead conjecture by Behrens \cite{Behrens2011} and Kuhn \cite{Kuhn2014}. The result reduces the study of $\Omega S^1_k$ to stable homotopy theory and provides a concrete entry point for the computational applications of the next section.

\subsection{Symmetric products}

The model for $S^1_k$ will be expressed in terms of certain well-studied spectra called \textit{symmetric product spectra}. The goal of this subsection is to give a brief exposition of them.

The symmetric product spectra are the first derivatives of functors that go by the same name. Classically, the $m^\mathrm{th}$ symmetric product functor is defined on topological spaces by $Sp^m(X) := X^m_{\Sigma_m}$, where the product is the cartesian product and the subscript denotes strict orbits under the action of the symmetric group $\Sigma_m$.

We prefer to adopt a homotopy-invariant definition from the start. By Elmendorf's theorem \cite{Gspaces}, every $G$-space $X$ is equivalent to a functor $\mathcal{O}(G)^{\mathrm{op}} \to \mathcal{S}$ from the orbit category, sending $G/H$ to $X^H$, and strict orbits correspond to colimits of such functors. This motivates the following definition.

\begin{definition}
	The $m^\mathrm{th}$ \textit{symmetric product} is the functor $Sp^m: \mathcal{S}_* \to \mathcal{S}_*$ defined by
	\[
	Sp^mX := \mathrm{colim}_{\Sigma_m/H \in \mathcal{O}^{\mathrm{op}}(\Sigma_m)}(X^m)^H \simeq \mathrm{colim}_{\Sigma_m/H}X^{m/H}
	\]
\end{definition}

To illustrate how these colimits can be manipulated, let us work out explicitly the case $m=2$ and examine how it behaves on spheres.

\begin{ex}
	The \textit{symmetric square} is given by the functor
	\[
	Sp^2X = \mathrm{colim}_{C_2/H}X^{2/H} = \mathrm{colim}(X \to X^{2}\circlearrowleft C_2)
	\]
	where the map is the diagonal. This functor fits into a cofiber sequence
	\[
	I \longrightarrow Sp^2 \longrightarrow Sp^{\wedge 2}
	\]
	where we define $Sp^{\wedge 2}X := \mathrm{colim}(X \to X^{\wedge 2}\circlearrowleft C_2)$. Indeed, this follows by taking vertical cofibers in the diagram below.
	% https://q.uiver.app/#q=WzAsNixbMSwwLCJcXG1hdGhybXtjb2xpbX0oKiJdLFsyLDAsIlggXFx2ZWUgWCBcXGNpcmNsZWFycm93bGVmdCkiXSxbMSwxLCJcXG1hdGhybXtjb2xpbX0oWCJdLFsyLDEsIlheMiBcXGNpcmNsZWFycm93bGVmdCkiXSxbMCwwLCJJICJdLFswLDEsIlNwXjIiXSxbMCwxXSxbMCwyLCIiLDIseyJvZmZzZXQiOi01fV0sWzIsM10sWzEsMywiIiwwLHsib2Zmc2V0IjozfV0sWzQsMCwiIiwyLHsibGV2ZWwiOjIsInN0eWxlIjp7ImhlYWQiOnsibmFtZSI6Im5vbmUifX19XSxbNSwyLCIiLDIseyJsZXZlbCI6Miwic3R5bGUiOnsiaGVhZCI6eyJuYW1lIjoibm9uZSJ9fX1dLFs0LDVdXQ==
\[\begin{tikzcd}
	{I } & {\mathrm{colim}(*} & {X \vee X \circlearrowleft C_2)} \\
	{Sp^2} & {\mathrm{colim}(X} & {X^2 \circlearrowleft C_2)}
	\arrow[equals, from=1-1, to=1-2]
	\arrow[from=1-1, to=2-1]
	\arrow[from=1-2, to=1-3]
	\arrow[shift left=5, from=1-2, to=2-2]
	\arrow[shift right=3, from=1-3, to=2-3]
	\arrow[equals, from=2-1, to=2-2]
	\arrow[from=2-2, to=2-3]
\end{tikzcd}\]
	Now observe that $Sp^{\wedge 2}S^n \simeq \Sigma^n(RP^{n-1})^\diamond$. Here $\diamond$ denotes the unreduced suspension, so for $n>0$ this simplifies to $\Sigma^{n+1}RP^{n-1}$. To see this, first note that
	\[
	Sp^{\wedge 2}S^n
	=
	\mathrm{colim}(S^n \to S^{2n}\circlearrowleft C_2)
	\simeq
	\Sigma^n \mathrm{colim}(S^0 \to S^{n\rho}\circlearrowleft C_2)
	\]
	Now note that $S^{n\rho} \simeq S(n\rho)^\diamond$. Since the unreduced suspension $\diamond:\mathcal{S}\to\mathcal{S}_*$ preserves contractible colimits, and since the orbit category is contractible because it has a terminal object, it follows that
	\[
	Sp^{\wedge 2}S^n
	\simeq
	\Sigma^n \mathrm{colim}^{unpt}(\emptyset \to S(n\rho)\circlearrowleft C_2)^\diamond
	\simeq
	\Sigma^n S(n\rho)_{hC_2}^\diamond
	=
	\Sigma^n(RP^{n-1})^\diamond
	\]
	Putting everything together, we conclude that the symmetric square of a sphere fits into the cofiber sequence
	\[
	S^n \longrightarrow Sp^2S^n \longrightarrow \Sigma^n(RP^{n-1})^\diamond
	\]
\end{ex}

There are two classical results concerning symmetric products.

\begin{prop}
	The symmetric product functors satisfy:
	\begin{enumerate}
		\item $Sp^mS^1 \simeq S^1$.
		\item $Sp^\infty$ is linear.
	\end{enumerate}
\end{prop}

\begin{proof}
	The first statement is a consequence of the fundamental theorem of algebra; see, for instance, \cite[lem. 7.10]{AroneDwyer}. The second is the content of the Dold--Thom theorem \cite{doldthom}.
\end{proof}

Combining these results, we deduce that $Sp^\infty \simeq \Omega^\infty(\mathbb{Z} \otimes -)$. In particular, we obtain the following filtration of the Hurewicz map:
\[
I = Sp^1
\longrightarrow \cdots
\longrightarrow Sp^{m-1}
\longrightarrow Sp^m
\longrightarrow \cdots
\longrightarrow Sp^\infty \simeq \Omega^\infty(\mathbb{Z} \otimes -)
\]

We are now ready to define the symmetric product spectra as the first derivatives of the symmetric product functors.

\begin{definition}
	\label{symmprod}
	The $m^\mathrm{th}$ \textit{symmetric product spectrum} is defined by
	\[
	\mathbb{Sp}^m := \partial_1 Sp^m
	\]
\end{definition}

\begin{lem}
	There is an equivalence
	\[
	\mathbb{Sp}^m \simeq \mathrm{colim}_{\Sigma_m/H} \mathbb{S}^{\oplus m/H}
	\]
\end{lem}

\begin{proof}
	By the definition of $Sp^m$, and since $\partial_1$ preserves colimits of reduced functors, we have
	\[
	\mathbb{Sp}^m
	= \partial_1 Sp^m
	= \partial_1 \mathrm{colim}_{\Sigma_m/H}(-)^{m/H}
	\simeq \mathrm{colim}_{\Sigma_m/H} \partial_1\bigl((-)^{m/H}\bigr)
	\]
	On the other hand, for every finite set $A$ one has
	\[
	\partial_1\bigl((-)^A\bigr) \simeq \mathbb{S}^{\oplus A}
	\]
	Hence the claim follows.
\end{proof}

\begin{rem}
	Here and below we use that $\partial_1$ preserves colimits of reduced functors. Indeed, for a reduced functor $F$ one has
	\[
	\partial_1F \simeq \mathrm{colim}_n \Sigma^{-n}\Sigma^\infty F(S^n)
	\]
	so $\partial_1$ is given by a colimit-preserving construction.
\end{rem}

An alternative description, due to Lesh \cite{ktheorymult} but not needed here, is that $\mathbb{Sp}^m$ can be identified with the algebraic $K$-theory of the category of multisets with multiplicity at most $m$.

\begin{ex}
	\label{reducedtransfer}
	For $m=2$, the previous discussion gives the cofiber sequence
	\[
	\mathbb{S} \longrightarrow \mathbb{Sp}^2 \longrightarrow \Sigma\mathbb{RP}^\infty
	\]
	This follows by applying $\partial_1$ to the cofiber sequence $I \to Sp^2 \to Sp^{\wedge 2}$ and using the fact that $Sp^{\wedge 2}S^n \simeq \Sigma^{n+1}RP^{n-1}$ for $n>0$.

	Moreover, the map $\mathbb{RP}^\infty \to \mathbb{S}$ obtained by rotating this sequence is the classical Kahn--Priddy map \cite{KP,KPWelcher}.
\end{ex}

From now on, we will no longer use the functors $Sp^m$, but only their first derivatives $\mathbb{Sp}^m$. Before leaving the functors behind, however, it is natural to ask the following question.

\begin{ques}
	What do the higher Goodwillie approximations of $Sp^m$ look like?
\end{ques}

By taking $\partial_1$ of the filtration above, we obtain the sequence of symmetric product spectra
\[
\mathbb{S} \simeq \mathbb{Sp}^1
\longrightarrow \cdots
\longrightarrow \mathbb{Sp}^{m-1}
\longrightarrow \mathbb{Sp}^m
\longrightarrow \cdots
\longrightarrow \mathbb{Sp}^\infty \simeq \mathbb{Z}
\]
This sequence is an equivalence on $\pi_0\cong\mathbb{Z}$. The original \emph{Whitehead conjecture} asks whether these maps are null on $\pi_*$ for $*>0$, a question answered affirmatively by Kuhn \cite{WhiteheadOriginal}.

More generally, we can consider the quotients of this filtration. First, note that at the functor level
\[
Sp^m/Sp^{m-1} \simeq Sp^{\wedge m} := \mathrm{colim}_{\Sigma_m/H}(-)^{\wedge m/H}
\]
Applying $\partial_1$, we obtain
\[
\mathbb{Sp}^m/\mathbb{Sp}^{m-1}
\simeq \partial_1Sp^{\wedge m}
\simeq \mathrm{colim}_{\Sigma_m/H}
\begin{cases}
\mathbb{S} & \text{if } m/H = 1 \\
* & \text{if } m/H \neq 1
\end{cases}
\simeq \Sigma^{\infty+\diamond}\mathrm{colim}_{\Sigma_m/H}^{unpt}
\begin{cases}
\emptyset & \text{if } m/H = 1 \\
* & \text{if } m/H \neq 1
\end{cases}
\]
This recovers Lesh's result \cite{nontransitive}:
\[
\mathbb{Sp}^m/\mathbb{Sp}^{m-1} \simeq \Sigma^{\infty+\diamond} B_{nt}\Sigma_m
\]
where $B_{nt}\Sigma_m$ is the classifying space of the family of non-transitive subgroups of $\Sigma_m$. It also recovers the calculation for $m=2$ from example \ref{reducedtransfer}, since $B_{nt}\Sigma_2 = B\Sigma_2 = RP^\infty$.

These quotients behave in a way reminiscent of the Goodwillie tower of the identity after $p$-localization: they are trivial unless $m=p^k$. Indeed, after $p$-localization, a standard transfer argument \cite[sec. 7]{nontransitive} allows us to compute the above colimit using only those orbits $\Sigma_m/H$ for which $H$ is a $p$-subgroup; if $m$ is not a power of $p$, none of these subgroups can act transitively on $m$ letters, and the colimit becomes contractible. Consequently, we only need to focus on the spectra $\mathbb{Sp}^{p^k}$.

The spectra $\mathbb{Sp}^{p^k}$ are well studied. Their $\mathbb{F}_p$-cohomology and $K(i)$-cohomology are completely understood.

To describe their $\mathbb{F}_p$-cohomology, recall from the previous section, before proposition \ref{cohomS}, the $\mathcal{A}$-module $(\mathcal{A}/\mathcal{A}\beta)_{\geq k}$ defined by the length filtration on $\mathcal{A}/\mathcal{A}\beta$. The following computation goes back to Nakaoka \cite{Nakaoka}, who showed that the map $\mathbb{Sp}^{p^k} \to \mathbb{Z}$ induces a surjection
\[
\mathcal{A}/\mathcal{A}\beta \cong \mathbb{F}_p^*\mathbb{Z} \twoheadrightarrow \mathbb{F}_p^*\mathbb{Sp}^{p^k}
\]
with kernel $(\mathcal{A}/\mathcal{A}\beta)_{\geq k+1}$. A more direct reference for the formulation below is \cite[thm. 4.1]{MitchellPriddy}.

\begin{prop}
	\label{symmprodfp}
	The $\mathbb{F}_p$-cohomology of the symmetric product spectra is given by
	\begin{align*}
		\mathbb{F}_p^*\mathbb{Sp}^{p^k}
		&\cong (\mathcal{A}/\mathcal{A}\beta)_{\leq k}
		:= (\mathcal{A}/\mathcal{A}\beta)/(\mathcal{A}/\mathcal{A}\beta)_{\geq k+1} \\
		&= \mathbb{F}_p\bigl\langle \theta^I \in \mathcal{A}/\mathcal{A}\beta \mid I \text{ admissible of length } \leq k \bigr\rangle
	\end{align*}
\end{prop}

\begin{rem}
	It is worth noting that the previous calculation shows that
	\[
	\mathbb{F}_p^*\tau_{>0}\mathbb{Sp}^{p^k}
	\cong
	\Sigma^{-1}(\mathcal{A}/\mathcal{A}\beta)_{\geq k+1}
	\]
	In this sense, both from this calculation and from the next result, as well as from the model itself, the connected symmetric product spectra appear to be more natural objects than the non-connected ones.
\end{rem}

From this cohomological information, one can run an Adams spectral sequence to compute the homotopy groups. For fixed $k$, computer calculations can produce charts over a large range; examples will appear in later sections. Here we record the general pattern of the first few nontrivial groups, which will suffice for our purposes.

\begin{lem}
	\label{homotopysp}
    The first three nontrivial homotopy groups of the symmetric product spectra are given by
    \[
    \pi_*\mathbb{Sp}^{p^k} \cong
    \begin{cases}
        \mathbb{Z} & * = 0 \\
        \mathbb{F}_p & * = 2(p^{k+1}-1)-1 \\
		\mathbb{F}_p & * = N(p^{k+1}-1)-1 \\
        0 & * \leq N(p^{k+1}-1)-1, \text{otherwise}
    \end{cases}
    \]
	where $N=3$ if $p=2$ and $N=4$ if $p>2$.
\end{lem}

\begin{proof}
	We already know that there is a copy of $\mathbb{Z}$ in degree $0$, and that
	\[
	\mathbb{F}_p^*\tau_{>0}\mathbb{Sp}^{p^k}
	\cong
	\Sigma^{-1}(\mathcal{A}/\mathcal{A}\beta)_{\geq k+1}
	\]
	The first nontrivial element in this cohomology is $M_{k+1} := P^{p^k}\cdots P^pP^1$ in degree $d := 2(p^{k+1}-1)$. This has a nontrivial Bockstein, namely the element $\beta P^{p^k}\cdots P^pP^1$, and hence gives a copy of $\mathbb{F}_p$ in degree $d-1$, because of the shift. Next, consider the fiber sequence
	\[
	\tau_{>d-1}\mathbb{Sp}^{p^k}
	\longrightarrow
	\tau_{>0}\mathbb{Sp}^{p^k}
	\longrightarrow
	\Sigma^{d-1}\mathbb{F}_p
	\]
	This induces the following long exact sequence in $\mathbb{F}_p$-cohomology:
	\[
	\cdots
	\longrightarrow
	\mathbb{F}_p^{*+d-2}\tau_{>d-1}\mathbb{Sp}^{p^k}
	\longrightarrow
	\mathcal{A}^*
	\xlongrightarrow{\cdot M_{k+1}}
	((\mathcal{A}/\mathcal{A}\beta)_{\geq k+1})^{*+d}
	\longrightarrow
	\cdots
	\]
	The next nontrivial homotopy group appears exactly where right multiplication by $M_{k+1}$ ceases to be an isomorphism. Equivalently, this amounts to studying right multiplication by $M_{k+1}\beta$ as a map
	\[
	\mathcal{A}^*
	\longrightarrow
	(\mathcal{A}_{\geq k+1}\beta)^{*+d+1}
	\]
	This analysis is carried out in the appendix (lem. \ref{mult2} and lem. \ref{multodd}). There we show that this map is an isomorphism for $*<e$, and that it has kernel $\mathbb{F}_p$ in degree $e$, where $e=2^{k+1}$ if $p=2$ and $e=2p^{k+1}-1$ if $p>2$. This implies that the next nontrivial homotopy group is in degree $d-2+e$. To conclude, note that $e-1=M(p^{k+1}-1)$, where $M=1$ if $p=2$ and $M=2$ if $p>2$, and so $d-2+e=(M+2)(p^{k+1}-1)-1$.
\end{proof}

The proof above is somewhat indirect. It would be more elegant to write down explicitly the beginning of a minimal resolution of the cohomology.

\begin{ques}
	What does the beginning of a minimal resolution of the module $(\mathcal{A}/\mathcal{A}\beta)_{\geq k+1}$ look like?
\end{ques}

Besides their $\mathbb{F}_p$-cohomology, the $K(i)$-cohomology of the symmetric product spectra can also be described; we learned this from a private communication with Nick Kuhn. We will only need this for $i\leq k+1$.

\begin{prop}
	\label{symmprodkn}
	The symmetric product spectrum $\mathbb{Sp}^{p^k}$ is $K(i)$-acyclic for $0 < i \leq k$, and its $K(k+1)$-cohomology has rank $p^{\binom{k+1}{2}}$ over $K(k+1)_*$.
\end{prop}

\begin{proof}
	The first statement is due to Welcher \cite[thm. 5.1]{Welcher}. Starting from Nakaoka's calculation, he showed that the $\mathbb{F}_p$-cohomology of $\tau_{>0}\mathbb{Sp}^{p^k}$, which he calls the \textit{symmetric fibers}, is free over the subalgebra $\mathcal{A}(k)$. This freeness implies that its Adams spectral sequence has a vanishing line of slope smaller than $1/|v_i|$ for $i\leq k$, and hence that it is $K(i)$-acyclic in this range. For the non-connected version, the difference is a copy of $\mathbb{Z}$, which is $K(i)$-acyclic only for $i>0$.
	
	The second statement follows from the first together with a result of Kuhn \cite[ex. 6.5]{Spchrom}. Consider the sequence
	\[
	\mathbb{Sp}^{p^k}
	\longrightarrow
	\mathbb{Sp}^{p^{k+1}}
	\longrightarrow
	\mathbb{Sp}^{p^{k+1}}/\mathbb{Sp}^{p^k}
	=: \Sigma^{k+1}L(k+1)
	\]
	By the first part, the boundary map $\Sigma^kL(k+1) \to \mathbb{Sp}^{p^k}$ is a $K(k+1)$-equivalence. Kuhn's result then tells us that the $K(n)$-cohomology of $L(n)$ has rank $p^{\binom{n}{2}}$, which gives the desired claim.
\end{proof}

\begin{rem}
	It is worth emphasizing again that this statement about cohomology is cleaner for the connected symmetric product spectrum $\tau_{>0}\mathbb{Sp}^{p^k}$, as we saw in the proof.
\end{rem}

Since the first nontrivial height appears at $k+1$, it is natural to ask the following question.

\begin{ques}
	Can we give an explicit computation of $\pi_*L_{K(k+1)}\mathbb{Sp}^{p^k}$, analogous to the known case $k = 0$?
\end{ques}

We conclude this subsection by making the connection between symmetric products and the Goodwillie tower of $S^1$ explicit. Nakaoka's calculation shows that the $\mathbb{F}_p$-cohomology of the quotient $\mathbb{Sp}^{p^k}/\mathbb{Sp}^{p^{k-1}}$ is given by the $\mathcal{A}$-module $(\mathcal{A}/\mathcal{A}\beta)_{=k}$. Up to a suspension shift, this coincides with the cohomology of the layer $S^1(k)$ computed by Arone--Mahowald. Remarkably, Arone--Dwyer lifted this to an equivalence of spectra.

\begin{prop}[{\cite[thm. 1.18]{AroneDwyer}}]
	\label{aronedwyer}
	There is an equivalence of spectra
	\[
	S^1(k) \simeq \Sigma^{1-2k}\mathbb{Sp}^{p^k}/\mathbb{Sp}^{p^{k-1}}
	\]
\end{prop}

\subsection{The model}

We are now ready to state and prove the main theorem of this paper.

\begin{thm}[Model for the approximation circles]
	\label{model}
	There is an equivalence
	\[
	\Omega S^1_k \simeq \mathbb{Z} \times \Omega^{\infty+2k}\tau_{>0}\mathbb{Sp}^{p^k}
	\]
\end{thm}

\begin{rem}
	Since $\tau_{>0}\mathbb{Sp}^m$ is rationally trivial, and since the maps $P_mI(S^1) \to P_{p^k}I(S^1)$ and $\tau_{>0}\mathbb{Sp}^{p^k} \to \tau_{>0}\mathbb{Sp}^m$ are $p$-local equivalences whenever $p^k \leq m < p^{k+1}$, a fracture square argument upgrades the result above to an integral equivalence
	\[
	\Omega P_mI(S^1) \simeq \mathbb{Z} \times \textstyle\prod_p\Omega^{\infty+2\lfloor\log_pm\rfloor}\bigl(\tau_{>0}\mathbb{Sp}^m\bigr)^\wedge_p
	\]
	We will not use this expression in what follows. It is nevertheless amusing to observe that the $m^\mathrm{th}$ stage of the Goodwillie tower of $S^1$ is determined by the spectrum $\mathbb{Sp}^m$, though not in a straightforward way. Rather, the spectrum is disassembled, shifted according to the arithmetic of $m$, and then reassembled.
\end{rem}

The proof proceeds in three steps. First, for a general functor $F$ we define the \emph{$m^\mathrm{th}$ remainder} $R_mF$ to be the fiber of the map $F \to P_mF$. We then show that $\Omega S^1_k \simeq \mathbb{Z} \times R_{p^k}I(S^1)$. Second, we construct maps between $R_{p^k}I(S^1)$ and $\Omega^{\infty+2k}\tau_{>0}\mathbb{Sp}^{p^k}$ using Arone--Dwyer's equivalence above. Finally, we prove that these maps are equivalences by appealing to the calculus version of the Whitehead conjecture, as resolved by Behrens \cite{Behrens2011} and Kuhn \cite{Kuhn2014}.

We start with a basic lemma about sections and splittings.

\begin{lem}[Splitting lemma]
	\label{splittinglemma}
    If the composite $X \to Y \to X$ is an equivalence, then
    \begin{enumerate}
        \item $\mathrm{fib}(X \to Y) \simeq \Omega \mathrm{fib}(Y \to X)$
        \item $\Omega Y \simeq \Omega X \times \Omega\mathrm{fib}(Y \to X)$
        \item $\Omega Y \simeq \Omega X \times \mathrm{fib}(X \to Y)$
    \end{enumerate}
\end{lem}

\begin{proof}
    The third statement follows immediately from the first two, so it suffices to prove the first two. Since the composite $X \to Y \to X$ is an equivalence, its fiber is contractible. By iterating pullbacks, we obtain the diagram

    % https://q.uiver.app/#q=WzAsNixbMCwxLCJYIl0sWzAsMCwiKiJdLFsxLDEsIlkiXSxbMSwwLCJcXG1hdGhybXtmaWJ9KFkgXFx0byBYKSJdLFsyLDEsIlgiXSxbMiwwLCIqIl0sWzEsMF0sWzAsMl0sWzEsM10sWzMsMl0sWzIsNF0sWzMsNV0sWzUsNF0sWzMsNCwiIiwxLHsic3R5bGUiOnsibmFtZSI6ImNvcm5lciJ9fV0sWzEsMiwiIiwxLHsic3R5bGUiOnsibmFtZSI6ImNvcm5lciJ9fV1d
\[
\begin{tikzcd}
	{*} & {\mathrm{fib}(Y \to X)} & {*} \\
	X & Y & X
	\arrow[from=1-1, to=1-2]
	\arrow[from=1-1, to=2-1]
	\arrow["\lrcorner"{anchor=center, pos=0.125}, draw=none, from=1-1, to=2-2]
	\arrow[from=1-2, to=1-3]
	\arrow[from=1-2, to=2-2]
	\arrow["\lrcorner"{anchor=center, pos=0.125}, draw=none, from=1-2, to=2-3]
	\arrow[from=1-3, to=2-3]
	\arrow[from=2-1, to=2-2]
	\arrow[from=2-2, to=2-3]
\end{tikzcd}
\]
    Both statements now follow from the fact that the left square is a pullback. Indeed, the first follows by taking horizontal fibers of the left square, while the second follows by rotating the induced fiber sequence
    \[
        \Omega Y \longrightarrow * \longrightarrow X \times \mathrm{fib}(Y \to X)
    \qedhere
    \]
\end{proof}

The previous lemma immediately gives the desired splitting. The point is that $S^1$ is already detected by its stable approximation, and hence by every further approximation.

\begin{lem}
    $\Omega S^1_k \simeq \mathbb{Z} \times R_{p^k}I(S^1)$.
\end{lem}

\begin{proof}
    By part (3) of the previous lemma, it suffices to show that the comparison map $S^1 \to S^1_k$ admits a retraction. Consider the composite
    \[
        S^1_k \longrightarrow S^1_0 \simeq \Omega^\infty \mathbb{S}^1 \longrightarrow \Omega^\infty\Sigma \mathbb{Z} \simeq S^1
    \]
    where the second map is induced by the suspension of the unit map $\mathbb{S} \to \mathbb{Z}$. The resulting composite $S^1 \to S^1_k \to S^1$ factors as $S^1 \to S^1_0 \to S^1$. The first map is the stabilization map, while the second is induced by the suspension of $\mathbb{S} \to \mathbb{Z}$, so both induce isomorphisms on $\pi_1$. Since $\pi_1$ is the only nontrivial homotopy group of $S^1$, it follows that the composite is an equivalence. Thus $S^1_k \to S^1$ is a retraction, and the claim follows.
\end{proof}

\begin{rem}
	\label{remisloop}
    The retraction constructed in the proof also shows that $R_{p^k}I(S^1)$ is a loop space. Indeed, by part (1) of the splitting lemma, it identifies with the loop space of the fiber of the retraction $S^1_k \to S^1$. We will use this later in the proof of the model.
\end{rem}

\begin{rem}
	\label{retractioninfloop}
	This holds more generally for every infinite loop space $\Omega^\infty X$. In particular, it also holds for $S^0$. Indeed, a retraction of the map $\Omega^\infty X \to P_mI(\Omega^\infty X)$ is given by the composite
	\[
	P_mI(\Omega^\infty X)
	\longrightarrow
	P_1I(\Omega^\infty X)
	\simeq
	\Omega^\infty \Sigma^\infty \Omega^\infty X
	\xlongrightarrow{\Omega^\infty\epsilon}
	\Omega^\infty X
	\]
	where $\epsilon$ is the counit of the adjunction $\Sigma^\infty \dashv \Omega^\infty$. This is indeed a retraction by the triangle identities.
\end{rem}

\begin{rem}
	\label{noretractionnocry}
	On the other hand, if $X$ is a simply connected finite non-contractible space, then the map $X \to P_mI(X)$ never admits a retraction, not even after looping sufficiently many times. In particular, no spheres other than $S^0$ and $S^1$ admit such a retraction.
	
	To see this, suppose there exist $m$ and $h$ together with a retraction
	\[
	\Omega^h P_mI(X) \xlongrightarrow{r} \Omega^h X
	\]
	Choose $m$ minimal with this property. We will prove that $m=0$, so that $\Omega^hX \simeq *$. This implies $X \simeq *$, since every simply connected finite space other than the point has infinitely many nontrivial homotopy groups, by an old result of Serre \cite[thm. 10]{Serre}.

	Suppose now that $m>0$. By Levy's spectral Sullivan conjecture \cite[thm. 1.1]{spectralSullivan}, the composite
	\[
	\Omega^{\infty+h}\mathbb{D}_mI(X)
	\longrightarrow
	\Omega^hP_mI(X)
	\xlongrightarrow{r}
	\Omega^hX
	\]
	factors through the $1$-truncation of the source. In particular, after looping twice, the map
	\[
	\Omega^{\infty+h+2}\mathbb{D}_mI(X)
	\longrightarrow
	\Omega^{h+2}P_mI(X)
	\xlongrightarrow{\Omega^2r}
	\Omega^{h+2}X
	\]
	is null. Hence we obtain a factorization
	\[
	\Omega^{\infty+h+2}\mathbb{D}_mI(X)
	\longrightarrow
	\Omega^2\mathrm{fib}(r)
	\longrightarrow
	\Omega^{h+2}P_mI(X)
	\]
	The fiber of this composite is described by the following pullback diagram:
	% https://q.uiver.app/#q=WzAsNixbMCwxLCJcXE9tZWdhXntcXGluZnR5ICsgaCArIDJ9XFxtYXRoYmJ7RH1fbUkoWCkiXSxbMSwxLCJcXE9tZWdhXjJcXG1hdGhybXtmaWJ9KHIpIl0sWzIsMSwiXFxPbWVnYV57aCsyfVBfbUkoWCkiXSxbMiwwLCIqIl0sWzAsMCwiXFxPbWVnYV57aCszfVBfe20tMX1JKFgpIl0sWzEsMCwiXFxPbWVnYV57aCszfVgiXSxbMCwxXSxbMSwyLCIiLDAseyJzdHlsZSI6eyJ0YWlsIjp7Im5hbWUiOiJob29rIiwic2lkZSI6InRvcCJ9fX1dLFszLDJdLFs0LDBdLFs1LDNdLFs0LDVdLFs1LDEsIioiLDFdLFs1LDIsIiIsMSx7InN0eWxlIjp7Im5hbWUiOiJjb3JuZXIifX1dLFs0LDEsIiIsMSx7InN0eWxlIjp7Im5hbWUiOiJjb3JuZXIifX1dXQ==
\[\begin{tikzcd}
	{\Omega^{h+3}P_{m-1}I(X)} & {\Omega^{h+3}X} & {*} \\
	{\Omega^{\infty + h + 2}\mathbb{D}_mI(X)} & {\Omega^2\mathrm{fib}(r)} & {\Omega^{h+2}P_mI(X)}
	\arrow[from=1-1, to=1-2]
	\arrow[from=1-1, to=2-1]
	\arrow["\lrcorner"{anchor=center, pos=0.125}, draw=none, from=1-1, to=2-2]
	\arrow[from=1-2, to=1-3]
	\arrow["{*}"{description}, from=1-2, to=2-2]
	\arrow["\lrcorner"{anchor=center, pos=0.125}, draw=none, from=1-2, to=2-3]
	\arrow[from=1-3, to=2-3]
	\arrow[from=2-1, to=2-2]
	\arrow[hook, from=2-2, to=2-3]
\end{tikzcd}\]
	which shows that $\Omega^{h+3}X$ is a retract of $\Omega^{h+3}P_{m-1}I(X)$, contradicting the minimality of $m$.
	
	This phenomenon is not specific to the Goodwillie tower, but holds for any finite extension of infinite loop spaces. It probably deserves to be stated as a lemma, but we prefer not to distract the reader from the main line of the argument.
\end{rem}

The rest of this subsection is devoted to proving the equivalence
\[
R_{p^k}I(S^1) \simeq \Omega^{\infty+2k}\tau_{>0}\mathbb{Sp}^{p^k}
\]
The first step is to produce maps between these spaces. The guiding idea is that, by the work of Arone--Dwyer, both sides come equipped with filtrations whose associated graded pieces are the same.

Indeed, the remainder $R_mF$ has the same associated graded as the Goodwillie tower, up to a shift, as the following lemma shows.

\begin{lem}
    If $F$ is a reduced functor, there are fiber sequences $R_mF \to R_{m-1}F \to D_mF$.
\end{lem}

\begin{proof}
    This follows by taking vertical fibers of the diagram
    % https://q.uiver.app/#q=WzAsNixbMCwwLCJGIl0sWzAsMSwiUF9tRiJdLFsxLDAsIkYiXSxbMiwwLCIqIl0sWzIsMSwiQkRfbUYiXSxbMSwxLCJQX3ttLTF9RiJdLFswLDFdLFswLDIsIiIsMix7ImxldmVsIjoyLCJzdHlsZSI6eyJoZWFkIjp7Im5hbWUiOiJub25lIn19fV0sWzIsM10sWzMsNF0sWzIsNV0sWzEsNV0sWzUsNF1d
\[\begin{tikzcd}
	F & F & {*} \\
	{P_mF} & {P_{m-1}F} & {BD_mF}
	\arrow[equals, from=1-1, to=1-2]
	\arrow[from=1-1, to=2-1]
	\arrow[from=1-2, to=1-3]
	\arrow[from=1-2, to=2-2]
	\arrow[from=1-3, to=2-3]
	\arrow[from=2-1, to=2-2]
	\arrow[from=2-2, to=2-3]
\end{tikzcd}\]
    The existence of a delooping of derivatives for reduced functors is a standard result in calculus \cite[lem. 2.2]{GoodwillieIII}.
\end{proof}

In particular, applying this to the identity functor, evaluating at $S^1$, and using proposition \ref{aronedwyer}, we obtain fiber sequences
\[
R_{p^k}I(S^1)
\longrightarrow
R_{p^{k-1}}I(S^1)
\longrightarrow
\Omega^\infty S^1(k)
\simeq
\Omega^{\infty-1+2k}\mathbb{Sp}^{p^k}/\mathbb{Sp}^{p^{k-1}}
\]
On the other hand, since $\mathbb{Sp}^{p^k}/\mathbb{Sp}^{p^{k-1}}$ is connected, taking connected covers gives fiber sequences
\[
\tau_{>0}\mathbb{Sp}^{p^{k-1}}
\longrightarrow
\tau_{>0}\mathbb{Sp}^{p^k}
\longrightarrow
\mathbb{Sp}^{p^k}/\mathbb{Sp}^{p^{k-1}}
\]
Before using these fiber sequences to construct the maps and prove the general result, we give a direct proof in the case $p=2$ and $k\leq 1$.

\begin{ex}
	\label{remainder0}
	We first identify the case $k=0$. We claim that
	\[
	R_1I(S^1) \simeq \Omega^\infty\tau_{>0}\mathbb{S}
	\]
	Indeed, by part (1) of the splitting lemma (lem. \ref{splittinglemma}), it suffices to identify the fiber of the retraction $S^1_0 \to S^1$. Such a retraction is given by applying $\Omega^\infty$ to the shifted unit map $\mathbb{S}^1 \to \Sigma\mathbb{Z}$. Its fiber is $\Sigma\tau_{>0}\mathbb{S}$, and the claim follows.
\end{ex}

\begin{ex}
	\label{remainder1}
	We now identify the case $k=1$, namely
	\[
	R_2I(S^1) \simeq \Omega^{\infty+2}\tau_{>0}\mathbb{Sp}^2
	\]
	By the previous example and the previous lemma, there is a fiber sequence
	\[
	R_2I(S^1)
	\longrightarrow
	R_1I(S^1) \simeq \Omega^\infty\tau_{>0}\mathbb{S}
	\xlongrightarrow{\delta}
	\Omega^\infty S^1(1) \simeq \Omega^\infty\mathbb{RP}^\infty
	\]
	On the other hand, the symmetric product spectra fit into a fiber sequence
	\[
	\mathbb{RP}^\infty
	\simeq
	\Sigma^{-1}(\mathbb{Sp}^2/\mathbb{S})
	\xlongrightarrow{\mu}
	\tau_{>0}\mathbb{S}
	\longrightarrow
	\tau_{>0}\mathbb{Sp}^2
	\]
	where the first map is the reduced Kahn--Priddy map. By the reduced Kahn--Priddy theorem, the composite
	\[
	\Omega^\infty\tau_{>0}\mathbb{S}
	\xlongrightarrow{\delta}
	\Omega^\infty\mathbb{RP}^\infty
	\xlongrightarrow{\Omega^\infty \mu}
	\Omega^\infty\tau_{>0}\mathbb{S}
	\]
	is an equivalence. Applying part (1) of the splitting lemma (lem. \ref{splittinglemma}), we conclude that
	\[
	R_2I(S^1)
	\simeq
	\Omega^{\infty+1}\mathrm{fib}(\mu)
	\simeq
	\Omega^{\infty+2}\tau_{>0}\mathbb{Sp}^2
	\]
	This is the example that motivated the whole paper. 
\end{ex}

\begin{rem}
	These examples can be generalized to all infinite loop spaces $\Omega^\infty X$, thanks to Kuhn's generalized Kahn--Priddy theorem \cite[thm. 1.1]{genKP}. First, note from remark \ref{retractioninfloop} and part (1) of the splitting lemma (lem. \ref{splittinglemma}) that
	\[
	R_1I(\Omega^\infty X) \simeq \Omega^{\infty+1}\mathrm{fib}(\epsilon)
	\]
	where $\epsilon$ is the counit of the adjunction $\Sigma^\infty \dashv \Omega^\infty$. Hence, by the previous lemma, the remainder fits into a fiber sequence
	\[
	R_2I(\Omega^\infty X)
	\longrightarrow
	R_1I(\Omega^\infty X)
	\simeq
	\Omega^{\infty+1}\mathrm{fib}(\epsilon)
	\xlongrightarrow{\delta}
	D_2I(\Omega^\infty X)
	\simeq
	\Omega^{\infty+1}(\Sigma^\infty\Omega^\infty X)^{\otimes 2}_{hC_2}
	\]
	On the other hand, Kuhn \cite[thm. 1.1]{genKP} shows that the multiplication map
	\[
	(\Sigma^\infty\Omega^\infty X)^{\otimes 2}_{hC_2}
	\longrightarrow
	\Sigma^\infty\Omega^\infty X
	\]
	factors as
	\[
	(\Sigma^\infty\Omega^\infty X)^{\otimes 2}_{hC_2}
	\xlongrightarrow{\mu}
	\mathrm{fib}(\epsilon)
	\longrightarrow
	\Sigma^\infty\Omega^\infty X
	\]
	and that $\Omega^{\infty+1}\mu$ provides a retraction to the map $\delta$ above. Applying part (1) of the splitting lemma, we conclude that
	\[
	R_2I(\Omega^\infty X) \simeq \Omega^{\infty+2}\mathrm{fib}(\mu)
	\]
	This assembles into an equivalence
	\[
	\Omega P_2I(\Omega^\infty X)
	\simeq
	\Omega^{\infty+1}X \times \Omega^{\infty+2}\mathrm{fib}(\mu)
	\]
	We believe that this behaviour generalizes to the whole Goodwillie tower, but we were unable to prove it, so we leave it as a question.
\end{rem}

\begin{ques}
	Is $\Omega P_mI(\Omega^\infty X)$ always an infinite loop space?
\end{ques}

\begin{rem}
	On the other hand, if $X$ is a simply connected finite non-contractible space, then the remainder $R_mI(X)$ cannot be an infinite loop space. Indeed, if it were, the map $R_mI(X) \to X$ would be null, since by Levy's spectral Sullivan conjecture \cite[thm. 1.1]{spectralSullivan} it would factor through the $1$-truncation of $R_mI(X)$. Hence $\Omega X$ would be a retract of $\Omega P_mI(X)$, and we saw in remark \ref{noretractionnocry} that this is impossible.
\end{rem}

To construct the maps, consider the rotations of the fiber sequences above:
\[
\Omega^{\infty+1}S^1(k)
\longrightarrow
R_{p^k}I(S^1)
\longrightarrow
R_{p^{k-1}}I(S^1)
\longrightarrow
\Omega^\infty S^1(k)
\]
\[
\Sigma^{-1}\mathbb{Sp}^{p^k}/\mathbb{Sp}^{p^{k-1}}
\longrightarrow
\tau_{>0}\mathbb{Sp}^{p^{k-1}}
\longrightarrow
\tau_{>0}\mathbb{Sp}^{p^k}
\longrightarrow
\mathbb{Sp}^{p^k}/\mathbb{Sp}^{p^{k-1}}
\]
Using these, together with Arone--Dwyer's equivalence (prop. \ref{aronedwyer})
\[
S^1(k) \simeq \Sigma^{1-2k}\mathbb{Sp}^{p^k}/\mathbb{Sp}^{p^{k-1}}
\]
we obtain two maps
\[
R_{p^k}I(S^1)
\longrightarrow
\Omega^\infty S^1(k+1)
\simeq
\Omega^{\infty+2k+1}\mathbb{Sp}^{p^{k+1}}/\mathbb{Sp}^{p^k}
\longrightarrow
\Omega^{\infty+2k}\tau_{>0}\mathbb{Sp}^{p^k}
\]
\[
\Omega^{\infty+2k}\tau_{>0}\mathbb{Sp}^{p^k}
\longrightarrow
\Omega^{\infty+2k}\mathbb{Sp}^{p^k}/\mathbb{Sp}^{p^{k-1}}
\simeq
\Omega^{\infty+1}S^1(k)
\longrightarrow
R_{p^k}I(S^1)
\]
We will prove that these maps are equivalences by showing that their composites are equivalences. These composites are
\[
R_{p^k}I(S^1)
\longrightarrow
\Omega^{\infty+2k+1}\mathbb{Sp}^{p^{k+1}}/\mathbb{Sp}^{p^k}
\xrightarrow{d_1^W}
\Omega^{\infty+2k}\mathbb{Sp}^{p^k}/\mathbb{Sp}^{p^{k-1}}
\longrightarrow
R_{p^k}I(S^1)
\]
\[
\Omega^{\infty+2k}\tau_{>0}\mathbb{Sp}^{p^k}
\longrightarrow
\Omega^{\infty+1}S^1(k)
\xrightarrow{d_1^G}
\Omega^\infty S^1(k+1)
\longrightarrow
\Omega^{\infty+2k}\tau_{>0}\mathbb{Sp}^{p^k}
\]
where $d_1^W$ and $d_1^G$ are the first differentials in the spectral sequences associated to the symmetric product filtration and the Goodwillie filtration of $S^1$, respectively.

We will show that these maps are equivalences using two finite variants of a proposition of Kuhn \cite[prop. A.1]{Kuhn2014}. The finite versions are necessary because we do not yet know that the remainder infinitely deloops, so we cannot produce an infinite diagram in the proof of corollary \ref{remainder}.

\begin{lem}
    Consider the following diagram
    % https://q.uiver.app/#q=WzAsMTEsWzAsMSwiWF97LTF9Il0sWzIsMSwiWF8wIl0sWzQsMSwiWF8xIl0sWzYsMSwiWF97ay0xfSJdLFs4LDEsIlhfayJdLFsxMCwxLCJYX3trKzF9Il0sWzEsMCwiWV97LTF9Il0sWzMsMCwiWV8wIl0sWzcsMCwiWV97ay0xfSJdLFs1LDEsIlxcY2RvdHMiXSxbOSwwLCJZX2siXSxbMCwxLCJkX3stMX0iLDAseyJvZmZzZXQiOi0xfV0sWzEsMiwiZF8wIiwwLHsib2Zmc2V0IjotMX1dLFszLDQsImRfe2stMX0iLDAseyJvZmZzZXQiOi0xfV0sWzQsNSwiZF9rIiwwLHsib2Zmc2V0IjotMX1dLFsxLDAsInNfey0xfSIsMCx7Im9mZnNldCI6LTF9XSxbMiwxLCJzXzAiLDAseyJvZmZzZXQiOi0xfV0sWzQsMywic197ay0xfSIsMCx7Im9mZnNldCI6LTF9XSxbNSw0LCJzX2siLDAseyJvZmZzZXQiOi0xfV0sWzYsMSwiaV97LTF9Il0sWzEsNywicF8wIl0sWzcsMiwiaV8wIl0sWzAsNiwiIiwwLHsibGV2ZWwiOjIsInN0eWxlIjp7ImhlYWQiOnsibmFtZSI6Im5vbmUifX19XSxbMyw4LCJwX3trLTF9Il0sWzIsOSwiIiwwLHsib2Zmc2V0IjotMX1dLFs5LDIsIiIsMCx7Im9mZnNldCI6LTF9XSxbOSwzLCIiLDAseyJvZmZzZXQiOi0xfV0sWzMsOSwiIiwwLHsib2Zmc2V0IjotMX1dLFs4LDQsImlfe2stMX0iXSxbNCwxMCwicF9rIl0sWzEwLDUsImlfayJdXQ==
\[\begin{tikzcd}[sep=small]
	& {Y_{-1}} && {Y_0} &&&& {Y_{k-1}} && {Y_k} \\
	{X_{-1}} && {X_0} && {X_1} & \cdots & {X_{k-1}} && {X_k} && {X_{k+1}}
	\arrow["{i_{-1}}", from=1-2, to=2-3]
	\arrow["{i_0}", from=1-4, to=2-5]
	\arrow["{i_{k-1}}", from=1-8, to=2-9]
	\arrow["{i_k}", from=1-10, to=2-11]
	\arrow[equals, from=2-1, to=1-2]
	\arrow["{d_{-1}}", shift left, from=2-1, to=2-3]
	\arrow["{p_0}", from=2-3, to=1-4]
	\arrow["{s_{-1}}", shift left, from=2-3, to=2-1]
	\arrow["{d_0}", shift left, from=2-3, to=2-5]
	\arrow["{s_0}", shift left, from=2-5, to=2-3]
	\arrow[shift left, from=2-5, to=2-6]
	\arrow[shift left, from=2-6, to=2-5]
	\arrow[shift left, from=2-6, to=2-7]
	\arrow["{p_{k-1}}", from=2-7, to=1-8]
	\arrow[shift left, from=2-7, to=2-6]
	\arrow["{d_{k-1}}", shift left, from=2-7, to=2-9]
	\arrow["{p_k}", from=2-9, to=1-10]
	\arrow["{s_{k-1}}", shift left, from=2-9, to=2-7]
	\arrow["{d_k}", shift left, from=2-9, to=2-11]
	\arrow["{s_k}", shift left, from=2-11, to=2-9]
\end{tikzcd}\]
    such that
    \begin{enumerate}
        \item every space is a loop space
        \item every triangle commutes, meaning that $d_j\simeq i_jp_j$ 
        \item $Y_j \xrightarrow{i_j} X_{j+1} \xrightarrow{p_{j+1}} Y_{j+1}$ is a fiber sequence
        \item the map $d_{j-1}s_{j-1} + s_jd_j$ is an equivalence for $j \leq k$
    \end{enumerate}
    Then the composite $Y_j \xrightarrow{i_j} X_{j+1} \xrightarrow{s_j} X_j \xrightarrow{p_j} Y_j$ is an equivalence for all $j \leq k$.
\end{lem}

\begin{proof}
    We argue by induction on $j$. For $j=-1$, the claim is immediate. Now assume that
    \[
        Y_{j-1} \xrightarrow{i_{j-1}} X_j \xrightarrow{s_{j-1}} X_{j-1} \xrightarrow{p_{j-1}} Y_{j-1}
    \]
    is an equivalence. In particular, $i_{j-1}$ admits a retraction. Since
    \[
        Y_{j-1} \xrightarrow{i_{j-1}} X_j \xrightarrow{p_j} Y_j
    \]
    is a fiber sequence of loop spaces, it follows that there is a splitting
    \[
        X_j \simeq Y_{j-1} \times Y_j
    \]
    under which $i_{j-1}$ identifies with the inclusion of the first factor and $p_j$ with the projection onto the second. Let $q_j \colon Y_j \to X_j$ denote the inclusion of the second factor. Consider the map
    \[
        u_j := d_{j-1}s_{j-1} + s_jd_j \colon X_j \longrightarrow X_j
    \]
    We claim that, after passing to homotopy groups and using the above product decomposition, $u_j$ is represented by an upper triangular matrix
    \[
        \begin{bmatrix}
            p_{j-1}s_{j-1}i_{j-1} & \star \\
            0 & p_js_ji_j
        \end{bmatrix}
    \]
    Since $u_j$ is an equivalence, this matrix is an isomorphism on homotopy groups. Since the upper-left entry is an isomorphism by the inductive hypothesis, the lower-right entry is also an isomorphism. Hence $p_js_ji_j$ is an equivalence.

    Indeed, the component $Y_{j-1} \to Y_j$ vanishes, since
    \[
        p_j u_j i_{j-1}
        \simeq
        p_j d_{j-1}s_{j-1}i_{j-1} + p_j s_j d_j i_{j-1}
        \simeq
        p_j i_{j-1}p_{j-1}s_{j-1}i_{j-1} + p_j s_j i_j p_j i_{j-1}
        \simeq *
    \]
    because $d_{j-1} \simeq i_{j-1}p_{j-1}$, $d_j \simeq i_jp_j$, and $p_j i_{j-1} \simeq *$.

    Similarly, its component $Y_{j-1}\to Y_{j-1}$ is given by
    \[
        p_{j-1} u_j i_{j-1}
        \simeq
        p_{j-1} d_{j-1}s_{j-1}i_{j-1} + p_{j-1} s_j d_j i_{j-1}
        \simeq
        p_{j-1} i_{j-1}p_{j-1}s_{j-1}i_{j-1} + *
        \simeq
        p_{j-1}s_{j-1}i_{j-1}
    \]
    Finally, its component $Y_j\to Y_j$ is given by
    \[
        p_j u_j q_j
        \simeq
        p_j d_{j-1}s_{j-1}q_j + p_j s_j d_j q_j
        \simeq
        * + p_j s_j i_j p_j q_j
        \simeq
        p_j s_j i_j
    \]
    This proves the matrix decomposition and completes the proof.
\end{proof}

We will also need the following wrong-way variant, whose proof is analogous.

\begin{lem}
    Consider the following diagram 
    % https://q.uiver.app/#q=WzAsMTEsWzAsMCwiWF97LTF9Il0sWzIsMCwiWF8wIl0sWzQsMCwiWF8xIl0sWzYsMCwiWF97ay0xfSJdLFs4LDAsIlhfayJdLFsxMCwwLCJYX3trKzF9Il0sWzEsMSwiWl97LTF9Il0sWzMsMSwiWl8wIl0sWzcsMSwiWl97ay0xfSJdLFs1LDAsIlxcY2RvdHMiXSxbOSwxLCJaX2siXSxbMCwxLCJkX3stMX0iLDAseyJvZmZzZXQiOi0xfV0sWzEsMiwiZF8wIiwwLHsib2Zmc2V0IjotMX1dLFszLDQsImRfe2stMX0iLDAseyJvZmZzZXQiOi0xfV0sWzQsNSwiZF9rIiwwLHsib2Zmc2V0IjotMX1dLFsxLDAsInNfey0xfSIsMCx7Im9mZnNldCI6LTF9XSxbMiwxLCJzXzAiLDAseyJvZmZzZXQiOi0xfV0sWzQsMywic197ay0xfSIsMCx7Im9mZnNldCI6LTF9XSxbNSw0LCJzX2siLDAseyJvZmZzZXQiOi0xfV0sWzEsNiwicV97LTF9Il0sWzcsMSwial8wIl0sWzIsNywicV8wIl0sWzAsNiwiIiwwLHsibGV2ZWwiOjIsInN0eWxlIjp7ImhlYWQiOnsibmFtZSI6Im5vbmUifX19XSxbOCwzLCJqX3trLTF9Il0sWzIsOSwiIiwwLHsib2Zmc2V0IjotMX1dLFs5LDIsIiIsMCx7Im9mZnNldCI6LTF9XSxbOSwzLCIiLDAseyJvZmZzZXQiOi0xfV0sWzMsOSwiIiwwLHsib2Zmc2V0IjotMX1dLFs0LDgsInFfe2stMX0iXSxbMTAsNCwial9rIl0sWzUsMTAsInFfayJdXQ==
\[\begin{tikzcd}[sep=small]
	{X_{-1}} && {X_0} && {X_1} & \cdots & {X_{k-1}} && {X_k} && {X_{k+1}} \\
	& {Z_{-1}} && {Z_0} &&&& {Z_{k-1}} && {Z_k}
	\arrow["{d_{-1}}", shift left, from=1-1, to=1-3]
	\arrow[equals, from=1-1, to=2-2]
	\arrow["{s_{-1}}", shift left, from=1-3, to=1-1]
	\arrow["{d_0}", shift left, from=1-3, to=1-5]
	\arrow["{q_{-1}}", from=1-3, to=2-2]
	\arrow["{s_0}", shift left, from=1-5, to=1-3]
	\arrow[shift left, from=1-5, to=1-6]
	\arrow["{q_0}", from=1-5, to=2-4]
	\arrow[shift left, from=1-6, to=1-5]
	\arrow[shift left, from=1-6, to=1-7]
	\arrow[shift left, from=1-7, to=1-6]
	\arrow["{d_{k-1}}", shift left, from=1-7, to=1-9]
	\arrow["{s_{k-1}}", shift left, from=1-9, to=1-7]
	\arrow["{d_k}", shift left, from=1-9, to=1-11]
	\arrow["{q_{k-1}}", from=1-9, to=2-8]
	\arrow["{s_k}", shift left, from=1-11, to=1-9]
	\arrow["{q_k}", from=1-11, to=2-10]
	\arrow["{j_0}", from=2-4, to=1-3]
	\arrow["{j_{k-1}}", from=2-8, to=1-7]
	\arrow["{j_k}", from=2-10, to=1-9]
\end{tikzcd}\]
    such that
    \begin{enumerate}
        \item every space is a loop space
        \item every triangle commutes, meaning that $s_j\simeq j_jq_j$ 
        \item $Z_j \xrightarrow{j_j} X_j \xrightarrow{q_{j-1}} Z_{j-1}$ is a fiber sequence
        \item the map $d_{j-1}s_{j-1} + s_jd_j$ is an equivalence for $j \leq k$
    \end{enumerate}
    Then the composite $Z_j \xrightarrow{j_j} X_j \xrightarrow{d_j} X_{j+1} \xrightarrow{q_j} Z_j$ is an equivalence for all $j \leq k$.
\end{lem}

\begin{cor}
    Consider the following diagram 
% https://q.uiver.app/#q=WzAsMTUsWzAsMSwiWF97LTF9Il0sWzIsMSwiWF8wIl0sWzQsMSwiWF8xIl0sWzYsMSwiWF97ay0xfSJdLFs4LDEsIlhfayJdLFsxMCwxLCJYX3trKzF9Il0sWzEsMiwiWl97LTF9Il0sWzMsMiwiWl8wIl0sWzcsMiwiWl97ay0xfSJdLFs1LDEsIlxcY2RvdHMiXSxbOSwyLCJaX2siXSxbMSwwLCJZX3stMX0iXSxbMywwLCJZXzAiXSxbNywwLCJZX3trLTF9Il0sWzksMCwiWV9rIl0sWzAsMSwiZF97LTF9IiwwLHsib2Zmc2V0IjotMX1dLFsxLDIsImRfMCIsMCx7Im9mZnNldCI6LTF9XSxbMyw0LCJkX3trLTF9IiwwLHsib2Zmc2V0IjotMX1dLFs0LDUsImRfayIsMCx7Im9mZnNldCI6LTF9XSxbMSwwLCJzX3stMX0iLDAseyJvZmZzZXQiOi0xfV0sWzIsMSwic18wIiwwLHsib2Zmc2V0IjotMX1dLFs0LDMsInNfe2stMX0iLDAseyJvZmZzZXQiOi0xfV0sWzUsNCwic19rIiwwLHsib2Zmc2V0IjotMX1dLFsxLDYsInFfey0xfSJdLFs3LDEsImpfMCJdLFsyLDcsInFfMCJdLFswLDYsIiIsMCx7ImxldmVsIjoyLCJzdHlsZSI6eyJoZWFkIjp7Im5hbWUiOiJub25lIn19fV0sWzgsMywial97ay0xfSJdLFsyLDksIiIsMCx7Im9mZnNldCI6LTF9XSxbOSwyLCIiLDAseyJvZmZzZXQiOi0xfV0sWzksMywiIiwwLHsib2Zmc2V0IjotMX1dLFszLDksIiIsMCx7Im9mZnNldCI6LTF9XSxbNCw4LCJxX3trLTF9Il0sWzEwLDQsImpfayJdLFs1LDEwLCJxX2siXSxbMCwxMSwiIiwwLHsibGV2ZWwiOjIsInN0eWxlIjp7ImhlYWQiOnsibmFtZSI6Im5vbmUifX19XSxbMTEsMSwiaV97LTF9Il0sWzMsMTMsInBfe2stMX0iXSxbMTMsNCwiaV97ay0xfSJdLFs0LDE0LCJwX2siXSxbMTQsNSwiaV9rIl0sWzEyLDIsImlfMCJdLFsxLDEyLCJwXzAiXV0=
\[\begin{tikzcd}[sep=small]
	& {Y_{-1}} && {Y_0} &&&& {Y_{k-1}} && {Y_k} \\
	{X_{-1}} && {X_0} && {X_1} & \cdots & {X_{k-1}} && {X_k} && {X_{k+1}} \\
	& {Z_{-1}} && {Z_0} &&&& {Z_{k-1}} && {Z_k}
	\arrow["{i_{-1}}", from=1-2, to=2-3]
	\arrow["{i_0}", from=1-4, to=2-5]
	\arrow["{i_{k-1}}", from=1-8, to=2-9]
	\arrow["{i_k}", from=1-10, to=2-11]
	\arrow[equals, from=2-1, to=1-2]
	\arrow["{d_{-1}}", shift left, from=2-1, to=2-3]
	\arrow[equals, from=2-1, to=3-2]
	\arrow["{p_0}", from=2-3, to=1-4]
	\arrow["{s_{-1}}", shift left, from=2-3, to=2-1]
	\arrow["{d_0}", shift left, from=2-3, to=2-5]
	\arrow["{q_{-1}}", from=2-3, to=3-2]
	\arrow["{s_0}", shift left, from=2-5, to=2-3]
	\arrow[shift left, from=2-5, to=2-6]
	\arrow["{q_0}", from=2-5, to=3-4]
	\arrow[shift left, from=2-6, to=2-5]
	\arrow[shift left, from=2-6, to=2-7]
	\arrow["{p_{k-1}}", from=2-7, to=1-8]
	\arrow[shift left, from=2-7, to=2-6]
	\arrow["{d_{k-1}}", shift left, from=2-7, to=2-9]
	\arrow["{p_k}", from=2-9, to=1-10]
	\arrow["{s_{k-1}}", shift left, from=2-9, to=2-7]
	\arrow["{d_k}", shift left, from=2-9, to=2-11]
	\arrow["{q_{k-1}}", from=2-9, to=3-8]
	\arrow["{s_k}", shift left, from=2-11, to=2-9]
	\arrow["{q_k}", from=2-11, to=3-10]
	\arrow["{j_0}", from=3-4, to=2-3]
	\arrow["{j_{k-1}}", from=3-8, to=2-7]
	\arrow["{j_k}", from=3-10, to=2-9]
\end{tikzcd}\]
    such that
    \begin{enumerate}
        \item every space is a loop space
        \item every triangle commutes, meaning that $d_j\simeq i_jp_j$ and $s_j\simeq j_jq_j$ 
        \item $Y_j \xrightarrow{i_j} X_{j+1} \xrightarrow{p_{j+1}} Y_{j+1}$ is a fiber sequence
        \item $Z_j \xrightarrow{j_j} X_j \xrightarrow{q_{j-1}} Z_{j-1}$ is a fiber sequence
        \item the map $d_{j-1}s_{j-1} + s_jd_j$ is an equivalence for $j \leq k$
    \end{enumerate}
    Then the composites $Y_j \xrightarrow{i_j} X_{j+1} \xrightarrow{q_j} Z_j$ and $Z_j \xrightarrow{j_j} X_j \xrightarrow{p_j} Y_j$ are equivalences for all $j \leq k$.
\end{cor}

\begin{proof}
    Let $a_j := q_ji_j$ and $b_j := p_jj_j$. We must show that both maps are equivalences.

    By the first lemma, the composite $p_js_ji_j$ is an equivalence. Since the bottom triangle commutes, we have $s_j \simeq j_jq_j$, so
    \[
        p_js_ji_j \simeq p_jj_jq_ji_j \simeq b_ja_j
    \]
    Hence $b_ja_j$ is an equivalence.

    Similarly, by the wrong-way variant, the composite $q_jd_jj_j$ is an equivalence. Since the top triangle commutes, we have $d_j \simeq i_jp_j$, so
    \[
        q_jd_jj_j \simeq q_ji_jp_jj_j \simeq a_jb_j
    \]
    Hence $a_jb_j$ is an equivalence.

    Since both $b_ja_j$ and $a_jb_j$ are equivalences, it follows that both $a_j$ and $b_j$ are equivalences.
\end{proof}

\begin{cor}
	\label{remainder}
    $R_{p^k}I(S^1) \simeq \Omega^{\infty + 2k}\tau_{>0}\mathbb{Sp}^{p^k}$.
\end{cor}

\begin{proof}
    Using the Goodwillie filtration for the remainder of $S^1$ and the symmetric product filtration, as discussed before the lemmas, we obtain the following diagram.

\adjustbox{scale=0.90,center}{
% https://q.uiver.app/#q=WzAsMTYsWzAsMSwiXFxPbWVnYV57aysxfVNeMSJdLFsyLDEsIlxcT21lZ2Fee1xcaW5mdHkraysxfVNeMSgwKSJdLFs0LDEsIlxcT21lZ2Fee1xcaW5mdHkra31TXjEoMSkiXSxbMSw0LCJcXE9tZWdhXntcXGluZnR5KzJ9U14xKGstMSkiXSxbMyw0LCJcXE9tZWdhXntcXGluZnR5KzF9U14xKGspIl0sWzUsNCwiXFxPbWVnYV5cXGluZnR5IFNeMShrKzEpIl0sWzEsMiwiXFxPbWVnYV57aysxfVNeMSJdLFszLDIsIlxcT21lZ2Fee1xcaW5mdHkgKyBrfVxcdGF1X3s+MH1cXG1hdGhiYntTfSJdLFsyLDUsIlxcT21lZ2Fee1xcaW5mdHkgKyAyayAtIDF9XFx0YXVfez4wfVxcbWF0aGJie1NwfV57cF57ay0xfX0iXSxbNSwxLCJcXGNkb3RzIl0sWzQsNSwiXFxPbWVnYV57XFxpbmZ0eSArIDJrfVxcdGF1X3s+MH1cXG1hdGhiYntTcH1ee3Bea30iXSxbMSwwLCJcXE9tZWdhXntrKzF9U14xIl0sWzMsMCwiXFxPbWVnYV5rUl8xSShTXjEpIl0sWzIsMywiXFxPbWVnYSBSX3twXntrLTF9fUkoU14xKSJdLFs0LDMsIlJfe3Bea31JKFNeMSkiXSxbMCw0LCJcXGNkb3RzIl0sWzAsMSwiIiwwLHsib2Zmc2V0IjotMX1dLFsxLDIsImRfMV5HIiwwLHsib2Zmc2V0IjotMX1dLFszLDQsImRfMV5HIiwwLHsib2Zmc2V0IjotMX1dLFsxLDAsIiIsMCx7Im9mZnNldCI6LTF9XSxbMiwxLCJkXzFeVyIsMCx7Im9mZnNldCI6LTF9XSxbNCwzLCJkXzFeVyIsMCx7Im9mZnNldCI6LTF9XSxbMSw2XSxbNywxXSxbMiw3XSxbMCw2LCIiLDAseyJsZXZlbCI6Miwic3R5bGUiOnsiaGVhZCI6eyJuYW1lIjoibm9uZSJ9fX1dLFs4LDNdLFsyLDksIiIsMCx7Im9mZnNldCI6LTF9XSxbOSwyLCIiLDAseyJvZmZzZXQiOi0xfV0sWzQsOF0sWzEwLDRdLFs1LDEwXSxbMCwxMSwiIiwwLHsibGV2ZWwiOjIsInN0eWxlIjp7ImhlYWQiOnsibmFtZSI6Im5vbmUifX19XSxbMTEsMV0sWzMsMTNdLFsxMyw0XSxbNCwxNF0sWzE0LDVdLFsxMiwyXSxbMSwxMl0sWzQsNSwiZF8xXkciLDAseyJvZmZzZXQiOi0xfV0sWzUsNCwiZF8xXlciLDAseyJvZmZzZXQiOi0xfV0sWzMsMTUsIiIsMCx7Im9mZnNldCI6LTF9XSxbMTUsMywiIiwwLHsib2Zmc2V0IjotMX1dXQ==
\begin{tikzcd}[column sep=small]
	& {\Omega^{k+1}S^1} && {\Omega^kR_1I(S^1)} && \\
	{\Omega^{k+1}S^1} && {\Omega^{\infty+k+1}S^1(0)} && {\Omega^{\infty+k}S^1(1)} & \cdots \\
	& {\Omega^{k+1}S^1} && {\Omega^{\infty + k}\tau_{>0}\mathbb{S}} \\
	&& {\Omega R_{p^{k-1}}I(S^1)} && {R_{p^k}I(S^1)} \\
	\cdots & {\Omega^{\infty+2}S^1(k-1)} && {\Omega^{\infty+1}S^1(k)} && {\Omega^\infty S^1(k+1)} \\
	&& {\Omega^{\infty + 2k - 1}\tau_{>0}\mathbb{Sp}^{p^{k-1}}} && {\Omega^{\infty + 2k}\tau_{>0}\mathbb{Sp}^{p^k}}
	\arrow[from=1-2, to=2-3]
	\arrow[from=1-4, to=2-5]
	\arrow[equals, from=2-1, to=1-2]
	\arrow[shift left, from=2-1, to=2-3]
	\arrow[equals, from=2-1, to=3-2]
	\arrow[from=2-3, to=1-4]
	\arrow[shift left, from=2-3, to=2-1]
	\arrow["{d_1^G}", shift left, from=2-3, to=2-5]
	\arrow[from=2-3, to=3-2]
	\arrow["{d_1^W}", shift left, from=2-5, to=2-3]
	\arrow[shift left, from=2-5, to=2-6]
	\arrow[from=2-5, to=3-4]
	\arrow[shift left, from=2-6, to=2-5]
	\arrow[from=3-4, to=2-3]
	\arrow[from=4-3, to=5-4]
	\arrow[from=4-5, to=5-6]
	\arrow[shift left, from=5-1, to=5-2]
	\arrow[from=5-2, to=4-3]
	\arrow[shift left, from=5-2, to=5-1]
	\arrow["{d_1^G}", shift left, from=5-2, to=5-4]
	\arrow[from=5-4, to=4-5]
	\arrow["{d_1^W}", shift left, from=5-4, to=5-2]
	\arrow["{d_1^G}", shift left, from=5-4, to=5-6]
	\arrow[from=5-4, to=6-3]
	\arrow["{d_1^W}", shift left, from=5-6, to=5-4]
	\arrow[from=5-6, to=6-5]
	\arrow[from=6-3, to=5-2]
	\arrow[from=6-5, to=5-4]
\end{tikzcd}
}

	By remark \ref{remisloop}, all terms in this diagram are loop spaces. It satisfies all the hypotheses of the previous corollary except condition (5), which is the only non-formal point. In the present case, condition (5) is precisely the statement that
    \[
        d_1^Wd_1^G + d_1^Gd_1^W
    \]
    is an equivalence. This follows from the solution of the Whitehead conjecture \cite{Kuhn2014,Behrens2011}. The previous corollary therefore shows that the composite
    \[
        R_{p^k}I(S^1) \longrightarrow \Omega^{\infty+2k}\tau_{>0}\mathbb{Sp}^{p^k}
    \]
    is an equivalence.
\end{proof}

\begin{rem}
	The proof above is somewhat indirect because it is phrased in terms of the exactness statement provided by the Whitehead conjecture. The intuition behind the result is simpler. By induction on $k$, the fiber sequence
	\[
	R_{p^k}I(S^1)
	\longrightarrow
	R_{p^{k-1}}I(S^1)
	\longrightarrow
	\Omega^\infty S^1(k)
	\]
	may be identified, using Arone--Dwyer's equivalence (prop. \ref{aronedwyer}), with a fiber sequence of the form
	\[
	R_{p^k}I(S^1)
	\longrightarrow
	\Omega^{\infty+2k-2}\tau_{>0}\mathbb{Sp}^{p^{k-1}}
	\longrightarrow
	\Omega^{\infty+2k-1}\mathbb{Sp}^{p^k}/\mathbb{Sp}^{p^{k-1}}
	\]
	The Whitehead conjecture says, in particular, that the right-hand map is injective on homotopy groups. The stronger intuition is that this map should actually be split injective, with retraction induced by the boundary map in the symmetric product filtration
	\[
	\Sigma^{1-2k}\mathbb{Sp}^{p^k}/\mathbb{Sp}^{p^{k-1}}
	\longrightarrow
	\Sigma^{2-2k}\tau_{>0}\mathbb{Sp}^{p^{k-1}}
	\]
	Assuming this splitting, part (1) of the splitting lemma (lem. \ref{splittinglemma}) gives
	\[
	R_{p^k}I(S^1)
	\simeq
	\Omega\,\fib\!\left(
	\Omega^{\infty+2k-1}\mathbb{Sp}^{p^k}/\mathbb{Sp}^{p^{k-1}}
	\longrightarrow
	\Omega^{\infty+2k-2}\tau_{>0}\mathbb{Sp}^{p^{k-1}}
	\right)
	\simeq
	\Omega^{\infty+2k}\tau_{>0}\mathbb{Sp}^{p^k}
	\]
	This is exactly the argument used in examples \ref{remainder0} and \ref{remainder1} in the cases $k=0,1$. The role of the proof of corollary \ref{remainder} is to make this splitting argument precise in general, using the exactness statement supplied by the Whitehead conjecture.
\end{rem}

\begin{rem}
    It is worth noting that this result immediately recovers two familiar convergence properties. First, the connectivity of the remainder tends to infinity as $k\to\infty$ (lem. \ref{homotopysp}), so the approximation map $S^1 \to S^1_k$ converges in the usual sense. Second, since $\tau_{>0}\mathbb{Sp}^{p^k}$ is $K(i)$-acyclic for $i\leq k$ (prop. \ref{symmprodkn}), the approximation map $S^1 \to S^1_k$ is also an equivalence in $v_i$-periodic homotopy for $i\leq k$, as shown in \cite[thm. 4.1]{AroneMahowald}.
\end{rem}

Before moving to the next section and exploiting this result for computations, we record a purely theoretical consequence of the previous analysis.

\begin{prop}
	\label{ext}
	After sufficiently many loops, the approximation spheres $S^n_k$ are finite extensions of
	\[
	\mathbb{Z},\quad
	\Omega^\infty\tau_{>0}\mathbb{S},\quad
	\dots,\quad
	\Omega^\infty\tau_{>0}\mathbb{Sp}^{p^k}
	\]
\end{prop}

\begin{proof}
	We first reduce to the case where $n$ is odd. Indeed, by the calculus version of the EHP sequence (prop. \ref{EHPcalc}), an even approximation sphere is an extension of two odd approximation spheres of lesser or equal approximation degree.

	Now assume that $n$ is odd. We argue by induction on $k$ and $n$. Here $\langle X_1,\dots,X_m\rangle$ denotes the class generated by $X_1,\dots,X_m$ under iterated finite extensions and equivalences. The inductive step is provided by the calculus version of the Gray sequence (thm. \ref{graysequence}), since
	\[
	\Omega^3S^n_k
	\in
	\bigl\langle \Omega S^{n-2}_k, W^{n-2}_k \bigr\rangle
	\subseteq
	\bigl\langle
	\Omega S^{n-2}_k,
	\Omega S^{(n-1)p-1}_{k-1},
	\Omega^3S^{(n-1)p+1}_{k-1}
	\bigr\rangle
	\]
	The base case $k=0$ is immediate, since $\Omega^nS^n_0 \simeq \Omega^\infty\mathbb{S}$ for all $n$. The base case $n=1$ is handled by the explicit model for $S^1_k$ (thm. \ref{model}).
\end{proof}

\begin{rem}
	This proposition already follows formally from the EHP and Gray sequences together with the work of Arone--Dwyer, without using the explicit model for $S^1_k$.
\end{rem}

\begin{ques}
	Can the hypothesis ``after sufficiently many loops'' be removed?
\end{ques}

We conclude this subsection with a speculation about the approximations to $S^0$. As a starting point, consider the case $p=2$ and $k=1$. There is a fiber sequence
\[
S^0_1
\longrightarrow
S^0_0 \simeq \Omega^\infty\mathbb{S}
\longrightarrow
B\Omega^\infty S^0(1) \simeq \Omega^\infty\mathbb{RP}^\infty_+
\]
Postcomposing the right-hand map with the projection $\mathbb{RP}^\infty_+ \to \mathbb{RP}^\infty$ gives precisely the map appearing in the Goodwillie tower of $S^1$, namely essentially the inverse to the reduced Kahn--Priddy map. In particular, the right-hand map is injective on $\pi_*$ for $*>0$. In degree $0$, the induced map $\mathbb{Z}\to\mathbb{Z}$ is given by $n\mapsto \binom{n}{2}$, although we omit the proof. Since this vanishes only for $n=0,1$, this suggests that
\[
S^0_1 \simeq S^0 \times \Omega^{\infty+1}\mathbb{RP}^\infty
\]
For a broader perspective at $p=2$, consider the calculus version of the EHP sequence (prop. \ref{EHPcalc})
\[
S^0_k
\longrightarrow
\Omega S^1_k
\longrightarrow
\Omega S^1_{k-1}
\]
Suppose that, in positive degrees, the right-hand map is null, while in degree $0$ it behaves as above. This would imply that
\[
S^0_k \simeq S^0 \times \Omega^{\infty+1}S^1(k)
\]
We conclude this subsection with the following question.

\begin{ques}
    Is there a similar model for $S^0_k$? Is it true that $S^0_k \simeq S^0 \times \Omega^{\infty+1}S^1(k)$ at $p=2$?
\end{ques}

This question is essentially equivalent to a question of Dwyer, recorded by Arone and Lesh in the discussion following conjecture 2.28 of \cite{AroneLesh2009}. At the prime $2$, the layers of the Goodwillie tower of the identity evaluated at $S^0$ are
$$
    S^0(k) \simeq \Omega^{\infty+k}M(k)
    \simeq
    \Omega^{\infty+k}L(k-1)
    \times
    \Omega^{\infty+k}L(k)
$$
where $M(k)$ denotes the $k^\text{th}$ spectrum in the mod $2$ Whitehead complex and $M(k)\simeq L(k-1) \oplus L(k)$ is its Mitchell--Priddy splitting. Dwyer's question asks whether the Goodwillie tower of $S^0$ provides a contracting homotopy for the mod $2$ Whitehead complex. If so, the same argument used to prove the model for the approximations to $S^1$ would give
$$
    S^0_k \simeq S^0 \times \Omega^{\infty+k}L(k)
$$
Finally, using the identifications
$$
    L(k) \simeq
    \Sigma^{-k}\mathbb{Sp}^{2^k}/\mathbb{Sp}^{2^{k-1}}
    \qquad
    S^1(k) \simeq
    \Sigma^{1-2k}\mathbb{Sp}^{2^k}/\mathbb{Sp}^{2^{k-1}}
$$
this becomes
$$
    S^0_k \simeq S^0 \times \Omega^{\infty+1}S^1(k)
$$
which is precisely the conjectural equivalence above.

\begin{rem}
	\label{convS0}
    This conjecture would also imply that the Goodwillie tower of $S^0$ converges at the prime $2$. This is raised as a question by Arone and Lesh in the same discussion, and appears to be unknown in the literature. We can nevertheless prove convergence independently of the conjecture. Consider the calculus version of the EHP sequence
    $$
        S^0_k
        \longrightarrow
        \Omega S^1_k
        \longrightarrow
        \Omega S^1_{k-1}
    $$
    Passing to the limit over $k$ and using convergence for $S^1$, we obtain a fiber sequence
    $$
        \mathrm{lim}_k^{} S^0_k
        \longrightarrow
        \Omega S^1
        \xlongrightarrow{H}
        \Omega S^1
    $$
    where $H$ is the Hopf map in the EHP sequence. Hence
    $$
        \mathrm{lim}_k^{} S^0_k
        \simeq
        \mathrm{fib}\;H
        \simeq
        S^0
    $$
    Therefore the Goodwillie tower of $S^0$ converges at the prime $2$. This argument was suggested to us by Gijs Heuts.
\end{rem}

\section{Computations}

In this section we illustrate the computational power of the model for the approximation circles $S^1_k$ (thm. \ref{model}). The guiding theme is that many unstable computations can be recovered from simpler stable ones once the contribution of $S^1_k$ has been isolated and understood through the model.

We begin by revisiting Behrens' computation of $\pi_*S^1_k$ \cite[6.5.5]{Behrens2010}. Using our model, these groups can be derived directly from the Adams spectral sequence for symmetric product spectra, providing both a conceptual simplification and additional structural information.

We then turn to the metastable homotopy groups of spheres. Here the model leads to a resolution that is lighter than the classical Goodwillie tower, and to a spectral sequence whose first page and differentials are stable. This allows us to reinterpret the metastable computations of \cite{Behrens2010,Metastable} as a problem in stable homotopy theory.

Finally, we address the homotopy groups of $S^3$. Once again, the model for the approximation circles isolates the contribution of $S^1_k$, reducing the problem to the computation of the related approximation $W^1_k$, which has a shorter Goodwillie filtration.

\subsection{Comparison with Behrens' computations}

The first application of the model is to recover the computation of $\pi_*S^1_k$ at $p=2$, originally carried out by Behrens \cite[6.5.5]{Behrens2010}. In that work, the \textit{Goodwillie spectral sequence} is used to obtain the result. For the reader's convenience, we reproduce the computation in the table below.
\begin{center}
{\footnotesize The groups $\pi_{1+*}S^1_k$ as computed by Behrens \cite[6.5.5]{Behrens2010}} \\[0.5em]
\adjustbox{scale=0.46,center}{
\begin{tikzcd}[sep = small]
    \hline
	3 & {1(\infty)} \\
    \hline
	2 & {1(\infty)} &&&&&&&&& {\sigma[3, 1]} &&&&&&& {\theta_3[3, 1]} & {\theta_3[4, 1]} &&& {\nu^*[3, 1]} \\
    \hline
	1 & {1(\infty)} &&& {\nu[1]} &&& {\nu^2[1]} & {\sigma[1]} & {\sigma\eta[1]} & \begin{array}{c} \sigma\eta^2[1] \\ \sigma\eta[2] \\ \sigma[3] \end{array} &&&&& {\theta_3[1]} & \begin{array}{c} \theta_3[2] \\ \kappa[2] \end{array} & \begin{array}{c} \theta_3[3] \\ \eta_4[1] \end{array} & \begin{array}{c} \eta\eta_4[1] \\ \eta_4[2] \\ \theta_3[4] \end{array} & \begin{array}{c} \nu\kappa[2] \\ \nu^*[1] \end{array} & {\bar\sigma[1]} & \begin{array}{c} \bar\kappa[1] \\ \bar\sigma[2] \\ \nu^*[3] \end{array} \\
    \hline
	0 & {1(\infty)} & \eta & {\eta^2} & \begin{array}{c} 4\nu \\ 2\nu \\ \nu \end{array} &&& {\nu^2} & \begin{array}{c} 8\sigma \\ 4\sigma \\ 2\sigma \\ \sigma \end{array} & \begin{array}{c} \epsilon \\ \sigma\eta \end{array} & \begin{array}{c} \epsilon\eta \\ \sigma\eta^2 \\ \alpha_5 \end{array} & {\eta\alpha_5} & \begin{array}{c} \alpha_6 \\ \alpha_{6/2} \\ \alpha_{6/3} \end{array} &&& \begin{array}{c} \kappa \\ \theta_3 \end{array} & \begin{array}{c} \eta\kappa \\ \alpha_8 \\ \alpha_{8/2} \\ \alpha_{8/3} \\ \alpha_{8/4} \\ \alpha_{8/5} \end{array} & \begin{array}{c} \eta_4 \\ \eta\alpha_{8/5} \end{array} & \begin{array}{c} \eta\eta_4 \\ \eta^2\alpha_{8/5} \\ \nu\kappa \\ \alpha_9 \end{array} & \begin{array}{c} \eta\alpha_9 \\ 4\nu^* \\ 2\nu^* \\ \nu^* \end{array} & \begin{array}{c} \bar\sigma \\ \alpha_{10} \\ \alpha_{10/2}\\ \alpha_{10/3} \end{array} & \begin{array}{c} 4\bar\kappa \\ 2\bar\kappa \\ \bar\kappa  \end{array} \\
    \hline
	\sfrac{k}{*} & 0 & 1 & 2 & 3 & 4 & 5 & 6 & 7 & 8 & 9 & 10 & 11 & 12 & 13 & 14 & 15 & 16 & 17 & 18 & 19 & 20
\end{tikzcd}
}
\end{center}
Using the equivalence
\[
\Omega S^1_k \simeq \mathbb{Z} \times \Omega^{\infty+2k}\tau_{>0}\mathbb{Sp}^{p^k}
\]
of theorem \ref{model}, we can recover these homotopy groups directly from the Adams spectral sequence for the symmetric product spectra. For $k=0$, this reduces to the classical computation of $\pi_*\mathbb{S}$. For $k=3$, the result follows instead from a connectivity argument, since $c(1,4)=23>20$. We therefore carry out the computation only for $k=1$ and $k=2$, displaying the corresponding Adams spectral sequence charts below. For $k=1$, the first three nontrivial homotopy groups can also be read off from \cite{NakaokaSp2}.

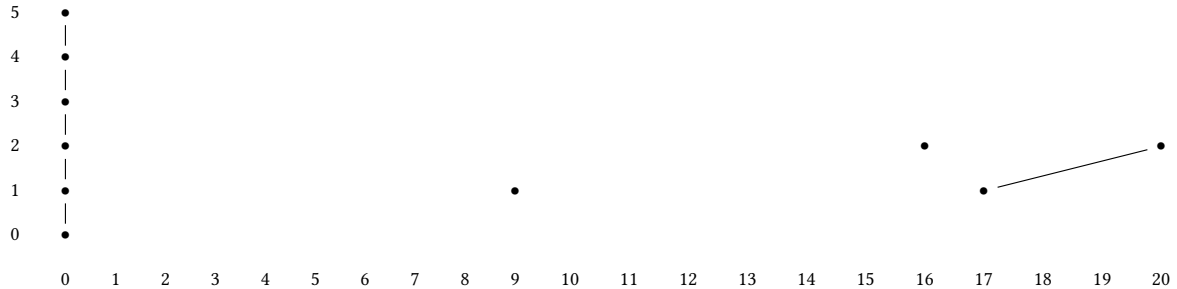
\begin{center}
{\footnotesize $E_2$-page of the ASS for $\mathbb{Z} \oplus \Sigma^{-2}\tau_{>0}\mathbb{Sp}^2$}
\adjustbox{scale=0.68,center}{
\begin{tikzcd}[sep=small]
	5 & \bullet \\
	4 & \bullet &&&&&&&&&&&&&&&&&& \bullet && \bullet \\
	3 & \bullet &&&&&&&&& \bullet &&&&&& \bullet && \bullet && \bullet & \bullet \\
	2 & \bullet &&&&&& \bullet && \bullet & \bullet &&&&& \bullet && {\bullet\bullet} & \bullet & \bullet && \bullet \\
	1 & \bullet &&& \bullet &&&& \bullet && \bullet &&&&&& \bullet && \bullet \\
	0 & \bullet \\
	& 0 & 1 & 2 & 3 & 4 & 5 & 6 & 7 & 8 & 9 & 10 & 11 & 12 & 13 & 14 & 15 & 16 & 17 & 18 & 19 & 20
	\arrow[no head, from=1-2, to=2-2]
	\arrow[no head, from=2-2, to=3-2]
	\arrow[no head, from=3-2, to=4-2]
	\arrow[no head, from=3-11, to=4-11]
	\arrow[no head, from=3-17, to=2-20]
	\arrow[no head, from=3-19, to=4-19]
	\arrow[no head, from=4-2, to=5-2]
	\arrow[no head, from=4-8, to=3-11]
	\arrow[no head, from=4-10, to=3-11]
	\arrow[no head, from=4-11, to=5-11]
	\arrow[no head, from=4-18, to=3-19]
	\arrow[no head, from=4-19, to=5-19]
	\arrow[no head, from=5-2, to=6-2]
	\arrow[no head, from=5-5, to=4-8]
	\arrow[no head, from=5-9, to=4-10]
	\arrow[no head, from=5-17, to=4-18]
	\arrow[no head, from=5-17, to=4-20]
	\arrow[no head, from=5-19, to=4-22]
\end{tikzcd}}
\end{center}

\begin{center}
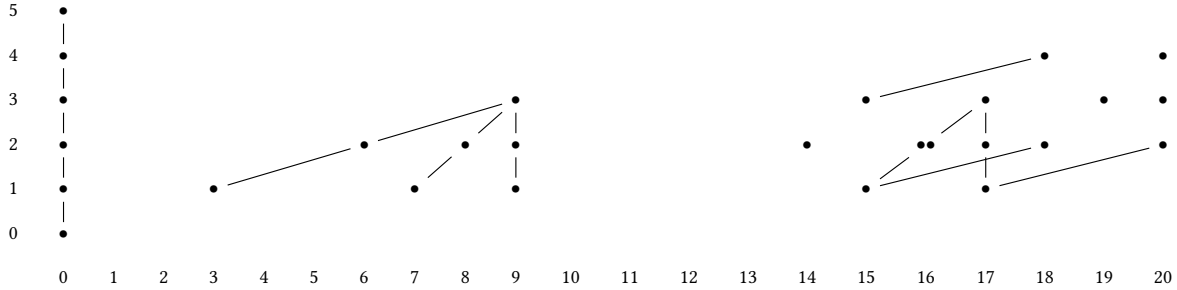
{\footnotesize $E_2$-page of the ASS for $\mathbb{Z} \oplus \Sigma^{-4}\tau_{>0}\mathbb{Sp}^4$}
\adjustbox{scale=0.68,center}{
% https://q.uiver.app/#q=WzAsMzcsWzEsNiwiMCJdLFsyLDYsIjEiXSxbMyw2LCIyIl0sWzQsNiwiMyJdLFs1LDYsIjQiXSxbNiw2LCI1Il0sWzcsNiwiNiJdLFs4LDYsIjciXSxbMCw1LCIwIl0sWzAsNCwiMSJdLFswLDMsIjIiXSxbMCwyLCIzIl0sWzksNiwiOCJdLFsxMCw2LCI5Il0sWzExLDYsIjEwIl0sWzEyLDYsIjExIl0sWzEzLDYsIjEyIl0sWzE0LDYsIjEzIl0sWzE1LDYsIjE0Il0sWzE2LDYsIjE1Il0sWzE3LDYsIjE2Il0sWzE4LDYsIjE3Il0sWzE5LDYsIjE4Il0sWzIwLDYsIjE5Il0sWzIxLDYsIjIwIl0sWzEsNSwiXFxidWxsZXQiXSxbMSw0LCJcXGJ1bGxldCJdLFsxLDMsIlxcYnVsbGV0Il0sWzEsMiwiXFxidWxsZXQiXSxbMCwxLCI0Il0sWzAsMCwiNSJdLFsxLDEsIlxcYnVsbGV0Il0sWzEsMCwiXFxidWxsZXQiXSxbMTAsNCwiXFxidWxsZXQiXSxbMTcsMywiXFxidWxsZXQiXSxbMTgsNCwiXFxidWxsZXQiXSxbMjEsMywiXFxidWxsZXQiXSxbMzIsMzEsIiIsMCx7InN0eWxlIjp7ImhlYWQiOnsibmFtZSI6Im5vbmUifX19XSxbMzEsMjgsIiIsMCx7InN0eWxlIjp7ImhlYWQiOnsibmFtZSI6Im5vbmUifX19XSxbMjgsMjcsIiIsMCx7InN0eWxlIjp7ImhlYWQiOnsibmFtZSI6Im5vbmUifX19XSxbMjcsMjYsIiIsMCx7InN0eWxlIjp7ImhlYWQiOnsibmFtZSI6Im5vbmUifX19XSxbMjYsMjUsIiIsMCx7InN0eWxlIjp7ImhlYWQiOnsibmFtZSI6Im5vbmUifX19XSxbMzUsMzYsIiIsMix7InN0eWxlIjp7ImhlYWQiOnsibmFtZSI6Im5vbmUifX19XV0=
\begin{tikzcd}[sep=small]
	5 & \bullet \\
	4 & \bullet \\
	3 & \bullet \\
	2 & \bullet &&&&&&&&&&&&&&&& \bullet &&&& \bullet \\
	1 & \bullet &&&&&&&&& \bullet &&&&&&&& \bullet \\
	0 & \bullet \\
	& 0 & 1 & 2 & 3 & 4 & 5 & 6 & 7 & 8 & 9 & 10 & 11 & 12 & 13 & 14 & 15 & 16 & 17 & 18 & 19 & 20
	\arrow[no head, from=1-2, to=2-2]
	\arrow[no head, from=2-2, to=3-2]
	\arrow[no head, from=3-2, to=4-2]
	\arrow[no head, from=4-2, to=5-2]
	\arrow[no head, from=5-2, to=6-2]
	\arrow[no head, from=5-19, to=4-22]
\end{tikzcd}}
\end{center}

In this range, there are no possible differentials, so the displayed pages already coincide with the $E_\infty$-term, and the resulting homotopy groups agree with those listed in the table above. The Adams spectral sequence has the additional advantage of recording all $2$-extensions, as well as the other extensions, thereby allowing us to recover the full group structure. Moreover, the computation is not intrinsically limited to these stems: in principle, it can be extended further, provided that the relevant differentials can be determined. Empirically, the homotopy groups of symmetric product spectra tend to be rather sparse, which makes such computations comparatively accessible.

We will not pursue further computations of $\pi_*S^1_k$ here, but we encourage interested readers to carry them out. To plot the Adams spectral sequences, we used Chatham's \href{https://spectralsequences.github.io/sseq/}{Adams spectral sequence calculator}. An example input file, used to compute $\mathbb{Sp}^2$ in the range considered here, is provided below and can serve as a template for the other computations.
\lstset{frame=tb,
  aboveskip=3mm,
  belowskip=3mm,
  showstringspaces=false,
  columns=flexible,
  basicstyle={\small\ttfamily},
  numbers=none,
  breaklines=true,
  breakatwhitespace=true,
  tabsize=3
}
\begin{center}
	{\footnotesize A sample input file for Chatham's Adams spectral sequence calculator, used here for $\mathbb{Sp}^2$}
\begin{lstlisting}
{
    "p": 2,
    "algebra": ["adem"],
    "type": "finitely presented module",
    "gens": {"x0": 0},
    "adem_relations": ["Sq1 x0", "Sq4 Sq2 x0", "Sq8 Sq2 x0", "Sq16 Sq2 x0", "Sq8 Sq4 x0", "Sq16 Sq4 x0", "Sq16 Sq8 x0"]
}
\end{lstlisting}
\end{center}

\subsection{Metastable computations}
\label{metastable}

The second application of the model is to the metastable homotopy groups of spheres, namely the computation of $\pi_*S^n_1$, at $p=2$. The main point is that the model gives a resolution of the metastable sphere which is smaller and more manageable than the standard Goodwillie tower, the modern name of the resolution underlying Mahowald's computations \cite{Metastable}. Indeed, the Goodwillie resolution gives
\[
\Omega^nS^n_1
\longrightarrow
\Omega^nS^n_0 \simeq \Omega^\infty\mathbb{S}
\longrightarrow
\Omega^{n-1}S^n(1) \simeq \Omega^\infty\mathbb{RP}^\infty_n
\]
By taking vertical fibers in the following diagram of fiber sequences
% https://q.uiver.app/#q=WzAsNixbMCwwLCJcXE9tZWdhXlxcaW5mdHkgXFxtYXRoYmJ7U30iXSxbMSwwLCJcXE9tZWdhXlxcaW5mdHkgXFxtYXRoYmJ7U30iXSxbMCwxLCJcXE9tZWdhXlxcaW5mdHkgXFxTaWdtYV57XFxpbmZ0eX1cXG1hdGhiYntSUH1eXFxpbmZ0eSJdLFsxLDEsIlxcT21lZ2FeXFxpbmZ0eSBcXFNpZ21hXntcXGluZnR5fVxcbWF0aGJie1JQfV5cXGluZnR5X24iXSxbMiwxLCJcXE9tZWdhXlxcaW5mdHkgXFxTaWdtYV57XFxpbmZ0eSsxfVxcbWF0aGJie1JQfV57bi0xfSJdLFsyLDAsIioiXSxbMCwxLCIiLDIseyJsZXZlbCI6Miwic3R5bGUiOnsiaGVhZCI6eyJuYW1lIjoibm9uZSJ9fX1dLFswLDJdLFsxLDNdLFsyLDNdLFszLDRdLFsxLDVdLFs1LDRdXQ==
\[\begin{tikzcd}
	{\Omega^\infty \mathbb{S}} & {\Omega^\infty \mathbb{S}} & {*} \\
	{\Omega^\infty \mathbb{RP}^\infty} & {\Omega^\infty \mathbb{RP}^\infty_n} & {\Omega^\infty \Sigma\mathbb{RP}^{n-1}}
	\arrow[equals, from=1-1, to=1-2]
	\arrow[from=1-1, to=2-1]
	\arrow[from=1-2, to=1-3]
	\arrow[from=1-2, to=2-2]
	\arrow[from=1-3, to=2-3]
	\arrow[from=2-1, to=2-2]
	\arrow[from=2-2, to=2-3]
\end{tikzcd}\]
and using the model (thm. \ref{model}), we obtain the sequence
\[
\Omega S^1_1 \simeq \mathbb{Z} \times \Omega^{\infty+2}\mathbb{Sp}^2
\longrightarrow
\Omega^nS^n_1
\longrightarrow
\Omega^\infty\mathbb{RP}^{n-1}
\]
This resolution is lighter in terms of homotopy groups, as we saw in the previous subsection, and involves less cancellation than the standard Goodwillie resolution. One can think of this as removing most of the ``noise'' responsible for differentials in the Goodwillie spectral sequence for $S^1$. This is also in line with the calculations of Behrens \cite{Behrens2010}, where most of the differentials in the Goodwillie spectral sequence for $S^n$ are shown to come from those for $S^1$.

We now show how to obtain a full metastable computation by purely stable methods. Consider the resolution
% https://q.uiver.app/#q=WzAsNyxbMCwwLCJcXE9tZWdhIFNeMV8xIl0sWzEsMCwiXFxjZG90cyJdLFsyLDAsIlxcT21lZ2FeblNebl8xIl0sWzMsMCwiXFxPbWVnYV57bisxfVNee24rMX1fMSJdLFszLDEsIlxcT21lZ2FeXFxpbmZ0eVxcbWF0aGJie1N9Xm4iXSxbNCwwLCJcXGNkb3RzIl0sWzUsMCwiXFxPbWVnYV5cXGluZnR5XFxtYXRoYmJ7U30iXSxbMCwxXSxbMSwyXSxbMiwzXSxbMyw0XSxbMyw1XSxbNSw2XV0=
\[\begin{tikzcd}
	{\Omega S^1_1} & \cdots & {\Omega^nS^n_1} & {\Omega^{n+1}S^{n+1}_1} & \cdots & {\Omega^\infty\mathbb{S}} \\
	&&& {\Omega^\infty\mathbb{S}^n}
	\arrow[from=1-1, to=1-2]
	\arrow[from=1-2, to=1-3]
	\arrow[from=1-3, to=1-4]
	\arrow[from=1-4, to=1-5]
	\arrow[from=1-4, to=2-4]
	\arrow[from=1-5, to=1-6]
\end{tikzcd}\]
The fibers are given by the calculus version of the EHP sequence (prop. \ref{EHPcalc}). Using the model, we obtain a spectral sequence
\[\begin{tikzcd}
	\pi_*\begin{array}{c} 
        \begin{Bmatrix}
         \Omega S^1_1 \simeq \mathbb{Z} \times \Omega^{\infty+2}\mathbb{Sp}^2 \\
         \mathbb{S}^n 
     \end{Bmatrix}_{n \geq 1}
     \end{array} 
     &\pi_*\mathbb{S}
	\arrow[Rightarrow, from=1-1, to=1-2]
\end{tikzcd}\]
Two remarks are immediate. First, truncating this spectral sequence at $1 \leq n < N$ gives convergence to $\Omega^NS^N_1$. Second, and more importantly, every ingredient in the construction can be worked out using purely \textit{stable} methods:
\begin{enumerate}
	\item The $E_1$-page consists entirely of stable homotopy groups: the first line records $\pi_*\mathbb{Sp}^2$, while all the other lines are copies of $\pi_*\mathbb{S}$.
	\item The portion of the spectral sequence below the first line coincides with the Atiyah--Hirzebruch spectral sequence for $\mathbb{RP}^\infty$, a stable computation carried out by Mahowald \cite{Metastable}.
	\item The remaining differentials, namely those hitting the first line, can also be understood stably. By the Whitehead conjecture, equivalently by the Kahn--Priddy theorem, every positive-degree class on the first line must eventually be killed by a differential. The problem is therefore to identify the sources of these differentials. This amounts to understanding the composite $\mathbb{Sp}^2 \to \Sigma\mathbb{RP}^\infty \to \Sigma\mathbb{RP}^\infty_n$, which is again a stable map. Concretely, for each class on the first line, we must determine the minimal $n$ for which it dies under this map. To this end, observe that the $E_2$-page of the Adams spectral sequence for $\mathbb{Sp}^2$ injects into that of $\mathbb{RP}^\infty$, by the algebraic Kahn--Priddy theorem \cite{AlgKP}. One can then track how these classes behave in the Adams spectral sequence for $\mathbb{RP}^\infty_n$. In the range of interest, every relevant class already dies in $\mathbb{RP}^\infty_5$, so any differential killing a class on the first line must have length at most four. To keep the exposition manageable, we omit the corresponding spectral sequences. For a quicker, but not entirely independent, determination, we compare with Behrens' computations from the previous subsection. For instance, the $17^{\mathrm{th}}$ stem of $S^1_1$ is generated by $\eta\eta_4[1]$, $\eta_4[2]$, and $\theta_3[4]$, and these are killed by differentials of lengths one, two, and four, respectively.
\end{enumerate}
We summarize the resulting spectral sequence up to the $18^{\mathrm{th}}$ stem below. The nontrivial differentials of length one, except those hitting the first line, are simply given by multiplication by $2$. Consequently, after the first page, the spectral sequence away from the first line becomes a vector space over $\mathbb{F}_2$. The longer differentials displayed in the figure should be interpreted as nontrivial linearly independent maps between these $\mathbb{F}_2$-vector spaces.
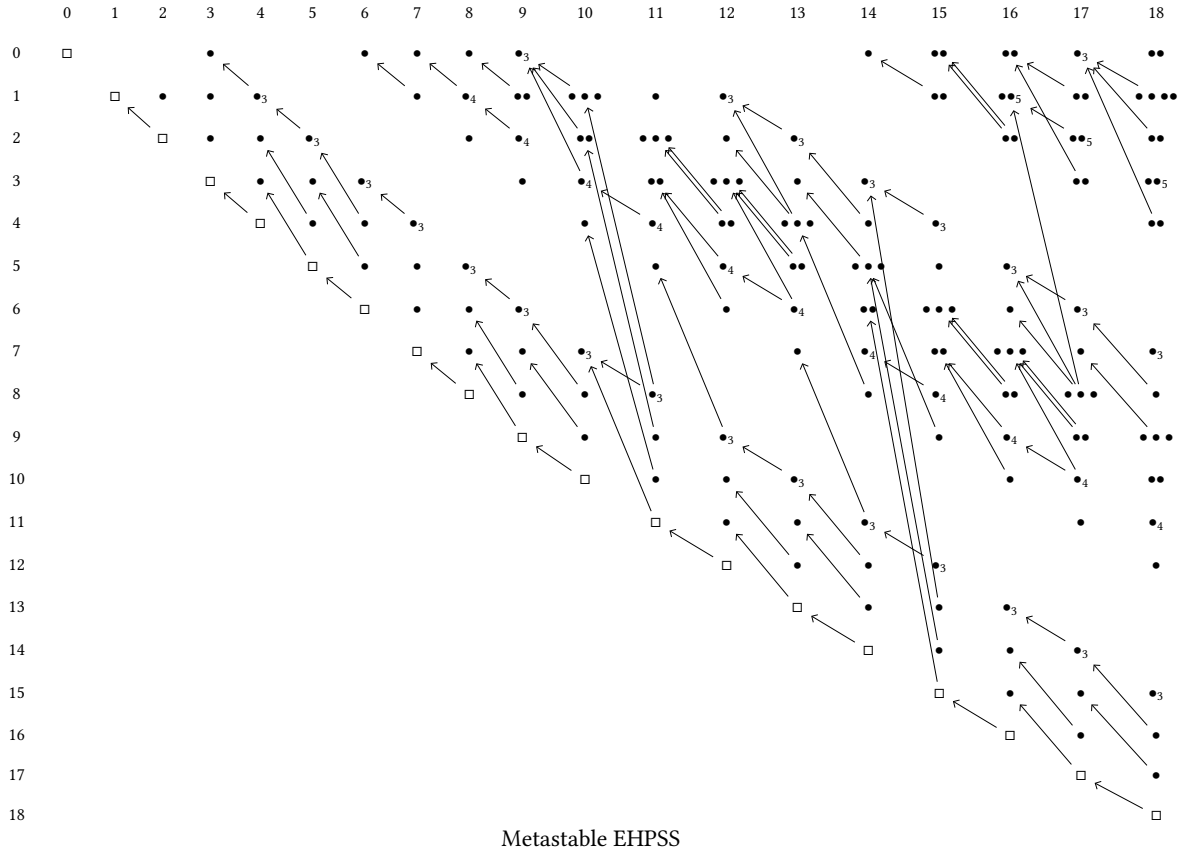
\begin{center}
\adjustbox{scale=0.61,center}{
\begin{tikzcd}[sep=small]
	& 0 & 1 & 2 & 3 & 4 & 5 & 6 & 7 & 8 & 9 & 10 & 11 & 12 & 13 & 14 & 15 & 16 & 17 & 18 \\
	0 & \square &&& \bullet &&& \bullet & \bullet & \bullet & {\bullet_3} &&&&& \bullet & {\bullet \bullet} & {\bullet \bullet} & {\bullet_3} & {\bullet \bullet} \\
	1 && \square & \bullet & \bullet & {\bullet_3} &&& \bullet & {\bullet_4} & {\bullet \bullet} & {\bullet \bullet \bullet} & \bullet & {\bullet_3} &&& {\bullet \bullet} & {\bullet \bullet_5} & {\bullet \bullet} & {\bullet \bullet \bullet \bullet} \\
	2 &&& \square & \bullet & \bullet & {\bullet_3} &&& \bullet & {\bullet_4} & {\bullet \bullet} & {\bullet \bullet \bullet} & \bullet & {\bullet_3} &&& {\bullet \bullet} & {\bullet \bullet_5} & {\bullet \bullet} \\
	3 &&&& \square & \bullet & \bullet & {\bullet_3} &&& \bullet & {\bullet_4} & {\bullet \bullet} & {\bullet \bullet \bullet} & \bullet & {\bullet_3} &&& {\bullet \bullet} & {\bullet \bullet_5} \\
	4 &&&&& \square & \bullet & \bullet & {\bullet_3} &&& \bullet & {\bullet_4} & {\bullet \bullet} & {\bullet \bullet \bullet} & \bullet & {\bullet_3} &&& {\bullet \bullet} \\
	5 &&&&&& \square & \bullet & \bullet & {\bullet_3} &&& \bullet & {\bullet_4} & {\bullet \bullet} & {\bullet \bullet \bullet} & \bullet & {\bullet_3} \\
	6 &&&&&&& \square & \bullet & \bullet & {\bullet_3} &&& \bullet & {\bullet_4} & {\bullet \bullet} & {\bullet \bullet \bullet} & \bullet & {\bullet_3} \\
	7 &&&&&&&& \square & \bullet & \bullet & {\bullet_3} &&& \bullet & {\bullet_4} & {\bullet \bullet} & {\bullet \bullet \bullet} & \bullet & {\bullet_3} \\
	8 &&&&&&&&& \square & \bullet & \bullet & {\bullet_3} &&& \bullet & {\bullet_4} & {\bullet \bullet} & {\bullet \bullet \bullet} & \bullet \\
	9 &&&&&&&&&& \square & \bullet & \bullet & {\bullet_3} &&& \bullet & {\bullet_4} & {\bullet \bullet} & {\bullet \bullet \bullet} \\
	10 &&&&&&&&&&& \square & \bullet & \bullet & {\bullet_3} &&& \bullet & {\bullet_4} & {\bullet \bullet} \\
	11 &&&&&&&&&&&& \square & \bullet & \bullet & {\bullet_3} &&& \bullet & {\bullet_4} \\
	12 &&&&&&&&&&&&& \square & \bullet & \bullet & {\bullet_3} &&& \bullet \\
	13 &&&&&&&&&&&&&& \square & \bullet & \bullet & {\bullet_3} \\
	14 &&&&&&&&&&&&&&& \square & \bullet & \bullet & {\bullet_3} \\
	15 &&&&&&&&&&&&&&&& \square & \bullet & \bullet & {\bullet_3} \\
	16 &&&&&&&&&&&&&&&&& \square & \bullet & \bullet \\
	17 &&&&&&&&&&&&&&&&&& \square & \bullet \\
	18 &&&&&&&&&&&&&&&&&&& \square
	\arrow[from=3-6, to=2-5]
	\arrow[from=3-9, to=2-8]
	\arrow[from=3-10, to=2-9]
	\arrow[from=3-11, to=2-10]
	\arrow[from=3-12, to=2-11]
	\arrow[from=3-17, to=2-16]
	\arrow[from=3-19, to=2-18]
	\arrow[from=3-20, to=2-19]
	\arrow[from=4-4, to=3-3]
	\arrow[from=4-7, to=3-6]
	\arrow[from=4-11, to=3-10]
	\arrow[from=4-12, to=2-11]
	\arrow[from=4-15, to=3-14]
	\arrow[from=4-18, to=2-17]
	\arrow[shift right, from=4-18, to=2-17]
	\arrow[from=4-19, to=3-18]
	\arrow[from=4-20, to=2-19]
	\arrow[from=5-12, to=2-11]
	\arrow[from=5-19, to=2-18]
	\arrow[from=6-6, to=5-5]
	\arrow[from=6-7, to=4-6]
	\arrow[from=6-8, to=4-7]
	\arrow[from=6-9, to=5-8]
	\arrow[from=6-13, to=5-12]
	\arrow[from=6-14, to=4-13]
	\arrow[shift right, from=6-14, to=4-13]
	\arrow[from=6-15, to=3-14]
	\arrow[from=6-15, to=4-14]
	\arrow[from=6-16, to=4-15]
	\arrow[from=6-17, to=5-16]
	\arrow[from=6-20, to=2-19]
	\arrow[from=7-7, to=5-6]
	\arrow[from=7-8, to=5-7]
	\arrow[from=7-14, to=5-13]
	\arrow[from=7-15, to=5-14]
	\arrow[shift right, from=7-15, to=5-14]
	\arrow[from=7-16, to=5-15]
	\arrow[from=8-8, to=7-7]
	\arrow[from=8-11, to=7-10]
	\arrow[from=8-14, to=5-13]
	\arrow[from=8-15, to=5-14]
	\arrow[from=8-15, to=7-14]
	\arrow[from=8-19, to=7-18]
	\arrow[from=10-10, to=9-9]
	\arrow[from=10-11, to=8-10]
	\arrow[from=10-12, to=8-11]
	\arrow[from=10-13, to=3-12]
	\arrow[from=10-13, to=9-12]
	\arrow[from=10-16, to=6-15]
	\arrow[from=10-17, to=9-16]
	\arrow[from=10-18, to=8-17]
	\arrow[shift right, from=10-18, to=8-17]
	\arrow[from=10-19, to=3-18]
	\arrow[from=10-19, to=7-18]
	\arrow[from=10-19, to=8-18]
	\arrow[from=10-20, to=8-19]
	\arrow[from=11-11, to=9-10]
	\arrow[from=11-12, to=9-11]
	\arrow[from=11-13, to=4-12]
	\arrow[from=11-14, to=7-13]
	\arrow[from=11-17, to=7-16]
	\arrow[from=11-18, to=9-17]
	\arrow[from=11-19, to=9-18]
	\arrow[shift right, from=11-19, to=9-18]
	\arrow[from=11-20, to=9-19]
	\arrow[from=12-12, to=11-11]
	\arrow[from=12-13, to=6-12]
	\arrow[from=12-15, to=11-14]
	\arrow[from=12-18, to=9-17]
	\arrow[from=12-19, to=9-18]
	\arrow[from=12-19, to=11-18]
	\arrow[from=13-13, to=9-12]
	\arrow[from=13-16, to=9-15]
	\arrow[from=14-14, to=13-13]
	\arrow[from=14-15, to=12-14]
	\arrow[from=14-16, to=12-15]
	\arrow[from=14-17, to=13-16]
	\arrow[from=15-15, to=13-14]
	\arrow[from=15-16, to=13-15]
	\arrow[from=15-17, to=5-16]
	\arrow[from=16-16, to=15-15]
	\arrow[from=16-17, to=7-16]
	\arrow[from=16-19, to=15-18]
	\arrow[from=17-17, to=8-16]
	\arrow[from=18-18, to=17-17]
	\arrow[from=18-19, to=16-18]
	\arrow[from=18-20, to=16-19]
	\arrow[from=19-19, to=17-18]
	\arrow[from=19-20, to=17-19]
	\arrow[from=20-20, to=19-19]
\end{tikzcd}}
{\footnotesize Metastable EHPSS} 
\end{center}
Unlike in the previous subsection, we have not determined the $2$-extensions, and hence not the full group structure, of the metastable homotopy groups. The reason is simple: at present, we do not know how to do this. As with the differentials, the problem splits into two parts: understanding the $2$-extensions among the lower rows, and understanding those involving the first row. The first reduces to studying the $2$-extensions in $\mathbb{RP}^n$, a problem we are confident can be resolved. The second, however, remains mysterious to us.
\begin{ques}
	What is the group structure of $\pi_*S^n_1$? In particular, what are the $2$-extensions that involve the first line of the spectral sequence above?
\end{ques}

\subsection{The homotopy groups of $S^3$}

The third application of the model is to simplify the computation of the homotopy groups of $S^3$. By the generalized Freudenthal theorem (thm. \ref{genfreudenthal}), it suffices to compute the groups of $S^3_k$, which fit into the fiber sequence
\[
W^1_k \longrightarrow S^1_k \longrightarrow \Omega^2S^3_k
\]
In the range of convergence, $S^1_k$ contributes only a single copy of $\mathbb{F}_p$, and this class does not survive in $S^3_k$. Hence the problem reduces to computing the homotopy groups of $W^1_k$. The advantage of $W^1_k$ over $S^3_k$ is that, while they converge to essentially the same space, $W^1_k$ has a Goodwillie filtration of one stage shorter. For instance, the metastable approximation of $S^3$ sits in a fiber sequence involving two spectra connected by an unstable map, whereas the metastable approximation of $W^1$ is itself a spectrum. In this sense, $W^1_k$ provides a smaller and more tractable model. Explicit computations will be given at the end of the section.

We begin by recording the low-dimensional homotopy groups of $S^1_k$.

\begin{lem}
    The first three nontrivial homotopy groups of $S^1_k$ are given by
    \[
    \pi_*S^1_k \cong
    \begin{cases}
        \mathbb{Z} & * = 1 \\
        \mathbb{F}_p & * = 2(p^{k+1}-1)-2k \\
		\mathbb{F}_p & * = N(p^{k+1}-1)-2k \\
        0 & * \leq N(p^{k+1}-1)-2k, \text{otherwise}
    \end{cases}
    \]
	where $N=3$ if $p=2$ and $N=4$ if $p>2$.
\end{lem}

\begin{proof}
    By theorem \ref{model}, there is an equivalence
    \[
    \Omega S^1_k \simeq \mathbb{Z} \times \Omega^{\infty+2k}\tau_{>0}\mathbb{Sp}^{p^k}
    \]
    Therefore, for $*>1$,
    \[
    \pi_*S^1_k \cong \pi_{*+2k-1}\mathbb{Sp}^{p^k}
    \]
    The claim now follows from lemma \ref{homotopysp}, which identifies the first two nontrivial positive homotopy groups of $\mathbb{Sp}^{p^k}$ as occurring in degrees $2(p^{k+1}-1)-1$ and $N(p^{k+1}-1)-1$, where $N=3$ if $p=2$ and $N=4$ if $p>2$.
\end{proof}

For $p>2$, the second nontrivial homotopy group is already beyond the radius of convergence of $S^3_k$, so the first one is the only one that can contribute. We will now see that this remaining class double suspends to zero, leading to the following proposition.

\begin{prop}
	\label{homotopyS3}
	For $p$ an odd prime and $k>0$, the homotopy groups of $S^3$ are given by
	\[
    \pi_{3+*}S^3 \cong 
    \begin{cases}
    \mathbb{Z} & * = 0\\
    \pi_*W^1_k \ /\  \mathbb{F}_p & * = 2(p^{k+1}-1) - 2k\\
    \pi_*W^1_k & * < 4(p^{k+1}-1) - 2(k+1), \text{otherwise}
    \end{cases}
    \]
\end{prop}

\begin{rem}
	In the exceptional degree, the notation $\pi_*W^1_k \ /\  \mathbb{F}_p$ should be interpreted as saying that $\pi_{3+*}S^3$ is an $\mathbb{F}_p$-vector space of dimension one less than $\pi_*W^1_k$.
\end{rem}

\begin{proof}
	By the generalized Freudenthal theorem (thm. \ref{genfreudenthal}), the comparison map $S^3 \to S^3_k$ induces an isomorphism on $\pi_{3+*}$ for
	\[
	* < c(3,k+1)-3 = 4(p^{k+1}-1)-2(k+1)
	\]
	It therefore suffices to compute $\pi_{3+*}S^3_k$ in this range. Consider the fiber sequence
	\[
	\Omega S^1_k \longrightarrow \Omega^3S^3_k \longrightarrow W^1_k
	\]
	By the previous lemma, the only nontrivial homotopy group of $\Omega S^1_k$ in the range under consideration occurs in degree $2(p^{k+1}-1)-2k-1$ and is isomorphic to $\mathbb{F}_p$. Indeed, since $p$ is odd, the next nontrivial homotopy group of $S^1_k$ occurs in degree $4(p^{k+1}-1)-2k$, which lies beyond the range of convergence $*<4(p^{k+1}-1)-2(k+1)$. It follows that, away from the degrees
\[
* = 2(p^{k+1}-1)-2k-1,\quad 2(p^{k+1}-1)-2k
\]
the long exact sequence in homotopy identifies
\[
\pi_{3+*}S^3_k \cong \pi_*W^1_k
\]
in the stated range. In the exceptional degree $*=2(p^{k+1}-1)-2k$, the long exact sequence gives
\[
0 \longrightarrow \pi_{*}\Omega^3S^3_k
\longrightarrow \pi_{*}W^1_k
\longrightarrow \mathbb{F}_p
\longrightarrow \pi_{*-1}\Omega^3S^3_k
\longrightarrow \pi_{*-1}W^1_k
\longrightarrow 0
\]
We claim that the rightmost map is an isomorphism. Since we are in the range of convergence, this map may be identified with the comparison map
\[
\pi_{*-1}W^1 \longrightarrow \pi_{*-1}W^1_k
\]
Moreover, $*- 1=c(1,k+1)$, so this map is an isomorphism precisely when $W^1_k$ converges early, in the sense of subsection \ref{earlyconv}. This holds for $k>0$ by proposition \ref{earlyconvW}. Therefore the long exact sequence reduces to a short exact sequence
\[
    0
    \longrightarrow
    \pi_{3+*}S^3
    \longrightarrow
    \pi_{*}W^1_k
    \longrightarrow
    \mathbb{F}_p
    \longrightarrow
    0
\]
This is a short exact sequence of $\mathbb{F}_p$-vector spaces, hence it splits. Indeed, since $* > 0$, the group $\pi_{3+*}S^3$ is an $\mathbb{F}_p$-vector space by Cohen--Moore--Neisendorfer. The same holds for $\pi_*W^1_k$, because multiplication by $p$ on $W$ is null, and hence also on $W^1_k$ by functoriality.
\end{proof}

The proposition above reduces the unstable computation to the study of $W^1_k$. In the metastable range, corresponding to $k=1$, this description becomes especially simple, since $W^1_1$ is equivalent to a Moore spectrum. As a result, the homotopy groups of $S^3$ can be expressed directly in terms of stable homotopy.

\begin{cor}
    In the metastable range, for $p$ an odd prime, the homotopy groups of $S^3$ are given by
    \[
    \pi_{3+*}S^3 \cong 
    \begin{cases}
    \mathbb{Z} & * = 0\\
    \pi_{*-2p+3}(\mathbb{S}/p) \ /\  \mathbb{F}_p & * = 2(p^2-2)\\
    \pi_{*-2p+3}(\mathbb{S}/p) & * < 4(p^2-2), \text{otherwise}
    \end{cases}
    \]
\end{cor}

\begin{proof}
    By example \ref{cohomW1}, there is an equivalence
    \[
    W^1_1 \simeq \Omega^\infty(\mathbb{S}^{2p-3}/p)
    \]
    The claim therefore follows from the previous proposition applied with $k=1$.
\end{proof}

For the sake of completeness, let us now verify that this does indeed reduce the problem to something computable.

\begin{ex}
\label{metastableS3}
Consider the case $p=3$. The corollary expresses $\pi_{3+*}S^3$ for $*<28$ in terms of $\pi_*(\mathbb{S}/3)$ for $*<25$. Since $\mathbb{S}/p$ has exponent $p$ for $p>2$, we have
\[
\pi_*(\mathbb{S}/p) \cong (\pi_*\mathbb{S})/p \oplus (\pi_{*-1}\mathbb{S})[p]
\]
Thus it suffices to know $\pi_*\mathbb{S}$ for $*<25$, which can be read off from \cite[table A.3.4 + im J]{Complexcob}. Alternatively, one can work directly with the Adams spectral sequence for $\mathbb{S}/3$. In that range, the only nontrivial differentials have source in bidegrees $(18,2)$, $(19,2)$, $(22,3)$, and $(23,3)$, and they all have length $2$. We summarize the relevant data below.

\adjustbox{scale=0.78,center}{
	\begin{tabular}{|c|c|c|c|c|c|c|c|c|c|c|c|c|c|c|c|c|c|c|c|c|c|c|c|c|c|c|c|c} 
	\hline
 	$*$ & 0 & 1 & 2 & 3 & 4 & 5 & 6 & 7 & 8 & 9 & 10 & 11 & 12 & 13 & 14 & 15 & 16 & 17 & 18 & 19 & 20 & 21 & 22 & 23 & 24 \\ \hline
	$\pi_{*}\mathbb{S}$ & $\square$ & & & $\bullet$ & & & & $\bullet$ & & & $\bullet$ & $\bullet_2$ & & $\bullet$ & & $\bullet$ & & & & $\bullet$ & $\bullet$ & & & $\bullet \bullet_2$ &  \\ \hline
	$\pi_{*}\mathbb{S}/3$ & $\bullet$ & & & $\bullet$ & $\bullet$ & & & $\bullet$ & $\bullet$ & & $\bullet$ & $\bullet \bullet$ & $\bullet$ & $\bullet$ & $\bullet$ & $\bullet$ & $\bullet$ & & & $\bullet$ & $\bullet \bullet$ & $\bullet$ & & $\bullet \bullet$ & $\bullet \bullet$ \\ \hline
	\end{tabular}
	}

	The homotopy groups of $S^3$ can then be read off from the last line by removing one copy in degree $11$, shifting everything three degrees to the right, and adding a copy of $\mathbb{Z}$ in degree $0$, as summarized below.

	\adjustbox{scale=0.70,center}{
	\begin{tabular}{|c|c|c|c|c|c|c|c|c|c|c|c|c|c|c|c|c|c|c|c|c|c|c|c|c|c|c|c|c|} 
	\hline
 	$*$ & 0 & 1 & 2 & 3 & 4 & 5 & 6 & 7 & 8 & 9 & 10 & 11 & 12 & 13 & 14 & 15 & 16 & 17 & 18 & 19 & 20 & 21 & 22 & 23 & 24 & 25 & 26 & 27 \\ \hline
	$\pi_{3+*}S^3$ & $\square$ & & & $\bullet$ & & & $\bullet$ & $\bullet$ & & & $\bullet$ & $\bullet$ & & $\bullet$ & $\bullet$ & $\bullet$ & $\bullet$ & $\bullet$ & $\bullet$ & $\bullet$ & & & $\bullet$ & $\bullet \bullet$ & $\bullet$ & & $\bullet \bullet$ & $\bullet \bullet$ \\ \hline
	\end{tabular}
	}

	The same argument applies at every odd prime, provided that $\pi_*\mathbb{S}$ is known. For example, the case $p=5$ can likewise be read off from \cite[table A.3.5 + im J]{Complexcob}.
\end{ex}

\begin{ex}
	\label{todasn}
	The result above also yields the following computation of $\pi_{n+*}S^n$ for $n$ odd, $p$ an odd prime, and $*<4(p^2-2)$:
	\[
    \pi_{n+*}S^n \cong 
    \begin{cases}
    \mathbb{Z} & * = 0\\
    \pi_*\Sigma^\infty B\Sigma_p^{(p-1)(n-1)} \ /\ \mathbb{F}_p & * = 2(p^2-2)\\
    \pi_*\Sigma^\infty B\Sigma_p^{(p-1)(n-1)} & * < 4(p^2-2), \text{otherwise}
    \end{cases}
    \]
	While this does not provide a complete calculation in the full \textit{metastable range} of $S^n$, it does yield a complete calculation in \textit{Toda's original range}, namely $*<20$. In other words, Toda's classical computations at odd primes can be recovered here by purely stable methods. For the reader's convenience, we recall that
	\[
	\mathbb{F}_p^*B\Sigma_p^{(p-1)(n-1)} \cong \Sigma^{-1}\mathbb{F}_p\langle \beta^\epsilon P^i \mid i \leq (n-1)/2\rangle
	\]
\end{ex}

The next approximation, corresponding to $k=2$, provides a way to compute the homotopy groups of $S^3$ up to the $97^{\mathrm{th}}$ stem at $p=3$. This goes beyond Toda's final range \cite{Toda3}, which stops at the $79^{\mathrm{th}}$ stem. The approximation $W^1_2$ has two nontrivial layers. The first is again $W^1(1) \simeq \mathbb{S}^{2p-3}/p$, whose homotopy groups are known in this range. The second is $W^1(2)$, whose cohomology as an $\mathcal{A}$-module is also known; see ex. \ref{cohomW2}. In particular, one can run an Adams spectral sequence for it; see appendix \ref{ASSW2}. In our range, only a few differentials can occur, and there are no extension problems since $W^1(2)$ also has exponent $p$. The remaining problem is therefore to understand the differential between $W^1(1)$ and $W^1(2)$. The curious reader can find a tentative sketch of the Goodwillie spectral sequence at this \href{https://q.uiver.app/#q=WzAsMjg4LFsyLDAsIjAiXSxbMywwLCIxIl0sWzQsMCwiMiJdLFs1LDAsIjMiXSxbNiwwLCI0Il0sWzcsMCwiNSJdLFs4LDAsIjYiXSxbOSwwLCI3Il0sWzEwLDAsIjgiXSxbMTEsMCwiOSJdLFsxMiwwLCIxMCJdLFsxMywwLCIxMSJdLFsxNCwwLCIxMiJdLFsxNSwwLCIxMyJdLFsxNiwwLCIxNCJdLFsxNywwLCIxNSJdLFsxOCwwLCIxNiJdLFsxOSwwLCIxNyJdLFsyMCwwLCIxOCJdLFs1LDMsIlxcb3JhbmdlXFxidWxsZXQiXSxbOSwzLCJcXG9yYW5nZVxcYnVsbGV0Il0sWzgsMywiXFxvcmFuZ2UgXFxidWxsZXQiXSxbMTIsMywiXFxvcmFuZ2VcXGJ1bGxldCJdLFsxNSwzLCJcXGJ1bGxldCJdLFsxMywzLCJcXG9yYW5nZVxcYnVsbGV0Il0sWzE2LDMsIlxcb3JhbmdlXFxidWxsZXRcXGJ1bGxldCJdLFsxNywzLCJcXG9yYW5nZVxcYnVsbGV0Il0sWzE4LDMsIlxcYnVsbGV0Il0sWzE5LDMsIlxcYnVsbGV0Il0sWzIwLDMsIlxcb3JhbmdlXFxidWxsZXQiXSxbMjEsMCwiMTkiXSxbMjEsMywiXFxvcmFuZ2VcXGJ1bGxldCJdLFsyNCwzLCJcXG9yYW5nZVxcYnVsbGV0Il0sWzI2LDMsIlxcYnVsbGV0Il0sWzI1LDMsIlxcb3JhbmdlXFxidWxsZXRcXGJ1bGxldCJdLFsyOCwzLCJcXG9yYW5nZVxcYnVsbGV0XFxidWxsZXQiXSxbMjksMywiXFxvcmFuZ2VcXGJ1bGxldFxcYnVsbGV0Il0sWzMxLDMsIlxcYnVsbGV0Il0sWzMyLDMsIlxcb3JhbmdlXFxidWxsZXRcXGJ1bGxldCJdLFszMywzLCJcXG9yYW5nZVxcYnVsbGV0Il0sWzM0LDMsIlxcYnVsbGV0Il0sWzM1LDMsIlxcYnVsbGV0XFxidWxsZXQiXSxbMzYsMywiXFxvcmFuZ2VcXGJ1bGxldFxcYnVsbGV0Il0sWzIyLDAsIjIwIl0sWzIzLDAsIjIxIl0sWzI0LDAsIjIyIl0sWzI1LDAsIjIzIl0sWzI2LDAsIjI0Il0sWzI3LDAsIjI1Il0sWzI4LDAsIjI2Il0sWzI5LDAsIjI3Il0sWzMwLDAsIjI4Il0sWzMxLDAsIjI5Il0sWzMyLDAsIjMwIl0sWzMzLDAsIjMxIl0sWzM0LDAsIjMyIl0sWzM1LDAsIjMzIl0sWzM2LDAsIjM0Il0sWzM3LDMsIlxcb3JhbmdlXFxidWxsZXQiXSxbNDAsMywiXFxvcmFuZ2VcXGJ1bGxldCJdLFs0MSwzLCJcXG9yYW5nZVxcYnVsbGV0XFxidWxsZXQiXSxbNDIsMywiXFxidWxsZXRcXGJ1bGxldCJdLFs0MywzLCJcXGJ1bGxldFxcYnVsbGV0Il0sWzQ0LDMsIlxcb3JhbmdlXFxidWxsZXRcXGJ1bGxldFxcYnVsbGV0Il0sWzQ0LDAsIjQyIl0sWzQyLDAsIjQwIl0sWzM5LDAsIjM3Il0sWzQxLDAsIjM5Il0sWzQzLDAsIjQxIl0sWzQ1LDMsIlxcb3JhbmdlXFxidWxsZXRcXGJ1bGxldFxcYnVsbGV0Il0sWzQ2LDMsIlxcYnVsbGV0Il0sWzQ3LDMsIlxcYnVsbGV0Il0sWzQ4LDMsIlxcb3JhbmdlXFxidWxsZXRcXGJ1bGxldCJdLFs0OSwzLCJcXG9yYW5nZVxcYnVsbGV0Il0sWzUwLDMsIlxcYnVsbGV0Il0sWzUxLDMsIlxcYnVsbGV0XFxidWxsZXQiXSxbNTIsMywiXFxvcmFuZ2VcXGJ1bGxldFxcYnVsbGV0XFxidWxsZXQiXSxbNTIsMCwiNTAiXSxbNTMsMywiXFxvcmFuZ2VcXGJ1bGxldFxcYnVsbGV0Il0sWzU0LDMsIlxcYnVsbGV0Il0sWzU1LDMsIlxcYnVsbGV0XFxidWxsZXQiXSxbNTYsMywiXFxvcmFuZ2VcXGJ1bGxldFxcYnVsbGV0Il0sWzU3LDMsIlxcb3JhbmdlXFxidWxsZXRcXGJ1bGxldCJdLFs1OCwzLCJcXGJ1bGxldCJdLFs2MCwzLCJcXG9yYW5nZVxcYnVsbGV0XFxidWxsZXQiXSxbNjAsMCwiNTgiXSxbNjEsMywiXFxvcmFuZ2VcXGJ1bGxldFxcYnVsbGV0Il0sWzY0LDMsIlxcb3JhbmdlXFxidWxsZXQiXSxbNjUsMywiXFxvcmFuZ2VcXGJ1bGxldCJdLFs2NywzLCJcXGJ1bGxldCJdLFs2OCwzLCJcXG9yYW5nZVxcYnVsbGV0XFxidWxsZXQiXSxbNjgsMCwiNjYiXSxbNjksMywiXFxvcmFuZ2VcXGJ1bGxldCJdLFs3MCwzLCJcXGJ1bGxldCJdLFs3MSwzLCJcXGJ1bGxldCJdLFs3MiwzLCJcXG9yYW5nZVxcYnVsbGV0Il0sWzczLDMsIlxcb3JhbmdlXFxidWxsZXRcXGJ1bGxldCJdLFs3NCwzLCJcXGJ1bGxldCJdLFs3NiwzLCJcXG9yYW5nZVxcYnVsbGV0Il0sWzc2LDAsIjc0Il0sWzc3LDMsIlxcb3JhbmdlXFxidWxsZXRcXGJ1bGxldCJdLFs3OCwzLCJcXGJ1bGxldCJdLFs3OSwzLCJcXGJ1bGxldCJdLFs4MCwzLCJcXG9yYW5nZVxcYnVsbGV0XFxidWxsZXRcXGJ1bGxldCJdLFs4MSwzLCJcXG9yYW5nZVxcYnVsbGV0XFxidWxsZXQiXSxbODMsMywiXFxidWxsZXQiXSxbODQsMywiXFxvcmFuZ2VcXGJ1bGxldFxcYnVsbGV0Il0sWzg0LDAsIjgyIl0sWzM4LDAsIjM2Il0sWzM3LDAsIjM1Il0sWzQwLDAsIjM4Il0sWzQ2LDAsIjQ0Il0sWzQ3LDAsIjQ1Il0sWzQ1LDAsIjQzIl0sWzQ5LDAsIjQ3Il0sWzQ4LDAsIjQ2Il0sWzUwLDAsIjQ4Il0sWzUxLDAsIjQ5Il0sWzY1LDAsIjYzIl0sWzcyLDAsIjcwIl0sWzczLDAsIjcxIl0sWzc5LDAsIjc3Il0sWzUzLDAsIjUxIl0sWzU0LDAsIjUyIl0sWzU1LDAsIjUzIl0sWzU2LDAsIjU0Il0sWzU3LDAsIjU1Il0sWzU4LDAsIjU2Il0sWzU5LDAsIjU3Il0sWzYxLDAsIjU5Il0sWzYyLDAsIjYwIl0sWzYzLDAsIjYxIl0sWzY0LDAsIjYyIl0sWzY2LDAsIjY0Il0sWzY3LDAsIjY1Il0sWzY5LDAsIjY3Il0sWzcwLDAsIjY4Il0sWzcxLDAsIjY5Il0sWzc4LDAsIjc2Il0sWzc3LDAsIjc1Il0sWzc0LDAsIjcyIl0sWzc1LDAsIjczIl0sWzgwLDAsIjc4Il0sWzgxLDAsIjc5Il0sWzgyLDAsIjgwIl0sWzgzLDAsIjgxIl0sWzE1LDIsIlxcYnVsbGV0Il0sWzMwLDIsIlxcYnVsbGV0Il0sWzMxLDIsIlxcYnVsbGV0Il0sWzM4LDIsIlxcYnVsbGV0Il0sWzM5LDIsIlxcYnVsbGV0Il0sWzQxLDIsIlxcYnVsbGV0Il0sWzQyLDIsIlxcYnVsbGV0XFxidWxsZXQiXSxbNDMsMiwiXFxidWxsZXQiXSxbNDYsMiwiXFxidWxsZXQiXSxbNDcsMiwiXFxidWxsZXQiXSxbNDksMiwiXFxidWxsZXQiXSxbNTAsMiwiXFxidWxsZXQiXSxbNDksMSwiXFxidWxsZXQiXSxbMiwxLCJcXHNxdWFyZSJdLFs4NSwwLCI4MyJdLFs5MiwwLCI5MCJdLFsxMDEsMCwiOTkiXSxbMTAxLDEsIlxcYnVsbGV0Il0sWzEwMiwwLCIxMDAiXSxbODYsMCwiODQiXSxbODcsMCwiODUiXSxbODgsMCwiODYiXSxbODksMCwiODciXSxbOTAsMCwiODgiXSxbOTEsMCwiODkiXSxbOTMsMCwiOTEiXSxbOTQsMCwiOTIiXSxbOTUsMCwiOTMiXSxbOTYsMCwiOTQiXSxbOTcsMCwiOTUiXSxbOTgsMCwiOTYiXSxbOTksMCwiOTciXSxbMTAwLDAsIjk4Il0sWzYxLDIsIlxcYnVsbGV0Il0sWzYyLDIsIlxcYnVsbGV0XFxidWxsZXQiXSxbNjMsMiwiXFxidWxsZXQiXSxbNjQsMiwiXFxidWxsZXQiXSxbNjUsMiwiXFxidWxsZXQiXSxbNzIsMiwiXFxidWxsZXQiXSxbNzMsMiwiXFxidWxsZXQiXSxbODEsMiwiXFxidWxsZXQiXSxbODIsMiwiXFxidWxsZXRcXGJ1bGxldCJdLFs4MywyLCJcXGJ1bGxldCJdLFs4NSwyLCJcXGJ1bGxldFxccmVkXFxidWxsZXQiXSxbODYsMiwiXFxidWxsZXRcXGJ1bGxldFxccmVkXFxidWxsZXQiXSxbODcsMiwiXFxidWxsZXRcXGJ1bGxldCJdLFs4OCwyLCJcXGJ1bGxldFxcYnVsbGV0Il0sWzg5LDIsIlxcYnVsbGV0XFxidWxsZXQiXSxbOTAsMiwiXFxidWxsZXRcXGJ1bGxldCJdLFs5MSwyLCJcXGJ1bGxldCJdLFsxMDIsMiwiXFxidWxsZXQiXSxbMTAxLDIsIlxcYnVsbGV0Il0sWzAsNSwiXFxPbWVnYV4zU14zXzIiXSxbMiw1LCJcXHNxdWFyZSJdLFs1LDUsIlxcYnVsbGV0Il0sWzgsNSwiXFxidWxsZXQiXSxbOSw1LCJcXGJ1bGxldCJdLFsxMiw1LCJcXGJ1bGxldCJdLFsxMyw1LCJcXGJ1bGxldCJdLFsxNSw1LCJcXGJ1bGxldCJdLFsxNiw1LCJcXGJ1bGxldCJdLFsxNyw1LCJcXGJ1bGxldCJdLFsxOCw1LCJcXGJ1bGxldCJdLFsxOSw1LCJcXGJ1bGxldCJdLFsyMCw1LCJcXGJ1bGxldCJdLFsyMSw1LCJcXGJ1bGxldCJdLFsyNCw1LCJcXGJ1bGxldCJdLFsyNiw1LCJcXGJ1bGxldCJdLFsyNSw1LCJcXGJ1bGxldFxcYnVsbGV0Il0sWzI4LDUsIlxcYnVsbGV0XFxidWxsZXQiXSxbMjksNSwiXFxidWxsZXRcXGJ1bGxldCJdLFszMiw1LCJcXGJ1bGxldCJdLFszMyw1LCJcXGJ1bGxldCJdLFszNCw1LCJcXGJ1bGxldCJdLFszNSw1LCJcXGJ1bGxldFxcYnVsbGV0Il0sWzM2LDUsIlxcYnVsbGV0XFxidWxsZXQiXSxbMzcsNSwiXFxidWxsZXQiXSxbMzgsNSwiXFxidWxsZXQiXSxbMzksNSwiXFxidWxsZXQiXSxbNDAsNSwiXFxidWxsZXQiXSxbNDEsNSwiXFxidWxsZXRcXGJ1bGxldCJdLFs0Miw1LCJcXGJ1bGxldCJdLFs0NCw1LCJcXGJ1bGxldFxcYnVsbGV0Il0sWzQ1LDUsIlxcYnVsbGV0XFxidWxsZXRcXGJ1bGxldCJdLFs0Niw1LCJcXGJ1bGxldCJdLFs0OCw1LCJcXGJ1bGxldCJdLFs0OSw1LCJcXGJ1bGxldCJdLFs1MSw1LCJcXGJ1bGxldFxcYnVsbGV0Il0sWzUyLDUsIlxcYnVsbGV0XFxidWxsZXRcXGJ1bGxldCJdLFs1Myw1LCJcXGJ1bGxldFxcYnVsbGV0Il0sWzU0LDUsIlxcYnVsbGV0Il0sWzU1LDUsIlxcYnVsbGV0XFxidWxsZXQiXSxbNTYsNSwiXFxidWxsZXRcXGJ1bGxldCJdLFs1Nyw1LCJcXGJ1bGxldFxcYnVsbGV0Il0sWzU4LDUsIlxcYnVsbGV0Il0sWzYyLDUsIlxcYnVsbGV0XFxidWxsZXQiXSxbNjMsNSwiXFxidWxsZXQiXSxbNjQsNSwiXFxidWxsZXQiXSxbNjUsNSwiXFxidWxsZXQiXSxbNjcsNSwiXFxidWxsZXQiXSxbNzAsNSwiXFxidWxsZXQiXSxbNzEsNSwiXFxidWxsZXQiXSxbNzgsNSwiXFxidWxsZXQiXSxbODEsNSwiXFxidWxsZXRcXGJ1bGxldCJdLFs2OCw1LCJcXGJ1bGxldCJdLFs2MCw1LCJcXGJ1bGxldCJdLFs2MSw1LCJcXGJ1bGxldFxcYnVsbGV0Il0sWzc3LDUsIlxcYnVsbGV0Il0sWzgwLDUsIlxcYnVsbGV0Il0sWzg1LDMsIlxcb3JhbmdlXFxidWxsZXQiXSxbODYsMywiXFxidWxsZXRcXGJ1bGxldCJdLFs4NywzLCJcXGJ1bGxldFxcYnVsbGV0XFxidWxsZXQiXSxbODgsMywiXFxvcmFuZ2VcXGJ1bGxldFxcYnVsbGV0Il0sWzg5LDMsIlxcb3JhbmdlXFxidWxsZXRcXGJ1bGxldFxcYnVsbGV0Il0sWzkwLDMsIlxcYnVsbGV0XFxidWxsZXRcXGJ1bGxldCJdLFs5MSwzLCJcXGJ1bGxldFxcYnVsbGV0Il0sWzkyLDMsIlxcb3JhbmdlXFxidWxsZXRcXGJ1bGxldCJdLFs5MywzLCJcXG9yYW5nZVxcYnVsbGV0Il0sWzk1LDMsIlxcYnVsbGV0Il0sWzk2LDMsIlxcb3JhbmdlXFxidWxsZXRcXGJ1bGxldFxcYnVsbGV0XFxidWxsZXQiXSxbOTcsMywiXFxvcmFuZ2VcXGJ1bGxldFxcYnVsbGV0XFxidWxsZXRcXGJ1bGxldCJdLFs5OCwzLCJcXGJ1bGxldFxcYnVsbGV0XFxidWxsZXQiXSxbOTksMywiXFxvcmFuZ2VcXGJ1bGxldFxcYnVsbGV0XFxidWxsZXQiXSxbMTAwLDMsIlxcb3JhbmdlXFxidWxsZXRcXGJ1bGxldCJdLFsxMDEsMywiXFxidWxsZXQiXSxbMCwxLCJcXE9tZWdhIFNeMV8yIFxcc2ltZXEgXFxtYXRoYmJ7Wn0gXFx0aW1lcyBcXE9tZWdhXntcXGluZnR5ICs0fVNwXjkiXSxbMCwyLCJXXjEoMikiXSxbMCwzLCJXXjEoMSkgXFxzaW1lcSBcXG1hdGhiYntTfV4zLzMiXSxbNTgsNywiVG9kYSAgXFxcXCBcXG1vZCBcXG9yYW5nZVxcYWxwaGEiXSxbMiw3LCJUb2RhIFxcXFwgb3JpZ2luYWwiXSxbODEsNywiVG9kYSBcXFxcIGVuZCJdLFs5MiwyLCJcXGJ1bGxldFxccmVkXFxidWxsZXQiXSxbOTMsMiwiXFxidWxsZXRcXGJ1bGxldFxyZWRcXGJ1bGxldFxyZWRcXGJ1bGxldCJdLFs5NCwyLCJcXGJ1bGxldFxcYnVsbGV0XFxyZWRcXGJ1bGxldCJdLFs5NSwyLCJcXGJ1bGxldFxyZWRcXGJ1bGxldCJdLFs5NiwyLCJcXGJ1bGxldFxyZWRcXGJ1bGxldFxyZWRcXGJ1bGxldCJdLFs5OSwyLCJcXGJ1bGxldCJdLFs5OCwyLCJcXGJ1bGxldFxyZWRcXGJ1bGxldCJdLFs5NywyLCJcXGJ1bGxldFxyZWRcXGJ1bGxldFxyZWRcXGJ1bGxldCJdLFs3OCwyLCJcXGJ1bGxldCJdLFs3OSwyLCJcXGJ1bGxldCJdLFswLDksIkxlZ2VuZDogXFxcXCBcXG9yYW5nZSBcXGJ1bGxldCA9IHZfMS1wZXJpb2RpYyBcXFxcIFxccmVkXFxidWxsZXQgPSBwb3NzaWJseSBcXCBkZWFkIl0sWzI1LDE0Nl0sWzM3LDE0N10sWzM4LDE0OF0sWzYxLDE1MV0sWzYyLDE1Ml0sWzYzLDE1M10sWzcxLDE1NF0sWzcyLDE1NV0sWzc0LDE1Nl0sWzE1NywxNThdLFs2MiwxNTIsIiIsMix7Im9mZnNldCI6LTF9XSxbOTYsMTg0XSxbOTcsMTg1XSxbMjc1LDE5OSwiIiwwLHsic2hvcnRlbiI6eyJ0YXJnZXQiOjUwfX1dLFsyNzQsMjQwLCIiLDAseyJzaG9ydGVuIjp7InRhcmdldCI6NTB9fV0sWzI3NiwyNDksIiIsMCx7InNob3J0ZW4iOnsidGFyZ2V0Ijo1MH19XSxbMTAzLDI4Nl0sWzEwMiwyODVdXQ==}{link}.

\section{Future directions}

The model for the approximation circles turned out to be very effective for unstable computations. It allowed us to recover several classical results and to push some low-dimensional computations further, but many phenomena remain only partially understood.

In this section we sketch several directions in which the model may offer new insights. We begin by discussing the apparent vanishing of classes coming from $S^1_k$ and formulate a broader structural conjecture. We then turn to synthetic perspectives, where our space-level results suggest a possible analogue at the level of Adams spectral sequences or unstable synthetic categories. Another direction comes from chromatic homotopy theory: the symmetric product spectra have strong chromatic properties, suggesting that the model may admit a chromatic reinterpretation. Beyond this, the homotopy fixed points of the approximation spheres interpolate between the Segal and Sullivan conjectures, making the approximation circles a natural testing ground for understanding such interpolations. Finally, we suggest that an analogous Whitehead conjecture in unitary calculus could shed light on a problem posed by Neisendorfer.

\subsection{Survival of elements from $S^1_k$}

The metastable computations at $p = 2$ in subsection \ref{metastable} reveal a curious phenomenon: none of the torsion classes in $\pi_*S^1_1$ survives to $\pi_*S^n$ in the metastable range. This does not mean that no such class survives to $\pi_*S^n_1$; for example, $\sigma\eta[2] \in \pi_{11}S^2_1$ does survive \cite[6.5.5]{Behrens2010}. Rather, the classes that survive to $\pi_*S^n_1$ seem to be detected only \textit{beyond the radius of convergence}. Empirically, for $0 < * < c(n,2)-n$, one finds
\[
\pi_{n+*}S^n \cong \frac{\pi_*\mathbb{RP}^{n-1}}{\pi_{*+1}\mathbb{Sp}^2}
\]
The same phenomenon appears in the computations of the homotopy groups of $S^3$ (prop. \ref{homotopyS3}): the only class coming from $S^1_k$, for $k > 0$, dies. Thus, for $0 < * < c(3,k+1)-3$, one obtains
\[
\pi_{3+*}S^3 \cong \frac{\pi_*W^1_k}{\pi_{*+2k-1}\mathbb{Sp}^{p^k}}
\]
We wonder whether this behavior persists for all spheres and all approximations with $k > 0$. Denoting by $\hat S^n_k$ the fiber of the map $S^1_k \to \Omega^{n-1}S^n_k$, we formulate the following conjecture.

\begin{conj}
    For $k > 0$, in the range of convergence $0 < * < c(n,k+1)-n$, one has
    \[
    \pi_{n+*}S^n \cong \frac{\pi_*\hat S^n_k}{\pi_{*+2k-1}\mathbb{Sp}^{p^k}}
    \]
\end{conj}

This is roughly equivalent to asking that the map $\Omega S^1_k \to \Omega^nS^n_k$ be null on $\pi_*$ in that range. In positive degrees, since the map $\Omega^\infty S^1(k) \to S^1_k$ admits a section, this map factors as
\[
\Omega S^1_k \simeq \mathbb{Z} \times \Omega^{\infty+2k}\tau_{>0}\mathbb{Sp}^{p^k}
\longrightarrow
\Omega^{\infty+1}S^1(k) \simeq \Omega^{\infty+2k}\mathbb{Sp}^{p^k}/\mathbb{Sp}^{p^{k-1}}
\longrightarrow
\Omega^{\infty+n}S^n(k)
\longrightarrow
\Omega^nS^n_k
\]
Except for the rightmost map, all maps in this factorization are obtained by applying $\Omega^\infty$ to maps of spectra. These maps are induced by the composite
\[
\Sigma^{-2k}\tau_{>0}\mathbb{Sp}^{p^k}
\longrightarrow
\Sigma^{-2k}\mathbb{Sp}^{p^k}/\mathbb{Sp}^{p^{k-1}}
\simeq
\Sigma^{-1}S^1(k)
\longrightarrow
\Sigma^{-n}S^n(k)
\]
Denote by $M(n,k)$ the first degree in which this map is nontrivial on $\pi_*$. We can then formulate the following stronger version of the conjecture.

\begin{conj}
	$M(n, k) \geq c(n, k+1) - n$ for $k > 0$.
\end{conj}

It is immediate that $M(n,k) \geq c(n,k)-n$, since $c(n,k)-n$ is the connectivity of $\Sigma^{-n}S^n(k)$. At $p = 2$ and $k = 1$, the map above is equivalent to
\[
\Sigma^{-2}\tau_{>0}\mathbb{Sp}^2
\longrightarrow
\Sigma^{-1}\mathbb{RP}^\infty
\longrightarrow
\Sigma^{-1}\mathbb{RP}^\infty_n
\]
From the metastable computations we find that $M(1,1)=3$, $M(2,1)=M(3,1)=9$, $M(4,1)=17$, and $M(n,1)>20$ for $n>4$.

Finally, let us emphasize that there \textit{do} exist classes in $\pi_*S^1_k$ which detect elements in $\pi_*S^n$. For instance, the class $\nu\kappa[2] \in \pi_{19}S^1_1$ survives to $\pi_{20}S^2_1$, and then further to $\pi_{20}S^2_2 \cong \pi_{20}S^2$ \cite[6.5.6]{Behrens2010}. This does not contradict the conjecture: the point is that this class detects an element that converges in the \textit{next} approximation.

\subsection{Synthetic perspectives}

Many of the computations in this paper use the Adams spectral sequence as a tool, but the results themselves are purely space-level statements. It is therefore natural to ask whether they can be lifted to statements about the ASS, or, better, to some unstable synthetic category, such as those being developed by Mor \cite{Mor} and by Balderrama--Pstrągowski \cite{BP1, BP2}. At the moment, it is not even clear to us what the right notion of Adams spectral sequence for the approximation spheres $S^n_k$ should be. Such a spectral sequence should interpolate between the unstable Adams spectral sequence of $S^n$ and the ASS of $\mathbb{S}^n$. Nevertheless, our computations suggest what some of its features might be.

For example, the metastable computations, together with the discussion in the previous subsection, show that for $* < 19$ one has
\[
\pi_{7+*}S^7 \cong \pi_*\mathbb{Z} \oplus \frac{\pi_*\mathbb{RP}^6}{\pi_{*+1}\mathbb{Sp}^2}
\]
This appears to lift to a statement relating the UASS of $S^7$ to the ASS of the spectra involved. Namely, in the same range, one finds
\[
\mathrm{UASS}^{*,*}(S^7)
\cong
\mathrm{ASS}^{*,*}\mathbb{Z}
\oplus
\frac{\mathrm{ASS}^{*,*-1}\mathbb{RP}^6}{\mathrm{ASS}^{*+1,*-1}\mathbb{Sp}^2}
\]
Indeed, plotted below is the ASS for $\mathbb{Z} \oplus \mathbb{RP}^6$, with the ASS of $\mathbb{Sp}^2$ highlighted in red. The remaining black part coincides with the displayed portion of the UASS of $S^7$, as one can check in \cite[page 10]{UASScharts}.
\begin{center}
	{\footnotesize $E_2$-page of the ASS for $\mathbb{Z} \oplus \mathbb{RP}^6$, with the part corresponding to $\mathbb{Sp}^2$ highlighted in red.} \\[0.5em]
\adjustbox{scale=0.65,center}{
\begin{tikzcd}[sep=small]
	11 & \bullet &&&&&&&&&&&&&&&&&&& \bullet \\
	10 & \bullet &&&&&&&&&&&&&&&&&& \bullet & \bullet \\
	9 & \bullet &&&&&&&&&&&&&&&&& \bullet & \bullet & \bullet \\
	8 & \bullet &&&&&&&&&&&&&&& \bullet && {\bullet\bullet} & \bullet \\
	7 & \bullet &&&&&&&&&&& \bullet &&&& \bullet & {\bullet\bullet} & \bullet & \bullet \\
	6 & \bullet &&&&&&&&&& \bullet & \bullet &&& {\ \ \bullet} & \bullet & {\bullet\bullet} & \bullet &&& \bullet \\
	5 & \bullet &&&&&&&&& \bullet & \bullet & \bullet &&& {\ \ \bullet} & {\bullet\bullet} && \bullet & \bullet & {\textcolor{red}\bullet} & \bullet \\
	4 & \bullet &&&&&&& \bullet && {\bullet\bullet} & {\bullet\textcolor{red}\bullet} &&& \bullet & {\bullet\bullet} & \bullet & {\textcolor{red}\bullet} && {\textcolor{red}\bullet\bullet} && {\textcolor{red}\bullet\bullet} \\
	3 & \bullet &&& \bullet &&&& {\textcolor{red}\bullet\bullet} & {\bullet\bullet} & {\bullet\textcolor{red}\bullet} & {\bullet\textcolor{red}\bullet} & \bullet &&& {\bullet \ \ } & {\textcolor{red}\bullet} && {\textcolor{red}\bullet\textcolor{red}\bullet} & {\textcolor{red}\bullet} & {\textcolor{red}\bullet\bullet} \\
	2 & \bullet && \bullet & \bullet & {\textcolor{red}\bullet} && \bullet & \bullet & {\bullet\textcolor{red}\bullet} && {\textcolor{red}\bullet} &&&&&& {\textcolor{red}\bullet} && {\textcolor{red}\bullet} \\
	1 & \bullet & \bullet && \bullet \\
	0 & \bullet \\
	& 0 & 1 & 2 & 3 & 4 & 5 & 6 & 7 & 8 & 9 & 10 & 11 & 12 & 13 & 14 & 15 & 16 & 17 & 18 & 19 & 20
	\arrow[no head, from=1-2, to=2-2]
	\arrow[no head, from=1-21, to=2-21]
	\arrow[no head, from=2-2, to=3-2]
	\arrow[no head, from=2-20, to=1-21]
	\arrow[no head, from=2-21, to=3-21]
	\arrow[no head, from=3-2, to=4-2]
	\arrow[no head, from=3-19, to=2-20]
	\arrow[no head, from=3-20, to=4-20]
	\arrow[no head, from=4-2, to=5-2]
	\arrow[no head, from=4-17, to=5-17]
	\arrow[no head, from=4-19, to=3-20]
	\arrow[no head, from=4-20, to=5-20]
	\arrow[no head, from=5-2, to=6-2]
	\arrow[no head, from=5-13, to=6-13]
	\arrow[no head, from=5-17, to=6-17]
	\arrow[no head, from=5-18, to=4-19]
	\arrow[shift left, no head, from=5-18, to=4-19]
	\arrow[no head, from=5-19, to=6-19]
	\arrow[no head, from=6-2, to=7-2]
	\arrow[no head, from=6-12, to=5-13]
	\arrow[no head, from=6-13, to=7-13]
	\arrow[shift left, no head, from=6-16, to=7-16]
	\arrow[no head, from=6-17, to=5-18]
	\arrow[no head, from=6-18, to=5-19]
	\arrow[no head, from=6-19, to=7-19]
	\arrow[no head, from=7-2, to=8-2]
	\arrow[no head, from=7-11, to=6-12]
	\arrow[no head, from=7-12, to=8-12]
	\arrow[shift left, no head, from=7-16, to=8-16]
	\arrow[no head, from=7-17, to=6-18]
	\arrow[shift left, no head, from=7-17, to=6-18]
	\arrow[no head, from=7-22, to=6-22]
	\arrow[no head, from=8-2, to=9-2]
	\arrow[no head, from=8-11, to=7-12]
	\arrow[shift left, no head, from=8-11, to=7-12]
	\arrow[no head, from=8-12, to=9-12]
	\arrow[shift left, color={rgb,255:red,214;green,92;blue,92}, no head, from=8-12, to=9-12]
	\arrow[no head, from=8-16, to=7-17]
	\arrow[shift right, no head, from=8-16, to=9-16]
	\arrow[no head, from=8-20, to=7-20]
	\arrow[color={rgb,255:red,214;green,92;blue,92}, no head, from=8-20, to=9-20]
	\arrow[no head, from=8-22, to=7-22]
	\arrow[no head, from=9-2, to=10-2]
	\arrow[no head, from=9-5, to=10-5]
	\arrow[no head, from=9-9, to=8-9]
	\arrow[no head, from=9-10, to=8-11]
	\arrow[shift left, no head, from=9-10, to=8-11]
	\arrow[color={rgb,255:red,214;green,92;blue,92}, no head, from=9-11, to=8-12]
	\arrow[shift left, color={rgb,255:red,214;green,92;blue,92}, no head, from=9-12, to=10-12]
	\arrow[no head, from=9-16, to=8-17]
	\arrow[color={rgb,255:red,214;green,92;blue,92}, no head, from=9-19, to=8-20]
	\arrow[color={rgb,255:red,214;green,92;blue,92}, no head, from=9-20, to=10-20]
	\arrow[no head, from=10-2, to=11-2]
	\arrow[no head, from=10-4, to=9-5]
	\arrow[no head, from=10-5, to=11-5]
	\arrow[no head, from=10-9, to=9-9]
	\arrow[no head, from=10-9, to=9-10]
	\arrow[color={rgb,255:red,214;green,92;blue,92}, no head, from=10-10, to=9-11]
	\arrow[shift left, no head, from=10-10, to=9-11]
	\arrow[color={rgb,255:red,214;green,92;blue,92}, no head, from=10-18, to=9-19]
	\arrow[no head, from=11-2, to=12-2]
	\arrow[no head, from=11-3, to=10-4]
\end{tikzcd}
}
\end{center}

The same pattern appears in the other computations carried out in this part of the paper, even when not all homotopy groups of $\mathbb{Sp}^2$ die. Consider, for example, the metastable computations of $S^3$. In the figure below, we display the ASS of $\mathbb{Z} \oplus \mathbb{RP}^2$ in black and that of $\mathbb{Sp}^2$ in red. Every differential appearing in the metastable computations can be lifted consistently to a differential between these two spectral sequences, shown in green. We expect the resulting $E_2$-page to coincide with the $E_2$-page of any sensible notion of ASS for $S^3_1$. In the range of convergence, namely for $* < 8$, it indeed coincides with the UASS for $S^3$ \cite[page 8]{UASScharts}.

\begin{center}
	{\footnotesize $E_2$-page of the ASS for $\mathbb{Z} \oplus \mathbb{RP}^2$ in black and for $\mathbb{Sp}^2$ in red, with the induced differentials shown in green.} \\[0.5em]
\adjustbox{scale=0.72,center}
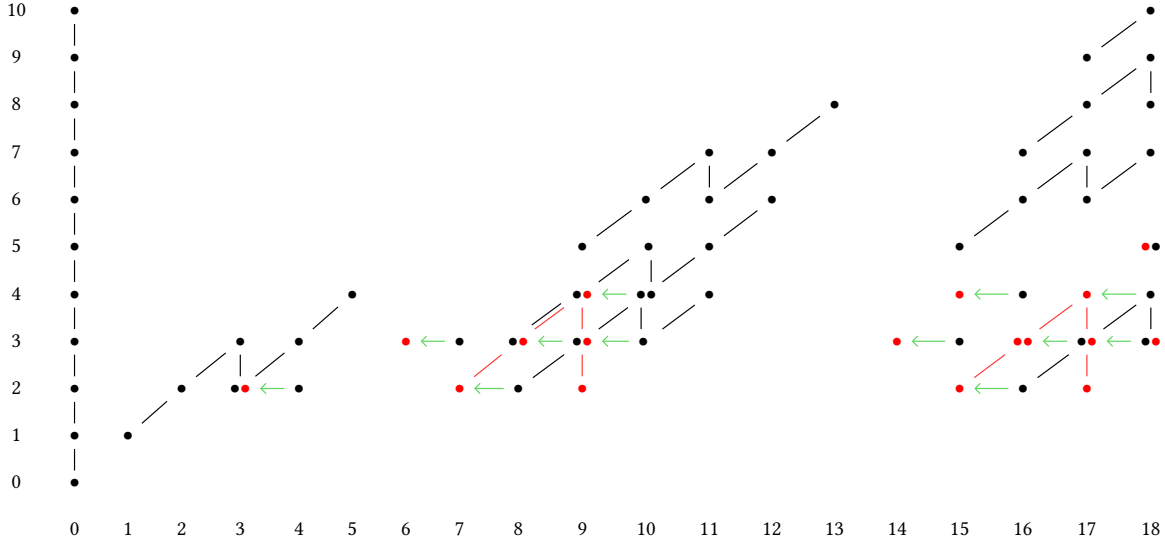
{
\begin{tikzcd}[sep=small]
	10 & \bullet &&&&&&&&&&&&&&&&&& \bullet \\
	9 & \bullet &&&&&&&&&&&&&&&&& \bullet & \bullet \\
	8 & \bullet &&&&&&&&&&&&& \bullet &&&& \bullet & \bullet \\
	7 & \bullet &&&&&&&&&&& \bullet & \bullet &&&& \bullet & \bullet & \bullet \\
	6 & \bullet &&&&&&&&&& \bullet & \bullet & \bullet &&&& \bullet & \bullet \\
	5 & \bullet &&&&&&&&& \bullet & {\ \bullet} & \bullet &&&& \bullet &&& {\textcolor{red}\bullet\bullet} \\
	4 & \bullet &&&&& \bullet &&&& {\bullet\textcolor{red}\bullet} & {\bullet\bullet} & \bullet &&&& {\textcolor{red}\bullet} & \bullet & {\textcolor{red}\bullet} & \bullet \\
	3 & \bullet &&& \bullet & \bullet && {\textcolor{red}\bullet} & \bullet & {\bullet\textcolor{red}\bullet} & {\bullet\textcolor{red}\bullet} & {\bullet \ } &&&& {\textcolor{red}\bullet} & \bullet & {\textcolor{red}\bullet\textcolor{red}\bullet} & {\bullet\textcolor{red}\bullet} & {\bullet\textcolor{red}\bullet} \\
	2 & \bullet && \bullet & {\bullet\textcolor{red}\bullet} & \bullet &&& {\textcolor{red}\bullet} & \bullet & {\textcolor{red}\bullet} &&&&&& {\textcolor{red}\bullet} & \bullet & {\textcolor{red}\bullet} \\
	1 & \bullet & \bullet \\
	0 & \bullet \\
	& 0 & 1 & 2 & 3 & 4 & 5 & 6 & 7 & 8 & 9 & 10 & 11 & 12 & 13 & 14 & 15 & 16 & 17 & 18
	\arrow[no head, from=1-20, to=2-19]
	\arrow[no head, from=2-2, to=1-2]
	\arrow[no head, from=2-2, to=3-2]
	\arrow[no head, from=2-20, to=3-20]
	\arrow[no head, from=3-2, to=4-2]
	\arrow[no head, from=3-19, to=2-20]
	\arrow[no head, from=4-2, to=5-2]
	\arrow[no head, from=4-13, to=5-13]
	\arrow[no head, from=4-14, to=3-15]
	\arrow[no head, from=4-18, to=3-19]
	\arrow[no head, from=4-19, to=5-19]
	\arrow[no head, from=5-2, to=6-2]
	\arrow[no head, from=5-12, to=4-13]
	\arrow[no head, from=5-13, to=4-14]
	\arrow[no head, from=5-18, to=4-19]
	\arrow[no head, from=5-19, to=4-20]
	\arrow[no head, from=6-2, to=7-2]
	\arrow[no head, from=6-11, to=5-12]
	\arrow[no head, from=6-13, to=5-14]
	\arrow[no head, from=6-17, to=5-18]
	\arrow[no head, from=7-2, to=8-2]
	\arrow[no head, from=7-11, to=6-12]
	\arrow[draw={rgb,255:red,255;green,54;blue,51}, no head, from=7-11, to=8-11]
	\arrow[shift right, no head, from=7-12, to=6-12]
	\arrow[no head, from=7-12, to=6-13]
	\arrow[draw={rgb,255:red,92;green,214;blue,92}, from=7-12, to=7-11]
	\arrow[draw={rgb,255:red,92;green,214;blue,92}, from=7-18, to=7-17]
	\arrow[draw={rgb,255:red,255;green,54;blue,51}, no head, from=7-19, to=8-19]
	\arrow[draw={rgb,255:red,92;green,214;blue,92}, from=7-20, to=7-19]
	\arrow[no head, from=7-20, to=8-20]
	\arrow[no head, from=8-2, to=9-2]
	\arrow[no head, from=8-5, to=9-5]
	\arrow[no head, from=8-6, to=7-7]
	\arrow[draw={rgb,255:red,92;green,214;blue,92}, from=8-9, to=8-8]
	\arrow[draw={rgb,255:red,255;green,54;blue,51}, no head, from=8-10, to=7-11]
	\arrow[shift left, no head, from=8-10, to=7-11]
	\arrow[no head, from=8-11, to=7-12]
	\arrow[draw={rgb,255:red,92;green,214;blue,92}, from=8-11, to=8-10]
	\arrow[draw={rgb,255:red,255;green,54;blue,51}, no head, from=8-11, to=9-11]
	\arrow[shift left, no head, from=8-12, to=7-12]
	\arrow[no head, from=8-12, to=7-13]
	\arrow[draw={rgb,255:red,92;green,214;blue,92}, from=8-12, to=8-11]
	\arrow[draw={rgb,255:red,92;green,214;blue,92}, from=8-17, to=8-16]
	\arrow[draw={rgb,255:red,255;green,54;blue,51}, no head, from=8-18, to=7-19]
	\arrow[no head, from=8-19, to=7-20]
	\arrow[draw={rgb,255:red,92;green,214;blue,92}, from=8-19, to=8-18]
	\arrow[draw={rgb,255:red,255;green,54;blue,51}, no head, from=8-19, to=9-19]
	\arrow[draw={rgb,255:red,92;green,214;blue,92}, from=8-20, to=8-19]
	\arrow[no head, from=9-2, to=10-2]
	\arrow[no head, from=9-4, to=8-5]
	\arrow[no head, from=9-5, to=8-6]
	\arrow[draw={rgb,255:red,92;green,214;blue,92}, from=9-6, to=9-5]
	\arrow[draw={rgb,255:red,255;green,54;blue,51}, no head, from=9-9, to=8-10]
	\arrow[no head, from=9-10, to=8-11]
	\arrow[draw={rgb,255:red,92;green,214;blue,92}, from=9-10, to=9-9]
	\arrow[draw={rgb,255:red,255;green,54;blue,51}, no head, from=9-17, to=8-18]
	\arrow[no head, from=9-18, to=8-19]
	\arrow[draw={rgb,255:red,92;green,214;blue,92}, from=9-18, to=9-17]
	\arrow[no head, from=10-2, to=11-2]
	\arrow[no head, from=10-3, to=9-4]
\end{tikzcd}
}
\end{center}

The reader may wish to compare further examples, especially the UASS for $S^3$ at $p = 3$ \cite[page 21]{UASScharts} with the ASS for $\mathbb{S}/3$: they are identical except for a single dot in low filtration.

Finally, the work of Konovalov \cite{AlgGSS} may be relevant here. If we understand the story correctly, the Goodwillie maps between the layers $S^n(k)$ can be lifted to maps between their ASS, and the resulting \textit{algebraic Goodwillie spectral sequence} collapses at the second page. Moreover, the formula for the first differential should be explicit. The point is that the $E_2 = E_\infty$ page of this algebraic spectral sequence should coincide with the $E_2$-page of the UASS. A truncated version should then provide the $E_2$-page of a possible ASS for $S^n_k$.

\subsection{Simpler approximations}

Although the core of this paper is a model for the Goodwillie tower of $S^1$, the computations suggest that, in many cases, this tower is not the essential part of the story. In fact, the computations become simpler once we quotient out the contribution of $S^1_k$ inside $S^n_k$. This is consistent with Behrens' computations of the Goodwillie spectral sequence, where most differentials are deduced from those for $S^1$.

We define a simpler approximation, denoted $\hat S^n_k$, as the fiber of the map $S^1_k \to \Omega^{n-1}S^n_k$. These approximations assemble into a filtration
\[
\Omega^nS^n \longrightarrow \cdots \longrightarrow \hat S^n_k \longrightarrow \hat S^n_{k-1} \longrightarrow \cdots \longrightarrow \hat S^n_0 \simeq *
\]
which converges on positive homotopy groups. The layers of this filtration are tractable. The first nontrivial layer is
\[
\hat S^n(1) \simeq \Sigma^\infty B\Sigma_p^{(p-1)(n-1)}
\]
and, as we saw in ex. \ref{todasn}, it already provides a good approximation to $\Omega^nS^n$ on homotopy groups for $0 < * < 4(p^2-2)$.

More generally, by Arone--Mahowald's computations (prop. \ref{cohomS}), for $n$ odd one has
\[
\mathbb{F}_p^*\hat S^n(k)
\cong
\Sigma^{1-2k}(\mathcal{A}/\mathcal{A}\beta)_{=k}^{\leq (n-1)/2}
:=
\Sigma^{1-2k}
(\mathcal{A}/\mathcal{A}\beta)_{=k}
/(\mathcal{A}/\mathcal{A}\beta)_{=k}^{\geq (n+1)/2}
\]
In particular, the map $\Omega^nS^n \to \hat S^n_k$ is an equivalence on $\pi_*$ for $0 < * < 2p^{k+1}-2k-3$.

It is natural to ask whether this process can be iterated to produce even simpler approximations. The map $S^1_k \to \Omega^{n-1}S^n_k$ has several desirable features that are difficult to replicate:

\begin{enumerate}
    \item The limit of $S^1_k$ is known, so $\Omega^nS^n$ can be reconstructed from the approximations $\hat S^n_k$.
    \item It kills the first nontrivial layer, i.e. it is an equivalence for $k=0$.
    \item It kills the rational part, while retaining higher chromatic information.
    \item It kills the first nontrivial homotopy group.
\end{enumerate}

Let $F_n$ denote the fiber of the map $I \to \Omega^n\Sigma^n$. Then the tower $\hat S^n_k$ is obtained by evaluating the Goodwillie tower of $F_{n-1}$ on $S^1$. To simplify this approximation further, one would like to find a functor $G$ and a map $G \to \Omega^?F_n\Sigma^?$ satisfying some or all of the properties above.

Although we do not yet have a convincing construction, we can point to two partial examples in the case $F_2 = W$.

First, by the connectivity of $W$, there is a map of orthogonal functors $S^{(V+1)p-3}/p \to W$. This induces a map
\[
P_{m/p}I(S^{2np-3}/p) \longrightarrow P_mW(S^{2n-1})
\]
which satisfies properties (4) and (2), but likely no others.

Second, by \cite[thm.~1.2]{L2W} there is a map $W \to \Omega^{2p}W\Sigma^2$. This induces a map $W_k \to \Omega^{2p}W_k\Sigma^2$, which satisfies property (2), but probably no others. Relatedly, we expect the existence of a sequence of orthogonal functors $W^{(k)}$, with $W^{(0)} = S^{(-)}$, and fiber sequences
\[
W^{(k+1)} \longrightarrow W^{(k)} \longrightarrow \Omega^{2p^k}W^{(k)}\Sigma^2
\]
such that the right-hand map is an equivalence on $\mathbb{D}_{p^k}$. The reader may think of this as a $p$-local analogue of Arone's integral construction \cite{iterates}.  

\subsection{Chromatic directions}

We have seen that the combination of the calculus version of the Gray sequence (thm. \ref{graysequence}) and the model for $S^1_k$ (thm. \ref{model}) implies that (prop. \ref{ext}), after sufficiently many loops, $S^n_k$ is a finite extension of
\[
\mathbb{Z},\quad \Omega^\infty\tau_{>0}\mathbb{S},\quad \dots,\quad \Omega^\infty\tau_{>0}\mathbb{Sp}^{p^k}
\]
In this paper we only use the computational consequences of this result, but it also seems to suggest a chromatic reinterpretation. Indeed, the spectra $\tau_{>0}\mathbb{Sp}^{p^k}$ are $K(i)$-acyclic for $i\leq k$ (prop. \ref{symmprodkn}). At present, however, we do not know how to formulate a precise statement. The subcategory generated under finite extensions by the first $k$ symmetric products always contains the sphere spectrum, and hence all finite spectra, so it is too large to capture this phenomenon. It may be more interesting to consider the complementary subcategory generated by
\[
\tau_{>0}\mathbb{Sp}^{p^k},\quad \tau_{>0}\mathbb{Sp}^{p^{k+1}},\quad \dots
\]
which lies entirely in the localizing subcategory of spectra that are $K(i)$-acyclic for $i\leq k$.

\subsection{Segal and Sullivan conjectures}

Another natural question is to understand the $C_p$-homotopy fixed points of the approximation spheres, namely
\[
(S^n_k)^{hC_p} \simeq \mathrm{Map}_*(BC_{p+},S^n_k)
\]
This provides a natural bridge between the \textit{Segal conjecture}, corresponding to $k=0$, and the \textit{Sullivan conjecture}, corresponding to $k=\infty$.

For $k=0$, the Segal conjecture gives
\begin{align*}
	\mathrm{Map}_*(BC_{p+},S^n_0)
	&\simeq \Omega^\infty \mathrm{Map}_*(\Sigma^\infty BC_{p+},\mathbb{S}^n) \\
	&\simeq \Omega^\infty \Sigma^n(\Sigma^{\infty}BC_{p+})^\vee \\
	&\simeq \Omega^\infty \Sigma^n(\Sigma^\infty BC_{p+} \oplus \mathbb{S}^\wedge_p) \\
	&\simeq \Omega^\infty\mathbb{S}^n \times \Omega^\infty\Sigma^n(\Sigma^\infty BC_p \oplus \mathbb{S}^\wedge_p) \\
	&\simeq S^n_0 \times \Omega^\infty\Sigma^n(\Sigma^\infty BC_p \oplus \mathbb{S}^\wedge_p)
\end{align*}
At the limit $k=\infty$, the Sullivan conjecture says that
\[
\mathrm{Map}_*(BC_{p+},S^n)
\simeq
\mathrm{Map}(BC_p,S^n)
\simeq
S^n
\]
The model for the approximation circles gives some information about the intermediate stages. Indeed, it reduces the computation of the homotopy fixed points of $S^1_k$ to the corresponding computation for symmetric product spectra:
\[
\Omega (S^1_k)^{hC_p}
\simeq
\mathbb{Z} \times \Omega^{\infty+2k}(\tau_{>0}\mathbb{Sp}^{p^k})^{hC_p}
\]
Thus we are left with the question of understanding the $C_p$-homotopy fixed points of the symmetric product spectra. It is possible that this computation is already known somewhere in the literature.

\subsection{Unitary Whitehead conjecture}

We expect that the ideas developed in this paper may also be relevant to a problem posed by Neisendorfer. The standard filtration
% https://q.uiver.app/#q=WzAsNyxbMCwwLCJVKDEpIl0sWzEsMCwiXFxjZG90cyJdLFsyLDAsIlUobikiXSxbMywwLCJVKG4rMSkiXSxbMywxLCJTXnsybisxfSJdLFs0LDAsIlxcY2RvdHMiXSxbNSwwLCJVKFxcaW5mdHkpIl0sWzAsMV0sWzEsMl0sWzIsM10sWzMsNF0sWzMsNV0sWzUsNl1d
\[
\begin{tikzcd}
	{U(1)} & \cdots & {U(n)} & {U(n+1)} & \cdots & {U(\infty)} \\
	&&& {S^{2n+1}}
	\arrow[from=1-1, to=1-2]
	\arrow[from=1-2, to=1-3]
	\arrow[from=1-3, to=1-4]
	\arrow[from=1-4, to=1-5]
	\arrow[from=1-4, to=2-4]
	\arrow[from=1-5, to=1-6]
\end{tikzcd}
\]
gives rise to a spectral sequence
\[
\begin{tikzcd}
	\pi_*\begin{array}{c}
        \begin{Bmatrix}
         U(1) \\
         S^{2n+1}
     \end{Bmatrix}_{n \geq 1}
     \end{array}
     & \pi_*U(\infty)
	\arrow[Rightarrow, from=1-1, to=1-2]
\end{tikzcd}
\]
Since the homotopy groups of $U(\infty)$ are known, there must be substantial cancellation in this spectral sequence. Neisendorfer's question is to understand what this spectral sequence looks like.

One can try to approximate this spectral sequence using unitary calculus \cite{unitary}. Indeed, both the functor $U$ and the functor $V \mapsto \Sigma S^V$ can be viewed as functors $\mathrm{Vect}_{\mathbb{C}} \to \mathcal{S}_*$. Assuming a comparison theorem between Goodwillie calculus and unitary calculus for analytic functors, analogous to \cite[thm. 3.5]{comparing}, one is led to expect a spectral sequence of the form
\[
\begin{tikzcd}
	\pi_*\begin{array}{c}
        \begin{Bmatrix}
         P_{p^k}U(\mathbb{C}) \\
         S_k^{2n+1}
     \end{Bmatrix}_{n \geq 1}
     \end{array}
     & \pi_*U(\infty)
	\arrow[Rightarrow, from=1-1, to=1-2]
\end{tikzcd}
\]
We now have a reasonable understanding of the approximation spheres, so the remaining input would be the unitary tower of $U(V)$ at $V=\mathbb{C}$. This is the content of the \textit{$bu$-Whitehead conjecture} \cite{AroneLesh2009}, which predicts that the unitary tower of $U(V)$ at $V=\mathbb{C}$ fits into an exact complex analogous to the one arising from the Goodwillie tower of $S^1$. If this analogy behaves as in the classical Whitehead conjecture, one would expect an equivalence of the form
\[
\Omega P_{p^k}U(\mathbb{C}) \simeq \mathbb{Z} \times \Omega^{\infty+2(k+1)}\tau_{>0}A_{p^k}
\]
Here the spectra $A_m$ are those appearing in \cite{AroneLesh2005,AroneLesh2009} and play a role analogous to that of the symmetric product spectra. One has
\[
A_0 = ku \qquad A_1/A_0 \simeq \Sigma \mathbb{CP}^\infty := \Sigma^{\infty+1}CP^\infty
\]
This suggests that $A_1$ itself may be the cofiber of the reduced Segal--Snaith map
\[
\mathbb{CP}^\infty \longrightarrow \mathbb{CP}^\infty_+ \longrightarrow ku
\]
although we do not know a reference for this identification. If so, then the first approximation to Neisendorfer's spectral sequence, corresponding to $k=0$, would be described by the Atiyah--Hirzebruch spectral sequence for $\mathbb{CP}^\infty$ on the positive lines (see, for instance, \cite{AHSSCP}), together with the homotopy groups of this cofiber on the first line. The differentials landing on the first line should then be induced by the Segal--Snaith map. This is completely analogous to the metastable spectral sequence described in subsection \ref{metastable}.

\appendix

\section{Stripping techniques}

The goal of this section is to find the first nontrivial kernel of the map $\mathcal{A} \to \mathcal{A}$ given by right multiplication by $P^{p^{k-1}}\cdots P^1\beta$ - where we keep using the convention that, at $p=2$, we write $\beta \leftrightarrow Sq^1$ and $P^i \leftrightarrow Sq^{2i}$. This is the main ingredient to compute the third nontrivial homotopy group of the symmetric products in lemma~\ref{homotopysp}. The solution will be the following.

\begin{enumerate}
    \item[\boxed{$p = 2$}] The first nontrivial kernel is the 1-dimensional vector space in degree $2^k$ generated by the element $N_k$ defined recursively by
    \begin{align*}
        N_0 &:= Sq^1 \\
        N_k &:= Sq^{2^k} + Sq^{2^{k-1}}N_{k-1}
    \end{align*}
    An example is given by $N_2 = Sq^4 + Sq^3Sq^1$.
    \item[\boxed{$p > 2$}] The first nontrivial kernel is the 1-dimensional vector space in degree $2p^k-1$ generated by the element $N_k := N_{k, k}$ defined recursively by
    \begin{align*}
        N_{k, 0} &:= P^{p^{k-1}}\cdots P^1\beta \\
		N_{k, h+1} &:= N_{k, h} - \frac{1}{2}(\xi_{k-h}^{p^h} \cap N_{k, h})(\tau_h \cap N_{k, 0})
    \end{align*}
    An example is given by $N_2 = P^pP^1\beta - \frac{1}{2}\beta P^pP^1 - \frac{1}{2}P^{p+1}\beta + \frac{1}{4}\beta P^{p+1}$.
\end{enumerate}

We have not yet defined what the notation $\xi \cap -$ means. 
This is known as the \textit{stripping technique} and is the main tool to solve this problem.

Before introducing it, let us briefly recall the structure of the \emph{dual Steenrod algebra} $\mathcal{A}_*$. 
This is the Hopf algebra dual to $\mathcal{A}^*$ and, as an algebra, it is the free graded commutative algebra generated by the elements $\xi_k$ in degrees $2^k - 1$ for $p = 2$ or by the elements $\xi_k, \tau_k$ in degrees $2(p^k - 1), 2p^k - 1$ for $p > 2$. The comultiplication is given by
\[
\Delta \xi_k = \textstyle \sum_{i=0}^k \xi_{k-i}^{p^i} \otimes \xi_i
\qquad 
\Delta \tau_k = \tau_k \otimes 1 + \textstyle \sum_{i=0}^k \xi_{k-i}^{p^i} \otimes \tau_i
\]
The comultiplication $\Delta$ on $\mathcal{A}^*$, together with the dual pairing $\langle, \rangle$, induces a natural left action of $\mathcal{A}_*$ on $\mathcal{A}^*$, defined as follows.

\begin{definition}
	The \textit{stripping pairing} is defined as
	\begin{align*}
    - \cap -: &\mathcal{A}_* \otimes \mathcal{A}^* \xrightarrow{1 \otimes \Delta} \mathcal{A}_* \otimes \mathcal{A}^* \otimes \mathcal{A}^* \xrightarrow{\langle,\rangle\otimes 1} \mathbb{F}_p \otimes \mathcal{A}^* \cong \mathcal{A}^* \\
    &\xi \otimes \theta \mapsto\xi \cap \theta := \textstyle \sum \xi(\theta_1)\theta_2
\end{align*}
\end{definition}

Here the notation $\theta_1, \theta_2$ follows the usual Sweedler convention. Notice that this definition applies to any bialgebra.

The definition may seem complicated, but for $\xi = \xi_k^{p^j}$ it is easy to describe what stripping does. Indeed, we have $$\xi_k^{p^j} \cap \theta^I = \textstyle \sum \theta^{I - (p^{j+k-1}, \dots, p^j)}$$

where the sum is taken over all possible ways of matching the vector $(p^{j+k-1}, \dots, p^j)$ to the length of $I$ by inserting zeros. For example, at $p = 2$, if $I$ has length $3$ and we want to strip by $(4, 2)$, the sum (and the differences) is taken over $(4, 2, 0)$, $(4, 0, 2)$, $(0, 4, 2)$. Note also that $I$ does not need to be an admissible sequence.

As an example of how this works, consider the relation $Sq^3Sq^2 = 0$.
\begin{enumerate}
    \item Stripping it by $(2, 1)$ gives $Sq^1Sq^1 = 0$.
    \item Stripping it by $(2)$ gives $Sq^1Sq^2 + Sq^3 = 0$, i.e. $Sq^1Sq^2 = Sq^3$.
    \item Stripping it by $(1)$ gives $Sq^2Sq^2 + Sq^3Sq^1 = 0$, i.e. $Sq^2Sq^2 = Sq^3Sq^1$.
\end{enumerate}

Note that $\xi_1 \cap -$, i.e. stripping by $(1)$, corresponds to the classical statement that the assignment $Sq^n \mapsto Sq^{n-1}$ acts as a derivation in the Steenrod algebra. Stripping techniques are just a generalization - or a more efficient version - of this idea. We learnt about stripping techniques in \cite[Sec.~5]{Woodproblems}.

Even if this is the intuition behind it, and the reader should keep it in mind while following the proofs, we will try to work as formally as possible. We summarize in the following lemma all the properties of stripping that we will need, providing complete proofs since the literature at odd primes often contains sign errors.

\begin{lem}
	The stripping pairing satisfies the following properties:
	\begin{enumerate}
\item $|\xi \cap \theta| = |\theta| - |\xi|$. In particular, $\xi \cap \theta = 0$ if $|\theta| < |\xi|$.
\item $\xi \cap \theta = \xi(\theta)$ if $|\xi| = |\theta|$.
\item $\xi \cap (\theta \psi) = \textstyle \sum (-1)^{|\xi_2||\theta|} (\xi_1 \cap \theta)(\xi_2 \cap \psi)$.
\item $\eta \cap (\xi \cap \theta) = \eta \xi \cap \theta$.
\item $\xi_0 \cap - = \id$.
\item $\xi_1^{p^k} \cap -$ is a derivation.
\item $\xi_h^j \cap Sq^i = Sq^{i-j}$ if $h = 1$, and $0$ if $h > 1$.
\item $\xi_h^j \cap P^i = P^{i-j}$ if $h = 1$, and $0$ if $h > 1$.
\item $\xi_h \cap \theta = 0$ for $h > \mathrm{len}\,\theta$.
\item $\xi_{h}^{p^{k-h}} \cap (\theta \psi) = \theta \cdot (\xi_{h}^{p^{k-h}} \cap \psi)$ if $|\theta| < p^{k - 1}|\xi_1|$.
\end{enumerate}
\end{lem}

\begin{proof}
	We prove each statement in turn.
	\begin{enumerate}
		\item Only the term $\xi(\theta_1)\theta_2$ with $|\theta_1| = |\xi|$ contributes to the sum, and this term has degree $|\theta_2| = |\theta| - |\xi|$.
		
		\item We have $\Delta \theta = 1 \otimes \theta + \theta \otimes 1 + \textstyle \sum \theta_i \otimes \psi_i$, where each $|\theta_i| < |\theta|$. Since $|\xi| = |\theta|$, only the middle term survives, hence $\xi \cap \theta = \xi(\theta)1 = \xi(\theta)$.

		\item This follows from two diagrams. First, since $\mathcal{A}^*$ is a bialgebra,
	% https://q.uiver.app/#q=WzAsNSxbMCwwLCJcXG1hdGhjYWx7QX1eKiBcXG90aW1lcyBcXG1hdGhjYWx7QX1eKiJdLFsxLDAsIlxcbWF0aGNhbHtBfV4qIl0sWzIsMCwiXFxtYXRoY2Fse0F9XiogXFxvdGltZXMgXFxtYXRoY2Fse0F9XioiXSxbMCwxLCJcXG1hdGhjYWx7QX1eKiBcXG90aW1lcyBcXG1hdGhjYWx7QX1eKiBcXG90aW1lcyBcXG1hdGhjYWx7QX1eKiJdLFsyLDEsIlxcbWF0aGNhbHtBfV4qIFxcb3RpbWVzIFxcbWF0aGNhbHtBfV4qIFxcb3RpbWVzIFxcbWF0aGNhbHtBfV4qIFxcb3RpbWVzIFxcbWF0aGNhbHtBfV4qIl0sWzAsMSwiXFxuYWJsYSJdLFsxLDIsIlxcRGVsdGEiXSxbMCwzLCJcXERlbHRhIFxcb3RpbWVzIFxcRGVsdGEiLDJdLFszLDQsIjEgXFxvdGltZXMgXFx0YXUgXFxvdGltZXMgMSIsMl0sWzQsMiwiIFxcbmFibGEgXFxvdGltZXMgXFxuYWJsYSIsMl1d
	\[\begin{tikzcd}
	{\mathcal{A}^* \otimes \mathcal{A}^*} & {\mathcal{A}^*} & {\mathcal{A}^* \otimes \mathcal{A}^*} \\
	{\mathcal{A}^* \otimes \mathcal{A}^* \otimes \mathcal{A}^*} && {\mathcal{A}^* \otimes \mathcal{A}^* \otimes \mathcal{A}^* \otimes \mathcal{A}^*}
	\arrow["\nabla", from=1-1, to=1-2]
	\arrow["{\Delta \otimes \Delta}"', from=1-1, to=2-1]
	\arrow["\Delta", from=1-2, to=1-3]
	\arrow["{1 \otimes \tau \otimes 1}"', from=2-1, to=2-3]
	\arrow["{ \nabla \otimes \nabla}"', from=2-3, to=1-3]
	\end{tikzcd}\]
	which translates to $$\Delta(\theta\psi) = \textstyle \sum (-1)^{|\theta_2||\psi_1|}\theta_1\psi_1 \otimes \theta_2\psi_2$$
	Second, since $\mathcal{A}_*$ and $\mathcal{A}^*$ are dual bialgebras,
	% https://q.uiver.app/#q=WzAsNSxbMSwwLCJcXG1hdGhjYWx7QX1fKiBcXG90aW1lcyBcXG1hdGhjYWx7QX1eKiJdLFsyLDAsIlxcbWF0aGJie0Z9X3AiXSxbMCwwLCJcXG1hdGhjYWx7QX1fKiBcXG90aW1lcyBcXG1hdGhjYWx7QX1eKiBcXG90aW1lcyBcXG1hdGhjYWx7QX1eKiJdLFswLDEsIlxcbWF0aGNhbHtBfV8qIFxcb3RpbWVzIFxcbWF0aGNhbHtBfV8qIFxcb3RpbWVzXFxtYXRoY2Fse0F9XiogXFxvdGltZXMgXFxtYXRoY2Fse0F9XioiXSxbMiwxLCJcXG1hdGhjYWx7QX1fKiBcXG90aW1lc1xcbWF0aGNhbHtBfV4qIFxcb3RpbWVzIFxcbWF0aGNhbHtBfV8qIFxcb3RpbWVzIFxcbWF0aGNhbHtBfV4qIl0sWzAsMSwiXFxsYW5nbGUsXFxyYW5nbGUiXSxbMiwwLCIxIFxcb3RpbWVzIFxcbmFibGEiXSxbMiwzLCJcXERlbHRhIFxcb3RpbWVzIDEgXFxvdGltZXMgMSIsMl0sWzMsNCwiMSBcXG90aW1lcyBcXHRhdSBcXG90aW1lcyAxIiwyXSxbNCwxLCJcXGxhbmdsZSxcXHJhbmdsZSBcXG90aW1lcyBcXGxhbmdsZSxcXHJhbmdsZSIsMl1d
\[\begin{tikzcd}
	{\mathcal{A}_* \otimes \mathcal{A}^* \otimes \mathcal{A}^*} & {\mathcal{A}_* \otimes \mathcal{A}^*} & {\mathbb{F}_p} \\
	{\mathcal{A}_* \otimes \mathcal{A}_* \otimes\mathcal{A}^* \otimes \mathcal{A}^*} && {\mathcal{A}_* \otimes\mathcal{A}^* \otimes \mathcal{A}_* \otimes \mathcal{A}^*}
	\arrow["{1 \otimes \nabla}", from=1-1, to=1-2]
	\arrow["{\Delta \otimes 1 \otimes 1}"', from=1-1, to=2-1]
	\arrow["{\langle,\rangle}", from=1-2, to=1-3]
	\arrow["{1 \otimes \tau \otimes 1}"', from=2-1, to=2-3]
	\arrow["{\langle,\rangle \otimes \langle,\rangle}"', from=2-3, to=1-3]
\end{tikzcd}\]
	which yields $$\xi(\theta\psi) = \textstyle \sum (-1)^{|\xi_2||\theta|}\xi_1(\theta)\xi_2(\psi)$$
	Combining the two gives:
	\begin{align*}
		\xi \cap (\theta\psi) &= \textstyle \sum \xi((\theta\psi)_1)(\theta\psi)_2 \\
		&= \textstyle \sum (-1)^{|\theta_2||\psi_1|}\xi(\theta_1\psi_1) \theta_2\psi_2 \\
		&= \textstyle \sum (-1)^{|\theta_2||\psi_1| + |\xi_2||\theta_1|}\xi_1(\theta_1)\xi_2(\psi_1) \theta_2\psi_2 \\
		&= \textstyle \sum (-1)^{|\theta_2||\xi_2| + |\xi_2||\theta_1|}\xi_1(\theta_1)\xi_2(\psi_1) \theta_2\psi_2 \\
		&= \textstyle \sum (-1)^{|\xi_2||\theta|}\xi_1(\theta_1)\xi_2(\psi_1) \theta_2\psi_2 \\
		&= \textstyle \sum (-1)^{|\xi_2||\theta|}(\xi_1 \cap \theta)(\xi_2 \cap \psi)
	\end{align*}
	where we used that $\xi_2(\psi_1) \neq 0$ only if $|\xi_2| = |\psi_1|$.

	\item Again, this follows from the dual bialgebra structure. We compute: $$\eta\xi \cap \theta = \textstyle \sum (\eta\xi)(\theta_1)\theta_2 = \textstyle \sum (-1)^{|\xi||(\theta_1)_1|}\eta((\theta_1)_1)\xi((\theta_1)_2)\theta_2$$
	By coassociativity, $(\theta_1)_1 = \theta_1$, $(\theta_1)_2 = (\theta_2)_1$ and $(\theta_2)_2 = \theta_2$, hence $$\eta\xi \cap \theta = \textstyle \sum (-1)^{|\xi||\theta_1|}\eta(\theta_1)\xi((\theta_2)_1)(\theta_2)_2 = \textstyle \sum (-1)^{|\xi||\theta_1|}\eta(\theta_1)(\xi \cap \theta_2)$$
	Nontrivial terms require $|\theta_1| = |\eta|$, so $$\eta\xi \cap \theta = \textstyle \sum (-1)^{|\xi||\eta|}(\xi \cap (\eta \cap \theta))$$
	and the result follows by graded commutativity.

	\item Same argument as in (2).
	
	\item Follows from (3) and the fact that the elements $\xi_1^{p^k}$ are primitive.
	
	\item By (4) it suffices to consider $j = 1$. By definition,
\[
\xi_h \cap Sq^i = \textstyle \sum \xi_h(Sq^\alpha) Sq^{i-\alpha}
\]
Now $\xi_h(Sq^\alpha) \neq 0$ if and only if $h = 1$ and $\alpha = 1$, which gives the claim.

	\item Same argument as in (7).
	\item Proceed by induction. Suppose $\theta = P^i \psi$. By (3),
\[
\xi_h \cap \theta = \textstyle \sum (\xi_{h - \alpha}^{p^\alpha} \cap P^i)(\xi_\alpha \cap \psi)
\]
By (8), the left-hand side vanishes for $h - \alpha > 1$, while by the induction hypothesis the second factor vanishes for $\alpha>\mathrm{len}\,\psi$. Since $\mathrm{len}\,\psi<h-1$, every index $\alpha$ satisfies either $\alpha<h-1$ or $\alpha>\mathrm{len}\,\psi$, and hence every summand vanishes.

	\item By (3),
\[
\xi_h^{p^{k-h}} \cap (\theta\psi) = \textstyle \sum (\xi_{h-i}^{p^{k-h+i}} \cap \theta)(\xi_i^{p^{k-h}} \cap \psi)
\]
By (1),
\[
|\xi_{h-i}^{p^{k-h+i}} \cap \theta| = |\theta| - |\xi_{h-i}^{p^{k-h+i}}| = |\theta| - p^{k-h+i}|\xi_{h-i}|
\]
Since $|\xi_j| \ge p^{j-1}|\xi_1|$ for all $j \ge 1$, we obtain $|\xi_{h-i}^{p^{k-h+i}} \cap \theta| < 0$ for $i < h$. Thus only the term $i = h$ survives, giving the desired identity. \qedhere
	\end{enumerate}
\end{proof}

With all the stripping techniques settled, we can now move on to solving the problem. First of all, note that right multiplication by $P^{p^{k-1}}\cdots P^1\beta$ factors through the map $\mathcal{A}^* \to (\mathcal{A}_{\geq k}\beta)^{* + 2p^k - 1}$. We begin by comparing the dimensions of the domain and codomain as $\mathbb{F}_p$-vector spaces.

\begin{lem}
If $* < 2p^k(p-1)$ then
\[
\dim_{\mathbb{F}_p} (\mathcal{A}_{\geq k}\beta)^{* + 2p^k - 1}
= \dim_{\mathbb{F}_p}\mathcal{A}^* - \dim_{\mathbb{F}_p}(\mathcal{A}_{=k}\beta)^*.
\]
In particular, for $* < 2p^k - 1$ we have
\[
\dim_{\mathbb{F}_p} (\mathcal{A}_{\geq k}\beta)^{* + 2p^k - 1}
= \dim_{\mathbb{F}_p}\mathcal{A}^*
\]
while for $* = 2p^k - 1$ the inequality
\[
\dim_{\mathbb{F}_p} (\mathcal{A}_{\geq k}\beta)^{* + 2p^k - 1}
< \dim_{\mathbb{F}_p}\mathcal{A}^*
\]
holds.
\end{lem}

\begin{proof}
We first note that
\[
(\mathcal{A}_{\geq k}\beta)^{* + 2p^k - 1}
\cong (\mathcal{A}/\mathcal{A}\beta)_{\geq k}^{* + 2(p^k - 1)}
\]
Under the assumption $* < 2p^k(p-1)$, the degree bound $* + 2p^k - 2 < 2(p^{k+1} - 1)$ implies that only monomials of \emph{length $k$} contribute in that degree, hence
\[
(\mathcal{A}/\mathcal{A}\beta)_{\geq k}^{* + 2(p^k - 1)}
= (\mathcal{A}/\mathcal{A}\beta)_{= k}^{* + 2(p^k - 1)}
\]
A counting argument using the Serre--Cartan basis then gives
\[
\dim_{\mathbb{F}_p} (\mathcal{A}/\mathcal{A}\beta)_{= k}^{* + 2(p^k - 1)}
= \dim_{\mathbb{F}_p}\mathcal{A}_{\leq k}^*
- \dim_{\mathbb{F}_p}(\mathcal{A}_{=k}\beta)^*
\]
The bijection comes from the index shift $i_j \mapsto i_j - p^{k-j}$, which preserves admissibility and decreases total degree by $2(p^k - 1)$. Since $\mathcal{A}_{\leq k}^* = \mathcal{A}^*$ in our range, the desired equality follows.

Finally, note that $(\mathcal{A}_{=k}\beta)^*$ first becomes nonzero in degree $2p^k - 1$ generated by the element $P^{p^{k-1}}\cdots P^1\beta$, giving the last statement.
\end{proof}

Thus, we already have an upper bound for injectivity. For $p > 2$, we will see that this is in fact the lower bound as well, whereas for $p = 2$ the situation is different.

\begin{ex}
	To help the reader become familiar with the proof below, we illustrate the case $k = 1$. Take $\theta$ such that $\theta Sq^2 Sq^1 = 0$. Stripping by $\xi_2 \leftrightarrow (2, 1)$ gives
	\[
	\theta + \theta_{(2)} (Sq^2 + Sq^1 Sq^1) + \theta_{(2,1)} Sq^2 Sq^1 = 0
	\] 
	If $|\theta| \leq 2$, then $\theta_{(2,1)} = 0$, so $\theta = \theta_{(2)} Sq^2$, which, as one can verify, also lies in the kernel of $Sq^2 Sq^1$.
\end{ex}

\begin{lem}
	\label{mult2}
	The map $\mathcal{A}^* \to (\mathcal{A}_{\geq k} Sq^1)^{* + 2^{k+1} - 1}$ given by right multiplication by $M_k := Sq^{2^k} \cdots Sq^1$ is an isomorphism in degrees $< 2^k$, and it has a one-dimensional kernel in degree $2^k$ generated by the element $N_k$, defined recursively as
	\begin{align*}
        N_0 &:= Sq^1 \\
        N_k &:= Sq^{2^k} + Sq^{2^{k-1}}N_{k-1}
    \end{align*}
\end{lem}

\begin{proof}  The strategy of the proof is as follows. The lemma above tells us that for $k > 0$ and $* \leq 2^k < 2^{k+1} - 1$ this is a map between $\mathbb{F}_2$-vector spaces of the same dimension. Hence, for $* < 2^k$, it suffices to prove that the map is injective. For $* = 2^k$, we will prove that the kernel is contained in the one-dimensional subspace generated by the element $N_k := \xi_k \cap M_k$, giving an upper bound. We cannot prove directly that $N_k$ lies in the kernel, so instead we show that the map is not surjective in that degree, giving a lower bound for the dimension of the kernel. Lastly, we prove the recursive formula for $N_k$. 
	\begin{enumerate}
    \item[Step 1.] Take $\theta$ such that $\theta M_k = 0$. Stripping this equation by $\xi_{k+1}$ gives  $$\textstyle \sum (\xi_{k+1-i}^{2^i} \cap \theta)(\xi_i \cap M_k) = 0$$
    If $|\theta| \leq 2^k$ for degree reason the sum reduces to 
	$$0 = (\xi_0 \cap \theta)(\xi_{k+1} \cap M_{k}) + (\xi_1^{2^k} \cap \theta)(\xi_k \cap M_k) = \theta + (\xi_1^{2^k} \cap \theta)(\xi_k \cap M_k)$$
	Here we used that $\xi_0 \cap -$ is the identity and that $\xi_{k+1}$ is the dual of $M_k$. Now note that $\xi_1^{2^k} \cap \theta$ is a scalar, and it is zero if $|\theta| < 2^k$, implying $\theta = 0$ and hence the map is injective. For $|\theta| = 2^k$, we conclude that $\ker(\cdot M_k) \leq \langle N_k\rangle$.
    \item[Step 2.] It is not immediately clear that $N_k$ lies in the kernel.  For $p>2$, the analogous statement fails, since $N_1 = P^1$ but $P^1 P^1 \beta = 2P^2 \beta \neq 0$. However, for $p=2$ we do have $P^1 P^1 \beta = Sq^2 Sq^2 Sq^1 = Sq^3 Sq^1 Sq^1 = 0$. Dealing directly with Adem relations is complicated, so we use an indirect argument.
	Since the map is between vector spaces of the same finite dimension, it suffices to show that it is not surjective in degree $2^k$.  Indeed, this will imply that the kernel is at least one-dimensional, and since we already know it cannot be larger, it must be exactly $\langle N_k \rangle$. Suppose for contradiction that there exists $\theta$ of degree $2^k$ such that 
	\[
	\theta M_k = Sq^{2^{k+1}} M_{k-1}
	\]
	Stripping by $\xi_1^{2^{k+1}}$ gives $0 = M_{k-1}$ - a contradiction. This follows because $\xi_1^{2^{k+1}} \cap -$ is a derivation and annihilates $\theta$ and every $Sq^i$ with $i < 2^{k+1}$ for degree reason. Hence, the map is not surjective in degree $2^k$, and the kernel is precisely $\langle N_k \rangle$.
    \item[Step 3.] It remains to determine the recursive formula for $N_k$.  By definition,
	\[
	N_k = \xi_k \cap M_k = \xi_k \cap (Sq^{2^k} M_{k-1})
	= \textstyle \sum (\xi_{k-i}^{2^i} \cap Sq^{2^k})(\xi_i \cap M_{k-1})
	\]
    Since monomials of length one are killed by $\xi_i$ for $i>1$, and $\xi_1^j \cap Sq^i = Sq^{i-j}$, the sum reduces to
    \[
    N_k = (\xi_0 \cap Sq^{2^k})(\xi_k \cap M_{k-1})
      + (\xi_1^{2^{k-1}} \cap Sq^{2^k})(\xi_{k-1} \cap M_{k-1})
      = Sq^{2^k} + Sq^{2^{k-1}} N_{k-1}
    \]
	The base case $N_0 = \xi_0 \cap Sq^1 = Sq^1$ is immediate, completing the induction. \qedhere

\end{enumerate}
\end{proof}

\begin{ex}
	The first few elements $N_k$ are given by
	\begin{align*}
		N_0 &= Sq^1 \\
		N_1 &= Sq^2 + Sq^1Sq^1 = Sq^2 \\
		N_2 &= Sq^4 + Sq^2Sq^2 = Sq^4 + Sq^3Sq^1 \\
		N_3 &= Sq^8 + Sq^4(Sq^4 + Sq^3Sq^1) = Sq^8 + Sq^6Sq^2 + Sq^7Sq^1 + Sq^5Sq^2Sq^1
	\end{align*}
\end{ex}

\begin{rem}
We can also give an alternative proof using the conjugation $\chi$ in the Steenrod algebra. Indeed, our element $N_k$ coincides with the conjugate $\chi(Sq^{2^k})$, since both satisfy the same recursive relation \cite{conjugate}. It follows that the conjugate of $M_k$ is $Sq^{2^{k+1}-1}$. To see this, note that from $M_k = Sq^{2^k} M_{k-1}$ we obtain
\[
\chi(M_k) = \chi(M_{k-1}) \chi(Sq^{2^k}) = Sq^{2^k-1} N_k = Sq^{2^k-1} Sq^{2^k} + Sq^{2^k-1} Sq^{2^{k-1}} N_{k-1} = Sq^{2^{k+1}-1}
\]
where the last equality follows from the Adem relations. With these facts in hand, the proof proceeds as follows. Suppose $\theta M_k = 0$. Passing to the conjugate yields $Sq^{2^{k+1}-1} \chi(\theta) = 0$. This cannot occur if $|\chi(\theta)| < 2^k$, by the properties of the Cartan--Serre basis. The first nontrivial case is $\chi(\theta) = Sq^{2^k}$, which gives $\theta = \chi(Sq^{2^k}) = N_k$. 

We do not currently know how to generalize this argument to the case $p>2$, so for that situation we continue to rely on the stripping technique.
\end{rem}

\begin{ex}
	To help the reader become familiar with the proof below, we illustrate the case $k = 1$. Take $\theta$ such that $\theta P^1\beta = 0$. Stripping by $\tau_1 \leftrightarrow (1, \beta)$ gives
	\[
	\theta + \theta_{(1)} P^1 \pm \theta_{(1, \beta)} P^1 \beta = 0
	\]
	If $|\theta| \leq 2p - 1$, stripping further by $\xi_1 \leftrightarrow (1)$ gives
	\[
	\theta_{(1)} + \theta_{(1)} \pm \theta_{(1, \beta)} \beta = 0
	\]
	Hence we conclude
	\[
	\theta = \pm \theta_{(1, \beta)} \Big(\frac{1}{2} \beta P^1 - P^1 \beta \Big)
	\]
	which, as one can verify, lies in the kernel.
\end{ex}

\begin{lem}
	\label{multodd}
	The map $\mathcal{A}^* \to (\mathcal{A}_{\geq k} \beta)^{* + 2p^k - 1}$ given by right multiplication by $M_k := P^{p^{k-1}} \cdots P^1 \beta$ is an isomorphism in degrees $< 2p^k - 1$, and it has a one-dimensional kernel in degree $2p^k - 1$ generated by the element $N_k := N_{k,k}$, defined recursively as
	\begin{align*}
		N_{k,0} &:= P^{p^{k-1}} \cdots P^1 \beta \\
		N_{k,h+1} &:= N_{k,h} - \frac{1}{2} (\xi_{k-h}^{p^h} \cap N_{k,h}) (\tau_h \cap M_k)
	\end{align*}
\end{lem}

\begin{proof}
	The strategy of the proof is as follows. By the dimension lemma, for $* < 2p^k - 1$ the domain and codomain have the same dimension, so it suffices to show that the map is injective. For $* = 2p^k - 1$, the domain has strictly larger dimension than the codomain, so it suffices to show that the kernel is contained in the one-dimensional subspace generated by $N_k$.

	Take $\theta$ such that $\theta M_k = 0$. Stripping this equation by $\tau_k$ yields $$(\tau_k \cap \theta)M_k + (-1)^{|\theta|}\textstyle \sum_{i = 0}^{k}(\xi_{k-i}^{p^i} \cap \theta)(\tau_i \cap M_k) = 0$$
	If $|\theta| \le 2p^k - 1$, then $\tau_k \cap \theta$ is a scalar. Let $\lambda := -(-1)^{|\theta|} (\tau_k \cap \theta)$. Then $$\theta = \lambda M_k - \textstyle \sum_{i = 0}^{k-1}(\xi_{k-i}^{p^i} \cap \theta)(\tau_i \cap M_k) =  \lambda N_{k, 0} - \textstyle \sum_{i = 0}^{k-1}(\xi_{k-i}^{p^i} \cap \theta)(\tau_i \cap M_k)$$
	We claim that for every $h \le k$, we have
	$$\theta = \lambda N_{k, h} - \textstyle \sum_{i = h}^{k-1}(\xi_{k-i}^{p^i} \cap \theta)(\tau_i \cap M_k)$$
	This is true for $h=0$, so it remains to prove the inductive step. For every $h \le i \le k-1$, we have $|\xi_{k-i}^{p^i} \cap \theta| < 2p^i \leq 2p^{k-1} \leq 2p^{k-1}(p-1)$, hence $\xi_{k-i}^{p^i} \cap \theta$ is linear for $\xi_{k-h}^{p^h}$ (property (10)). It follows that stripping by $\xi_{k-h}^{p^h}$ yields
	$$\xi_{k-h}^{p^h} \cap \theta = \lambda(\xi_{k-h}^{p^h} \cap N_{k, h}) - \textstyle \sum_{i = h}^{k-1}(\xi_{k-i}^{p^i} \cap \theta)(\xi_{k-h}^{p^h}\tau_i \cap M_k)$$
	Now note that
	$$\xi_{k-h}^{p^h}\tau_i \cap M_k = \tau_i\xi_{k-h}^{p^h} \cap M_k = \tau_i \cap M_h = \delta_{i,h}$$
	where we used that $\xi_{k-h}^{p^h} \cap M_k = M_h$ (left to the reader) and that $\tau_i \cap M_h = 0$ for $i>h$ by length considerations. We conclude that $2\xi_{k-h}^{p^h} \cap \theta = \lambda(\xi_{k-h}^{p^h} \cap N_{k, h})$ and hence, since $p$ is odd,
	\begin{align*}
		\theta 
		&= \lambda \bigg(N_{k, h} -  \frac{1}{2}(\xi_{k-h}^{p^h} \cap N_{k, h})(\tau_h \cap M_k)\bigg) - \textstyle \sum_{i = {h+1}}^{k-1}(\xi_{k-i}^{p^i} \cap \theta)(\tau_i \cap M_k) \\
		&= \lambda N_{k, h+1} - \textstyle \sum_{i = {h+1}}^{k-1}(\xi_{k-i}^{p^i} \cap \theta)(\tau_i \cap M_k) 
	\end{align*}
	Iterating this, we obtain $\theta = \lambda N_{k,k}$. If $|\theta| < 2p^k - 1$, then $\lambda = 0$ and the map is injective. If $|\theta| = 2p^k - 1$, the kernel is contained in the one-dimensional $\mathbb{F}_p$-vector space generated by $N_{k,k}$. Since the map cannot be injective for dimension reasons, we conclude that $\ker(\cdot M_k) = \langle N_{k,k} \rangle$.
\end{proof}

\begin{ex}
	The first few elements $N_k$ are given by
	\begin{align*}
		N_0 &= \beta \\
		N_1 &= P^1\beta - \frac{1}{2}\beta P^1 \\
		N_2 &= P^pP^1\beta - \frac{1}{2}\beta P^pP^1 - \frac{1}{2}P^{p+1}\beta + \frac{1}{4}\beta P^{p+1}
	\end{align*}
\end{ex}

We expect that a simpler proof and a more transparent recursive formula exist. In particular, computations in low degrees suggest that $N_k$ may satisfy the simpler relation

\begin{ques}
	Does $N_k = P^{p^{k-1}} N_{k-1} - \frac{1}{2} N_{k-1} P^{p^{k-1}}$?
\end{ques}

\section{ASS generator for $W^{2n-1}(2)$}
\label{ASSW2}

The following Python code uses the calculation of the cohomology of $W^{2n-1}(2)$ from example \ref{cohomW2} to produce the content of a \texttt{.json} file that can be uploaded to Chatham's \href{https://spectralsequences.github.io/sseq/}{Adams spectral sequence calculator}. 

\begin{lstlisting}[language=Python]
from math import *
            
def W2(p, n, N):

    # return the .json file for the ASS of W^{2n-1}(2) at the prime p in a range increasing with N

    # make a list of generators
    gens = []
    for i in range(n*p, (n+N)*p):
        d = 2*n*p - 5 + 2*(p-1)*i
        gens.append(f"\"p{i}\":{d}")
        gens.append(f"\"q{i}\":{d + 1}")
        if i > n*p:
            gens.append(f"\"r{i}\":{d + 1}")
            gens.append(f"\"s{i}\":{d + 2}")
    g = ','.join(gens)

    # make a list of the Steenrod actions
    actions = []

    # Bocksteins
    for i in range(n*p, (n+N)*p):
        actions.append(f"b p{i} = q{i}")
        if i > n*p:
            actions.append(f"b r{i} = s{i}")

    # Reduced powers
    for e in range(int(log(N-1, p))+2):
        a = p**e
        for i in range(n*p, (n + N)*p - a):
            c1 = ((-1)**a*comb((p-1)*i - 1, a)) % p
            c2 = ((-1)**(a-1)*comb((p-1)*i - 1, a-1)) % p
            d = (c1 - c2) % p
            if c1:
                actions.append(f"P{a} p{i} = {c1} p{i + a}")
            if c2 and not d:
                actions.append(f"P{a} q{i} = {c2} r{i + a}")
            if not c2 and d:
                actions.append(f"P{a} q{i} = {d} q{i + a}")
            if c2 and d:
                actions.append(f"P{a} q{i} = {d} q{i + a} + {c2} r{i + a}")
            if i > n*p:
                if c1:
                    actions.append(f"P{a} r{i} = {c1} r{i + a}")
                if d:
                    actions.append(f"P{a} s{i} = {d} s{i + a}")

    a = ('\"' + '\",\"'.join(actions) + '\"').replace(' 1 ', ' ')

    return f"""{{"p": {p}, "algebra": ["adem"], "type": "finite dimensional module","gens": {{{g}}},"actions": [{a}]}}"""
\end{lstlisting}

\sloppy
\bibliographystyle{alpha}
\bibliography{bib}

\end{document}